\documentclass[letterpaper, 11pt]{amsart}
\usepackage{amssymb}
\usepackage{amsmath}
\usepackage{amsfonts}
\usepackage{amscd}
\usepackage{amsthm}
\usepackage{color}
\usepackage{soul}
\usepackage{enumerate}
\usepackage{tikz}
\usetikzlibrary{matrix,arrows,calc,decorations.pathreplacing,decorations.pathmorphing,fadings}

\copyrightinfo{2015}{Dingyu Yang}

\newtheorem{theorem}{Theorem}[section]
\newtheorem{lemma}[theorem]{Lemma}
\newtheorem{proposition}[theorem]{Proposition}
\newtheorem{corollary}[theorem]{Corollary}
\newtheorem{construction}[theorem]{Construction}

\theoremstyle{definition}
\newtheorem{definition}[theorem]{Definition}
\newtheorem{example}[theorem]{Example}

\theoremstyle{remark}
\newtheorem{remark}[theorem]{Remark}

\numberwithin{theorem}{section}

\newcommand{\R}{\mathbb{R}}

\usepackage{hyperref}

\begin{document}

\title[Virtual harmony]{Virtual harmony}

\author[Dingyu Yang]{Dingyu Yang}
%\address{Courant Institute of Mathematical Sciences, New York University, 251 Mercer Street, New York, N.Y. 10012-1185, United States}
\curraddr{CNRS, UMR 7586, Institut de Math\'{e}matiques de Jussieu-Paris Rive Gauche, Universit\'{e} Pierre et Marie Curie, Case 247, Bureau 15-25/501, 4 place Jussieu, F-75005, Paris, France}
\email{dingyu\_yang@cantab.net, dingyu.yang@imj-prg.fr}

\thanks{The author is currently supported by the ERC via the Starting Grant ERC StG-259118 managed by the CNRS}

%\subjclass[2010]{Primary 53D99, 58D27; Secondary 32Q65, 53D42}

\date{October 23, 2015}

%\dedicatory{}

%\keywords{Symplectic topology, moduli space, Kuranishi structure, polyfold Fredholm theory, abstract perturbation}

\begin{abstract}\addcontentsline{toc}{section}{Abstract}
This article serves a few purposes. First of all, it reviews \cite{DI} and previews and samples some results from four papers \cite{MAXCAT}, \cite{DII}, \cite{DIII} and \cite{Yang} I have been preparing. It is also a written-up and expanded version of a talk I gave at a symplectic conference in Chengdu on June 28, 2015, and it intends to provide bridges and compatibility between various pairs of virtual techniques and to demonstrate some unity among various technical viewpoints in the constructions of structures on moduli spaces in symplectic geometry. More precisely, the abstract perturbative structures (or interchangeably, virtual structures) present in each virtual theory discussed in this paper (and sometimes even the way they essentially originate in applications) are identified pairwise in a way that intertwines the (non-)perturbation mechanisms. To be more helpful to readers and not get them buried under technicalities and notations, we give the ideas and appropriate level of details so that the results will be clear to the relevant experts; meanwhile the ideas of each virtual machinery and how they are related should come through to more application-minded readers so that they might get encouraged to read papers on a given virtual machinery and possibly apply it to remove some technical assumptions in their results. It is meant to be a service to the symplectic community.
\end{abstract}

\maketitle

\tableofcontents

\section{Introduction}

This article paints a grand unifying picture which illustrates in a precise way the compatibility among various virtual theories in the symplectic geometry literature. Namely, the structures featured in any two virtual theories included in the current discussion can be identified up to R-equivalences in a way that intertwines the perturbation mechanisms and algebraic constructions to extract invariants used in the respective virtual machineries. A couple of reasons for the existence of virtual harmony are (a) each approach is dealing with the same Fredholm problem (with varying domain and analytic limiting phenomena) and all the finite dimensional approaches at least in the differential geometric category are just shadows (outputs of forgetful functors constructed one way or another) of R-equivalent (sc-)Fredholm problems, and (b) in principle\footnote{To make this principle into a theorem, one needs to establish e.g. \ref{SIMPLEYETSIG} (2) among other things.} polyfold is a universal analytic framework that other analytic approaches factorize through. To reach wider readership while being directly useful, the article tries to strike a balance between having informal discussions of key ideas, reasons and/or summaries of each approach, and including precise formulations of definitions and tools/results and quoting locations of origins. Readers should not let technical details in certain sections deter them and should read relevant virtual theories that interest them. Compatibility of various existing gluing methods is not discussed here and might be pursued in a future article.

I must admit that I have not thoroughly investigated the entire package of each approach, and I mostly get informed by listening to talks by or through personal discussions with one/two member(s) of each group. The exceptions are Hofer-Wysocki-Zehnder's polyfold package for Gromov-Witten theory and Fukaya-Oh-Ohta-Ono's Lagrangian Floer theory in \cite{FOOO} (sans the gluing analysis part of the latter) that I have been looking at more closely and thinking about for quite a while. The identification of the structures and compatibility of the perturbation theories between HWZ's sc-Fredholm structure (or interchangeably, polyfold Fredholm structure) in polyfold theory and FOOO's Kuranishi structure imposed with two extra natural conditions \ref{MAXMATCHING} (1) and (3) are the topics of my thesis \cite{Dingyu}, and they are being improved and expanded as a series of papers \cite{DI}, \cite{DII} and \cite{DIII}, the latter two of which should be available in next few months. We remark that because \cite{DI} is only a slight improvement of Part I of \cite{Dingyu} while having tikZ diagrams and being more accessible through arXiv, I will henceforth only refer to results in \cite{DI} (rather than mentioning both \cite{Dingyu} and \cite{DI} all the time) despite the fact that \cite{Dingyu} appeared earlier (and now made available on my webpage). \cite{DII} and \cite{DIII} will be substantially expanded based on Part II of \cite{Dingyu} to include more related results, so that the forgetful and globalization functors will enjoy more properties.

First, let me say a word about my approach of Kuranishi structures in \cite{DI}. The starting point of my theory of Kuranishi structures is a modified version of FOOO's Kuranishi structures presented in \cite{FOOO}, so that it has better properties, e.g. easy to obtain a good coordinate system from it (almost like obtaining a finite cover for a compact topological space), and then just sufficiently precompactly shrink a good coordinate system (in a strong shrinking sequence), one can get a Hausdorff one to further construct things (level-1 structure, perturbation, fiber product, etc), and this easy process to achieve a Hausdorff good coordinate system is crucial for the \textbf{functoriality} aspect. The modification is done by imposing two conditions to FOOO's Kuranishi structures and they are called the maximality condition and the topological matching condition; and they are non-destructive to the original definition, as my goal was to unify HWZ's theory of polyfolds/polyfold Fredholm structures and FOOO's theory of Kuranishi structures after all, and this approach pays off in \ref{FOOOISDIR} which essentially says that the theory applies to FOOO's Kuranishi structures as well. A key point of the theory in \cite{DI} is that one should \textbf{view Kuranishi structures on the level of R-equivalences} (admitting a common refinement), e.g. perturbation theory factorizes through it, and the importance of this notion is also illustrated by the fact that following a loop of arrows (in the arrow direction) back to starting virtual structure in the unification diagram \ref{UNIFDIAG}, the effect is identity up to R-equivalence of that said structure (or effect is the identity morphism between R-equivalence classes). \textbf{The global nature of R-equivalence has the advantage that it has no structural damage to the virtual structure, and if one would like, he can also work on the level of Kuranishi structures, do things up to R-equivalence.} R-equivalence is more appealing and easier for functoriality if we have a common refinement (more geometric) rather than a zig-zag of refinements (of a cobordism flavor, less geometric) connecting two `essentially the same' Kuranishi structures; and then for R-equivalence to be an equivalence relation, it is important to have some extra structure for achieving transitivity. This structure is captured by the notion of \textbf{level-1 structure} and is very useful for other things, e.g. using a level-1 structure ten distinct virtual theories can be unified in this article, and level-1 structures also play a key role in defining a good notion of morphisms between Kuranishi structures (via localizing a class of maps generated by embeddings and submersions between Kuranishi structures), and in doing a single perturbation for the whole non-compact moduli spaces with compact filtrations by virtual structures with corners. I did not include a level-1 structure in my starting structure, because a choice of level-1 structure can always be summoned (and choices can be compared via common level-1 refinements): one can construct such global level-1 structures for $\mathcal{K}_X$ and for an embedding $\mathcal{K}_X\Rightarrow \mathcal{K}'_X$, and it is less minimal to be included in the basic definition.

To get the exact correspondence between HWZ's polyfold Fredholm structure and FOOO's Kuranishi structure version as explained in \cite{FOOO}, I recently established among other things in \cite{MAXCAT} that, for an FOOO's Kuranishi structure as defined in the \cite{FOOO} (see \ref{KURANISHISTRUCTURE}), there exists a Kuranishi structure which refines it (see definition \ref{KURANISHIREFINEMENT}) and has two conditions \ref{MAXMATCHING} (1) and (3) I required (see theorem \ref{FOOOISDIR}); and any two such refinements are R-equivalent \ref{REQUIV} as Kuranishi structures with these two conditions. This essentially removes those two conditions \ref{MAXMATCHING} (1) and (3) I imposed\footnote{Although these two conditions are convenient to work with (also as just said, we can always assume them after refining, and this is independent of choices up to R-refinement).}, because the theory in \cite{DI} factors through R-equivalence (admitting common refinements). Then by using the theory in \cite{DI}, one immediately has a (Hausdorff) good coordinate system \ref{GCS} (in a way without doing induction\footnote{Just chartwise modify to have a partial order, using compactness to get a finite covering which is a good coordinate system (see \ref{GCSWOIND}), and then take a sufficently small precompact shrinking to achieve Hausdorffness.}and applicable for a map between two such structures) which is a good coordinate system as defined in \cite{FOOO} (charts have global group actions and coordinate changes are global group equivariant embeddings intertwining sections with global group actions) with extra properties. The perturbation theory in \cite{DI} is a more structured treatment and a special case of FOOO's chartwise inductive abstract perturbation explained in \cite{FO} and \cite{FOOO}, and with the help of a level-1 structure (always achievable up to a precompact shrinking with the same index set \ref{LEVEL1CATEXIST}), this perturbation agrees on the nose with their CF-perturbation theory as detailed in \cite{FOOOCF}. We note that a good coordinate system constructed directly from a Kuranishi structure defined in \cite{FOOO} without extra conditions will a priori have orbifold sections in Kuranishi charts and orbibundle embeddings as coordinate changes. Thus, in \cite{FOOONEWER}, \cite{FOOOCF}, Kuranishi charts and coordinate changes in a Kuranishi structure are respectively formulated as (not necessarily global quotient) orbifold sections and orbibundle embeddings, so that charts and coordinate changes of a good coordinate system obtained are of the same form/category as those of the starting Kuranishi structure (as one might need to go back and forth, at least partially, between these two different types of structures), and abstract perturbation theory works as well. However, the version of \cite{FOOO}, \cite{Dingyu}, \cite{DI} are written with global quotient charts and coordinate changes. So on one hand, they can adapted to the general orbifold versions without much conceptual difficulty (see discussion in \cite{DI} 1.3 (4) about local models); on the other hand, it would still be nice to work harder as in \cite{MAXCAT} and achieve global quotient charts and coordinate changes in the associated good coordinate system directly without rewriting or re-interpretation of prior works. And the relationship of such a good coordinate system with the starting point of other approaches such as \cite{MW} and \cite{CLW} are closer. The key concepts in \cite{DI} are level-1 structures (\ref{LEVEL1CATEXIST}, which gives some immediate results in the spirit and direction of this article) and R-equivalence (\ref{REQUIV}, just called equivalence therein), and they will play important roles\footnote{As we shall see that the ability to equip level-1 structure to any Kuranishi structure (and any embedding between them) in \ref{LEVEL1CATEXIST} will directly identify many things on the nose; and in the grand unifying picture, starting from any structure at any point and following a sequence of directed arrows back, the effect is identity up to R-equivalence at that point.} in reaching the unifying picture.

Besides learning ingredients from other virtual machineries from their originators at conferences, preparing a talk about how the polyfold theory fits with other virtual machineries gives me an opportunity to take a very quick look at other virtual machineries, and I hope to update this article, as my knowledge about details and advantages of other virtual theories grow. Therefore, along the way, any comments and criticisms are welcome. Strictly speaking, this article does not validate each approach (in the sense of they having no errors or gaps and having full applicability) beyond demonstrating a virtual compatibility/harmony, because one needs to solve all issues along the way in constructing invariants starting from moduli spaces to have a successful theory (for example, one issue solved in one way or by a trick might lead to difficulties in the upcoming constructions to be addressed and worked out). The author wants to emphasize that each theory involves creativity and technical power and has its virtues and advantages, and their existence is not in danger after the identifications are shown here; because, for example, being efficient versus powerful in one way in one theory might lead to an opposite comparison in other parts within the same theory, and it always good to have many perspectives (especially when they are unified). 

Assuming each virtual theory works all the way, one should be able to use the identifications provided in the grand unifying picture in this article to view the methods, tricks, reasons or results of any other theories from the viewpoint of a given particular framework, and this should be beneficial, e.g. it might help solve his/her potential problems in new settings. Users of virtual theories will get a sense of unity and less overwhelmed. After various gluing methods (at least four distinct ways) get related in the future (maybe not possible to be related as perfectly as structures and perturbations of virtual theories in this article), each package will produce independent and yet agreeing/relatable results. Hopefully, eventually a common understanding of soundness of each virtual theory is reached, researchers in symplectic geometry will explore situations in fullest generality and examine the new phenomena not accessible in semipositive cases, and will not let the scare of virtual theories get in the way of finding out reality of symplectic geometry.

\textbf{Acknowledgement:} I am extremely grateful to Helmut Hofer for his enthusiasm and insights in math, his suggestion of PFH as a project and vision of existence of forgetful functor, numerous helpful discussions and suggestions, and supports. I am indebted to Kaoru Ono for helpful discussions, in particular informing me of Kawasaki's groupoid construction via the frame bundle from a Satake orbifold representative and explaining to me their several gluing constructions among various topics, and his kind invitation to RIMS. I am grateful to Alexandru Oancea for his kindness, encouragement, discussions and, during exchanges and seminars, conceptual remarks. I am thankful to Bohui Chen for the invitation to lecturing on polyfolds and the opportunity to give a talk from which the current article originates and for his explanation of their gluing construction. I am appreciative of Dominic Joyce's insights in email exchanges, helpful discussion on a category of Kuranishi structures (not R-equivalence classes), and sharing preprint \cite{2CAT} after I sent him the previous version of this paper. I thank the organizers for the invitations to the March conference on moduli space and one-month stay at Simons center, where I met and briefly discussed with most of the groups working on virtual theories discussed in this article and got to listen to their talks. I would like to thank Lino Amorim, Yalong Cao, Cheol-Hyun Cho, Octav Cornea, Huijun Fan, Urs Fuchs, Kenji Fukaya, Andreas Gerstenberger, Gregory Ginot, Roman Golovko, Kei Irie, Janko Latschev, Maksim Maydanskiy, Dusa McDuff, Hiroshi Ohta, Frederic Paugam, Tim Perutz, Bai-Ling Wang, Shaofeng Wang, Katrin Wehrheim, Jinxin Xue, and Jingyu Zhao for questions, suggestions and/or discussions on virtual theories or math in general, and Yalong Cao, Roman Golovko, Kei Irie, Dominic Joyce, Kaoru Ono and Jake Solomon for helpful comments on the previous version of this paper. I thank the great writing environment between talks in Oberwolfach where a draft of this article was completed. Finally, I thank the excellent research atmosphere in IMJ-PRG, in particular the interactions in the Algebraic Analysis group. The author is currently supported by the ERC via the Starting Grant ERC StG-259118 managed by the CNRS.

\section{Polyfold--Kuranishi correspondence}

The main result of my PhD thesis \cite{Dingyu} is about the abstract correspondence between polyfold theory and theory of Kuranishi structures, as well as a theory of Kuranishi structure R-equivalence classes via level-1 structure. This work was inspired from a vision of Helmut Hofer, that is, one should be able to construct a polyfold analogue \cite{Yang} of Joyce's innovative Kuranishi Homology (KH) \cite{KH}. The point of KH is that no perturbation is needed. For a simple example, a certain equivalence class of a map (added with some gauge-fixing/stability data to ensure finite automorphisms) from a Kuranishi structure without boundary to an orbifold will immediately give a cycle in Kuranishi homology theory of this orbifold without any perturbation. The perturbation is in some sense pushed to the proof of showing KH is isomorphic to rational singular homology, so it is like a non-perturbative homology theory (or universal perturbation theory). Since the polyfold constructions of moduli spaces are canonical as we shall briefly indicate below, polyfolds are global (with two-way coordinate changes) and sc-Fredholm theory formally looks like a Sard-Smale theory, the construction will be simpler, and after developing some properties of this non-perturbative homology theory, one can get invariants without ever worrying about transversality. I have a version of a construction of this envisioned Polyfold Fredholm Homology (PFH) \cite{Yang}, which is still not chain-level commutative yet due to infinite dimensionality. So with this PFH, a natural question would be how PFH and KH are related. This led to my effort in \cite{Dingyu} to functorially identify polyfold Fredholm structures and Kuranishi structures as geometric structures in a way intertwining respective perturbation theories (intertwining perturbation theories is not needed for PFH = KH, but is important in unifying structures in virtual theories devised for perturbations). This polyfold--Kuranishi correspondence \cite{Dingyu} has many nice properties (one reason for delay of my papers is that I want to have a good theory of structures at the output of functors and I want to get more properties for the functors); and I want to emphasize that the content is not trying to cast polyfold perturbation theory as a Kuranishi structure theory, instead I start from understanding an abstract Kuranishi structure on its own right and figure out a functorial way to convert between these two structures making both theories compatible and sometimes providing some new (even though so far just academic) understanding for each theory (e.g. two ways of perturbations in FOOO's Kuranishi theory agree on the nose via a level-1 structure; and one can still do compact perturbation in polyfold theory even if the ambient/base space does not support sc-smooth cut-off functions, by transferring via the forgetful map to compact perturbation of a Kuranishi structure R-equivalence class). The polyfold--Kuranishi correspondence is also important if one wants to construct a relative symplectic field theory which direcly extends both theory of symplectic field theory (SFT) \cite{SFT} using Hofer-Wysocki-Zehnder's polyfold theory \cite{PolyfoldI}, \cite{PolyfoldII}, \cite{PolyfoldIII} and \cite{Polyfoldinteg} together with Fish-Hofer-Wysocki-Zehnder's polyfold construction of SFT moduli spaces \cite{FHWZ} (based on \cite{Polyfoldanalysis} and \cite{PolyfoldGW}) and Fukaya-Oh-Ohta-Ono's construction of Lagrangian Floer theory \cite{FOOO} using their original version of Kuranishi structures.

To illustrate key ideas better, we will not focus on respecting fiber product structures on the boundaries of the abstract structures that are present for the moduli spaces considered up to certain finite energies and combinatorial types, but we will occasionally refer to this more general situation in passing in various places. We recall that the theory of moduli spaces in symplectic geometry can be roughly divided into three parts (possibly interwoven/combined in various approaches), an analysis part showing moduli space has an abstract perturbative structure, a perturbation theory part for this abstract perturbative structure to achieve transversality, and an algebraic construction part which organizes the end result of previous two steps into choice-independent algebraic invariants obtained from moduli spaces. Here independence of abstract choices made with given a fixed admissible geometric choice is due to mechanisms/constructions in the last two parts above, and independence of geometric choices made is shown by considering similar moduli spaces interpolating two different geometric choices and interpolating between two different such interpolations.

\subsection{Polyfold theory in a nutshell}

Polyfold theory are developed and elaborated in Hofer-Wysocki-Zehnder's series of papers, \cite{PolyfoldI}, \cite{PolyfoldII}, \cite{PolyfoldIII}, \cite{Polyfoldinteg}, \cite{Polyfoldanalysis}, \cite{PolyfoldGW}, \cite{PolyfoldSurveyNew} and \cite{PFI}. Good places to learn polyfolds are Hofer's updated survey \cite{PolyfoldSurveyNew}, polyfold Gromov-Witten paper \cite{PolyfoldGW} and book \cite{PFI}. 

To be concrete and have a feeling of polyfold notions and ingredients, let us look at Gromov-Witten moduli spaces following \cite{PolyfoldGW}: Let $(Q,\omega)$ be a closed symplectic manifold and choose a compatible almost complex structure $J$, then consider the moduli space $\overline{\mathcal{M}}(Q,\omega, J)$ of finite energy stable marked noded $J$-holomorphic curves (such a $J$-holomorphic curve is an isomorphism class of a $J$-holomorphic map from a marked noded Riemann surface to $(Q,J)$, where an isomorphism is a diffeomorphism of Riemann surfaces identifying between two tuples of data: complex structures, nodes, marked points and maps). Each of its component is compact and can be enumerated as $\overline{\mathcal{M}}_{g,m,\beta}(Q,\omega,J)$ for $g\geq 0, 3-2g-m\leq 0, \beta\in H_2(Q,\mathbb{Z})$. One needs to make $\overline{\mathcal{M}}(Q,\omega, J)$ transverse before one can define Gromov-Witten invariants by integration to satisfy Kontsevich-Manin axioms. The polyfold approach is to view $\overline{\mathcal{M}}(Q,\omega, J)$ inside the stable map space $B(Q,\omega)$ (see \cite{PolyfoldGW} 1.4), and $\overline{\mathcal{M}}(Q,\omega, J)$ is the solution space (in the orbit space) of a section $\bar\partial_J: B(Q,\omega)\to E(Q,\omega, J)$. 

To describe a chart around $\alpha\in [\alpha]$ in $B(Q,\omega)$, one stabilizes the domain Riemann surface of $\alpha$ by adding extra marked points that are chosen with some rules to satisfy (\cite{PolyfoldGW} 3.1), describes the variation of newly marked and now stable domain Riemann surface using (reformulated) Deligne-Mumfold theory, then considers the variation of maps near the map part of $\alpha$ with codimension two transversal constraints at the images of those new marked points while allowing gluing of those maps at nodes. After specifying a tuple of data with properties, one can have a chart invariant under $G:=Aut(\alpha)$, called a \textbf{good uniformizer}, see \cite{PolyfoldGW} 3.9. Coordinate changes are included as special cases of the \textbf{universal property} of a good uniformizer, which says that given a variation $t\mapsto\beta_t$ and a good uniformizer $\tau\mapsto \alpha_\tau$ such that $\beta_{t_0}$ is isomorphic to $\alpha_{\tau_0}$ via $\phi_0$, then there exist unique $t\mapsto \tau(t)$ with $\tau(t_0)=\tau_0$ and a unique family of isomorphisms $\phi_t:\beta_t\mapsto \alpha_{\tau(t)}$ extending $\phi_{t_0}=\phi_0$ (see \cite{PolyfoldGW} 3.13 and for the moment we postpone the sc-smoothness content of this important theorem and revisit it shortly). 

To have perturbation theory (which boils down to the contraction mapping theorem on a Banach space somewhere), we need to consider the stable map space locally as a Banach space, hence being of finite regularity (for example, $H^{3,\delta}$ with $\delta>0$ in a holomorphic polar coordinate around nodes and $H^3_{\text{loc}}$ around points elsewhere). The aforementioned $\tau(t)$ in the universal property is a composition of maps, one of which is a domain reparametrization and nowhere $C^1$. This together with the forms of other maps appearing in $\tau(t)$ means that $\tau(t)$ is nowhere $C^1$, hence in particular we only have continuous coordinate change on $B(Q,\omega)$ for this choice of smooth structure $H^{3,\delta}$, or any other classical choice.

For clarity, we abstractly write $f:E_0\to F_0$ for a localized domain reparametrization $(u,a)\mapsto u\circ\phi_a$ where $a\mapsto \phi_a$ is a family of embeddings between disks in $\mathbb{C}$ parametrized by a finite dimensional vector space $\{a\}$ and $u$ is a map from a fixed disk to $\R^{2n}$ of class $H^{3,\delta}$. Observe that if we restrict to a subspace $E_1\subset E_0$ of pairs $(u,a)$ where the map $u$ in the first entry enjoys one higher regularity (here $E_1$ is still considered as a Banach space, a subspace of $E_0$ singled out by a stronger norm), then $f: E_1\to F_0$ is $C^1$. Then for $x\in E_1$, the Fr\'echet derivative $df(x): E_1\to F_0$ drops regularity levels, which prevents composition of differentials, let alone the chain rule. But if we assume $E_1$ is dense in $E_0$ and $df(x)$ extends to $Df(x): E_0\to F_0$ so that $Df(x)$ is bounded, then we can still hope for the chain rule. If we try to establish the chain rule for the composition of two maps of these features, we really need the inclusion $E_1\subset E_0$ to be also compact, similar requirement for a $F_1$ of $F_0$ with $f(E_1)\subset F_1$ (as $F_1$ is the base domain of another tangent map to compose) as well as the continuity of $(x,v)\mapsto Df(x)v$, $E_1\times E_0\to F_0$. We will asume all that for this class of maps, but this is not enough: In order to have some structure in place to talk about one higher smoothness inductively, we need to consider above again for the tangent map $Tf: (x,v)\mapsto (f(x), Df(x)v), E_1\times E_0\mapsto F_1\times F_0$ in the place of $f: E_0\to F_0$; a natural choice of data analogous to $E_1$ would be $E_2\times E_1\subset E_1\times E_0$ with a compact dense inclusion. The first factor of the map $Tf$ is $f:E_1\to F_1$, so we need the assumptions in the above discussion to hold with subscripts shifted up by 1. Then we consider $T(Tf), \cdots, T(T^{k-1}f), \cdots$. So in general, we need $E_0\supset E_1\supset E_2\supset E_3\supset\cdots\supset E_{k+1}\supset\cdots$ a nested sequence of Banach spaces with the next one compactly densely included into the previous one, this $(E_m)_{m\in\mathbb{N}_0}$ is precisely the \textbf{sc-structure} (for an sc-Banach space). All of the above assumptions are summarized as \cite{Polyfoldanalysis} 2.1. However, an equivalent definition which is more direct and easier to use and check is provided in \cite{PolyfoldI} 2.13 as the official definition of $sc^1$, the new smooth notion of a map. The amazing fact is that this is not the order HWZ discovered it: They found the official definition first and \cite{Polyfoldanalysis} 2.1 was then obtained when they tried to relate it to the classical smoothness. This gives a glimpse of the analytic prowess of HWZ. $sc^k$ can be defined inductively via $T^{k-1}f$. Being \textbf{sc-smooth} just means being $sc^k$ for all $k$. We can regard sc-smoothness for a map $f$ between two sc-Banach spaces as a property for the map $f$ between 0-levels and its actual sc-smoothness is then specified by the information from the sc-structures of domain and target sc-Banach spaces.

An sc-open set $O$ in an sc-Banach space $E$ means $(O_m)_{m\in\mathbb{N}_0}$ such that $O_0$ is an open set in $E_0$ and $O_m=O_0\cap E_m$. Following HWZ, we often use $E$ to mean both level-0 Banach space $E_0$ as well as $(E_m)_m$. It is often clear in each context which meaning is being used; and when there are two possible meanings, the contents are the same. For example, $O\subset E$ can mean both the inclusion between the 0-th levels and the levelwise inclusion, and both contents are equivalent by definition. This new sc-smooth notion gives rise to new a local model $O:=r(U)$, where $r:U\to U$ is an sc-smooth map between the same sc-open set $U$ in sc-Banach $E$ and idempotent ($r\circ r=r$). We only need to remember $(O,E)$ (for the case without corners), and all the related constructions using a choice of $r$ with $O=r(U)$ are independent of such choices. We will omit $E$ in the notation and write retract just as $O$, but this means there is an associated chosen sc-Banach space $E\supset O$\footnote{Different such $E$'s (of possibly different dimensions) can give rise to the same sc-structure on $O$, and one can also fix an equivalence class of such a choice $E$ in this sense. Namely, $E$ and $E'$ are equivalent if there exist sc-smooth maps $\phi: E\to E'$ and $\phi':E'\to E$ such that $(\phi'\circ\phi)|_O=Id_O$ and $(\phi\circ\phi')|_O=Id_O$. Note that we do not need $\phi|_O=Id_O$.} giving rise to a smooth structure on $O$ and all the relevant discussions (e.g. sc-smoothness of a map from $O$ and its tangent) are done via $E$. $O$ is called a \textbf{retract}, and in general it cannot be locally modelled by an sc-Banach space.

In fact gluing of (not necessarily $J$-holomorphic) stable maps (see \cite{PolyfoldGW}) can be described as a retract, and when we turn on the gluing parameter $a\not=0$, the modulus of which is converted via a \textbf{gluing profile} into the length of the glued domain cylinder, and two half cylinder maps (arised by localizing near nodes) are twist-glued together into a finite cylinder map, and this process is denoted by $\oplus_a$ which is a linear map from two unglued half cylinders into a twist-glued cylinder (and is the identity when $a=0$). The end regions of maps are not used in gluing, hence the dimension of glued map space (as an ambient space of interest) jumps at $a=0$. We can find a linear map $\ominus_a$ such that $(\oplus_a,\ominus_a)$ is an isomorphism. So that we can parametrize geometric glued cylinder space $\bigsqcup_a\text{im}\oplus_a$ by $\bigsqcup_a\ker\ominus_a=r(U)$ realizeable as a retract, where maps in $r(U)$ have the same domain to compare. So $\bigsqcup_a\ker\ominus_a$ provides a local model as a retract for space $\bigsqcup_a\text{im}\oplus_a$ which is dimension varying. We can use a retract built from this plus other information as a parameter space of a good uniformizer of $B(Q,\omega)$ centered at a smooth map such that the universal property holds sc-smoothly. This says a good uniformizer is an \textbf{M-polyfold} (a paracompact and second countable space build with retracts as local models with sc-diffeomorphic coordinate changes, and here we just have a single chart). 

We can also keep track of symmetry using \'etale proper Lie groupoid language based on M-polyfold with a generalization of properness (\cite{PolyfoldIII} 2.2) and such a groupoid is called an \textbf{ep-groupoid}. From properties of $B(Q,\omega)$ and sc-smooth universal property of good uniformizers, one can construct an ep-groupoid $B$ such that the orbit space $\underline{B}$ (modulo isomorphisms) is identified with $B(Q,\omega)$. A Morita-equivalence class of an ep-groupoid is called a \textbf{polyfold}. One observes that $B(Q,\omega)$ is canonically a polyfold after fixing merely a choice of appropriate gluing profile and a sequence of strictly increasing non-zero weights below $2\pi$ for the sc-structure (important for sc-Fredholmness later), see also \ref{PKAPPLICATION} (3).

A first order differential operator maps as a (local) section from the base into the fiber with 1 functional regularity lower. If this is defined notationally as $U_m\to F_m$, then $F_{m+1}$ (from functional regularity perspective) still makes sense over a point in $U_m$. This structure is encoded in a local \textbf{strong bundle}, written as $U\triangleleft F$, and is respected by strong bundle maps (linear in the fiber direction). A strong bundle retract typically is denoted by $K:=R(U\triangleleft F)$. A section and respectively an $sc^+$ section, which mapping from the base to fibers preserves the regularity and respectively increases the regularity by one, make sense. We can lift a good uniformizer of $B(Q,\omega)$ to a good uniformizer of $E(Q,\omega, J)$ with the parameter space of a good uniformizer for $E(Q,\omega, J)$ as a strong bundle retract. Then $E(Q,\omega, J)$ becomes a strong bundle polyfold the same way as $B(Q,\omega)$ is a polyfold. The Cauchy-Riemann operator $\bar\partial_J$ is an sc-smooth section of strong bundle $\pi: E(Q,\omega, J)\to B(Q,\omega)$. Although it is a section between spaces of varying dimensions, near every smooth point, we can locally fill $f:r(U)\to R(U\triangleleft F)$ into the graph $\text{gr}(h)$ for some \textbf{filling} (\cite{Polyfoldanalysis} 1.35) $h: U\to F$ between sc-Banach spaces without changing its zero set and kernel and cokernel of its linearization at zeros and thus in essence has the same\footnote{Here, $\text{gr}(h)$ being an extension of $f$, namely $f$ sits in $\text{gr}(h)$, is also important.} Fredholm behavior. $h$ has linear sc-Fredholm linearizations but implicit function theorem does not hold for general sc-smooth maps with linear sc-Fredholm linearizations, because $x\mapsto Dh(x), E_1\mapsto L(E_0, F_0)$ is not continuous, which is used in the usual argument. We have to restrict to a special class of maps, called basic class, where a version of implicit function theorem holds. First, let us define an open germ around $0$ in $E$, namely, $\mathcal{O}(E,0):=(V_m)_{m\in \mathbb{N}_0}$ is an \textbf{open germ} around $0$ in $E$ if $V_m$ is open in $E_m$ and contains 0, and $V_{m+1}\subset V_m$\footnote{Here $\cap V_m=\{0\}$ is allowed. This is as opposed to an sc-open set $(U_m)_{m\in\mathbb{N}_0}$ in $E$ which means $U_0$ being open in $E_0$ and $U_m=U_0\cap E_m$, and thus $\cap U_m$ is dense in $U_0$.}). A map germ mapping from $\mathcal{O}(E,0)$ is also denoted by $[f,0]$. A map germ $h: \mathcal{O}(\R^m\oplus W,0)\mapsto (W\oplus \R^N,0)$ is \textbf{of basic class} if it is of the form $(v,x)\mapsto (x-B(v,x), C(v,x))$ such that $B(v,\cdot)$ between each level of the same regularity is a contraction parametrized by $v$ with arbitrarily small contraction constant if $(v,x)$ is sufficiently close to $(0,0)$. 

\begin{definition} (sc-Fredholm section, \cite{PolyfoldII} 3.6, \cite{Polyfoldanalysis} 1.40) An sc-smooth section of a strong M-polyfold bundle is called an \textbf{sc-Fredholm section}, if it is regularizing\footnote{All $x$ of the level $m$ mapping under $f$ to the level $m+1$ in the fiber is actually of the level $m+1$, for all $m\in\mathbb{N}_0$.}, and locally at every smooth point $x$, $f$ is of the form $f: (r(U),0)\to (R(U\triangleleft F),f(0))$ (still denoted by $f$) and there exist a filling germ $[h,0]$ and an $sc^+$ section germ $[\text{gr}(t),0]:\mathcal{O}(U,0)\to U\triangleleft F$ with $t(0)=h(0)(=f(0))$ such that $[h-t,0]$ under some strong bundle diffeomorphism is a map germ of basic class.
\end{definition}

A germ of the solution space of $f$ locally around $0$ corresponds (via some compact perturbation $t$ given in the definition) to a germ of solution space of $\R^n\mapsto \R^m, v\mapsto C(v, x(v))$ around $v=0$, where $x(v)$ is the unique fixed point of $B(v,\cdot)$. If the former thus the the latter is transverse (in good position) at $0$, then there exists a good parametrization (\cite{PolyfoldII} 4.2) for a manifold germ with corners at $0$, the collection of those can be sc-smoothly (hence smoothly in the usual sense as being in finite dimension) compatibly fit together in the base M-polyfold of the section. This tranversality can always be achieved, and it involves locally reducing to a finite dimensional situations (where one never needs to fit local things together globally and then perturb), and one just locally perturbs them (by multisection if keep tracking of symmetry) and brings them back to the global structure and use a global canonical section type argument. Then by the version of the above with group actions, one can show that the perturbed solution set is just a usual finite dimensional weighted branched orbifold. We need the level-0 Banach space involved in defining retract for the base to be a separable Hilbert space to have sc-smooth cut-off functions, and from this we can construct an auxiliary norm\footnote{In \cite{PolyfoldII}, existence of an auxiliary norm requires reflexive $(0,1)$-fibers; but in \cite{PFI}, a more general auxiliary norm is introduced which always exists and it was stated there this general version works just as well. It seems that reflexivity of auxiliary norms plays a role from Hofer's SFT polyfold lectures at IHES.} to control the size of the perturbation so that the solution set of perturbed section is still compact. This is polyfold Fredholm theory developed in \cite{PolyfoldII} and \cite{PolyfoldIII}. 

To directly check that a given section is sc-Fredholm might be hard, but there is a criterion for sc-Fredholmness in \cite{PolyfoldGW} 4.26 abstractly summarizing the content of section 4.5 therein on deriving sc-Fredholmness from properties 4.23 and 4.25 (see also \cite{Wehrheim} 4.3 and 4.5), and $\bar\partial_J$ satisfies this criterion (which essentially boils down to the fact that one can localize and rearrange into the situation of standard $\bar\partial_0$ between Banach spaces on each level, which one can then solve and have norm controls throughout this process, see \cite{PolyfoldGW} section 4.6 and chapter 5).

Although I have quickly explained the concepts in polyfold theory with the help of Gromov-Witten setting for concreteness, the sc-analysis and the polyfold theory of sc-Fredholm section can be completely separated into independent parts: abstract Fredholm part in \cite{PolyfoldI}, \cite{PolyfoldII} and \cite{PolyfoldIII}, and sc-analysis part for $J$-holomorphic curves in \cite{Polyfoldanalysis} and \cite{PolyfoldGW}. This means an easy transfer to the setting of a new moduli spaces. If a moduli space is not compact, but has compact filtrations with boundary being a union of (covers of) fiber products of some previous filtrations (as in SFT or Lagrangian Floer theory), one hopes to (and can in the SFT and Lagrangian Floer cases) perturb to achieve sufficient transversality while keeping this `master equation' structure and if this is achievable, this can be done using a single multisectional perturbation\footnote{This is also doable via level-1 structure for a moduli space filtered by Kuranishi structures with corners where the boundary of a Kuranishi structure is a union of fiber products of Kuranishi structures appeared earlier in the filtration, see \ref{LEVELONECC} and \cite{MAXCAT}.} on the whole moduli space (constructed on filtrations inductively). Then one uses his/her favorite algebra model to extract invariants. HWZ in \cite{Polyfoldinteg} uses an sc-version of de Rham cohomology.

\subsection{Sc-Fredholm sections viewed as generalizations of classical Fredholm sections.}

To see some polyfolds in action in an extremely simple case and to have some feeling of how sc-setting works and also in the unifying spirit of this article, we explain next that polyfold Fredholm theory can be regarded as a vast generalization of classical Fredholm theory in the following sense. We emphasize that theorem \ref{SCFREDGENERALIZE} below will not capture general situation in polyfold theory, but it fits certain classical Fredholmness into the polyfold framework. This subsection is purely instructive to somewhat illustrate the concept of sc-Fredholmness and not essential in what follows. Since the result is local, we can assume the bundle is trivial, and it can be packaged into a global setting without additional work except notation. The following is a simple observation from reading \cite{PolyfoldGW} and \cite{Wehrheim}, but this viewpoint of sc-Fredholmness generalizing usual Fredholmness in some sense seems to be new.

\begin{theorem}\label{SCFREDGENERALIZE} (sc-Fredholmness generalizing usual Fredholmness)
\begin{enumerate}[(1)]
\item If $f:U\to F$ is a $C^\infty$ section in a Banach bundle $U\times F\to U$ in the usual sense, and suppose that there exist sc-structures $(U_m)_{m\in\mathbb{N}_0}$ and $(F_m)_{m\in\mathbb{N}_0}$ for $U$ and $F$ respectively, and the induced $f: U_m\to F_m$ is $C^\infty$ for all $m$, then $f$ is sc-smooth. 
\item If additionally $f:U_m\to F_m$ is classically Fredholm\footnote{\label{NEARZERO} $f: U_m\to F_m$ has Fredholm linearizations at zeros, which then implies by openness of Fredholm condition that, $f-f(p): U_m\to F_m$ has Fredholm linearizations for all points $p\in U_m$ sufficiently close to $f^{-1}(0)$.} with the same Fredholm index for all $m$ and regularizing (if $y\in U_m$ and $f(y)\in F_{m+1}$, then $y\in U_{m+1}$), then $f$ restricted to an open germ near $f^{-1}(0)$ is sc-Fredholm in the polyfold sense (of basic class germ after bringing the value of section to $0$, filling and a strong bundle coordinate change). 
\item Moreover, if such an $f$ is in good position (transverse) in the usual sense, then it is trivially in good position in the polyfold sense, then the zero set is a smooth manifold with corners in the usual sense in both settings and the induced smooth structures agree.
\end{enumerate}
\end{theorem}
\begin{proof}
\begin{enumerate}[(1)]
\item Check directly or apply \cite{Polyfoldanalysis} 2.4.
\item We first use a lemma: 
\begin{lemma}\label{LINEARSCFREDEQUIV} $T: E\to F$ is linear sc-Fredholm (\cite{PolyfoldI} 2.8) if and only if the induced $T: E_m\to F_m$ are linear Fredholm in the usual sense with the same Fredholm index for all $m\in\mathbb{N}_0$. In fact, if the latter holds, a Fredholm splitting for $T: E_0\to F_0$ will induce a Fredholm splitting for $T:E_m\to F_m$ for all $m$.
\end{lemma}

\begin{proof} Only the `if' part is non-trivial. Let $E_0:=K\oplus H$ and $F_0:=H'\oplus C$ be a Fredholm splitting\footnote{$K:=\ker T$ and  $C:=F_0/T(E_0)$ finite dimensional, $H':=T(E_0)$ is Banach and $T|_{H}: H\to H'$ is a linear isomorphism between Banach spaces.} for $T: E_0\to F_0$. Hence the inverse of $T|_H$ is bounded, and the injective estimate $\|(T|_H) (h)\|_{H'}\geq c\|h\|_{H}$ holds. It is better for the following to denote $T: E_m\to F_m$ etc explicitly as $T|_{E_m}$ and denote $\iota_m:E_m\to E_0$ and $\tilde\iota_m:F_m\to F_0$, the inclusions between Banach spaces in the sc-structures. Note that $\tilde\iota_m\circ (T|_{E_m})=T\circ\iota_m: E_m\to F_0$. We will arrive at the desired conclusion via a few elementary observations:
\begin{enumerate}[(a)]
\item $T|_{H\cap E_m}:H\cap E_m\to F_m$ is injective for all $m$:

If $v\in H\cap E_m$ such that $Tv=0$ in $F_m$, then $T(\iota_m(v))=(T\circ\iota_m)(v)=(\tilde\iota_m\circ (T|_{E_m}))(v)=0$ in $F_0$ and $\iota_m(v)\in H\cap E_0=H$, thus $\iota_m(v)=0$ in $H\cap E_0$. The injectivity of $\iota_m: E_m\to E_0$ hence the induced $\iota_m|_{H\cap E_m}: H\cap E_m\to H\cap E_0$ gives $v=0$ in $H\cap E_m$.

\item $K\cap E_m=K$ for all $m$. 

Suppose not true for some $m$. Since $K$ is finite dimensional, and $K\cap E_m\subset K$ but $K\cap E_m\not=K$, then $\overline{K\cap E_m}=K\cap E_m\not=K$. Then there exists a $k_0\in K$, such that $\|k-k_0\|_K>\epsilon_0>0$ for all $k\in K\cap E_m$. Since $E_m$ is dense in $E_0$, there exists $x_i\in E_m$ such that $\lim_i\|x_i-k_0\|_{E_0}=0$. So $T(x_i)\to T(k_0)=0$ in $F_0$, so $T(x_i)=T(pr_K x_i)+T(pr_H x_i)=T(pr_{H} x_i)\to 0$ in $F_0\cap T(E_0)=H'$. By the above injective estimate of $\|pr_H x_i\|_H\leq \frac{1}{c}\|(T|_H) (pr_H x_i)\|_{H'}\to 0$, we have $pr_H x_i\to 0$ in $H$ and in $E_0$. So, $0<\epsilon_0\leq \lim_i\|pr_K x_i-k_0\|_{E_0}=\lim_i\|(x_i-pr_H x_i)-k_0\|_{E_0}\leq \lim_i\|pr_H x_i\|_{E_0}+\lim_i\| x_i-k_0\|_{E_0}=0$, a contradiction.

\item $K\cap E_m\subset \ker (T|_{E_m})$ for all $m$:

If $v\in K\cap E_m$, again by $(\tilde\iota_m\circ (T|_{E_m}))(v)=(T\circ\iota_m)(v)=T(v)=0$. Since $\tilde \iota_m$ is injective, $(T|_{E_m})(v)=0$ in $F_m$, so $v\in \ker(T|_{E_m})$.

\item $\ker (T|_{E_m})=\ker T=K$ for all $m$:

By (b) and (c), we have $\ker (T|_{E_m})\supset K\cap E_m=K=\ker T$. As $\ker (T|_{E_m})\subset \ker T$, we have $\ker (T|_{E_m})=\ker T$.

\item $T(E_m)=T(E_0)\cap F_m(=H'\cap F_m)$ and $F_m/T(E_m)=F_0/T(E_0)=C$ for all $m$:

Note that $T(E_m)\subset T(E_0)\cap F_m$. Next we establish that this inclusion is an equality. Since $F_m$ is dense in $F_0$, we have a dense inclusion $F_m/(T(E_0)\cap F_m)\subset F_0/T(E_0)$ into a finite dimensional space, hence $F_m/(T(E_0)\cap F_m)=F_0/T(E_0)$. The inclusion between two closed subspaces of $F_m$ at the start of this paragraph gives $(\ast)$ $F_m/T(E_m)\supset F_m/(T(E_0)\cap F_m)=F_0/T(E_0)$. Since $\text{ind}\;T=\dim K-\dim F_0/T(E_0)=\text{ind} (T|_{E_m})=\dim K-\dim F_m/T(E_m)$. So, $F_m/T(E_m)=F_0/T(E_0)=C$, and this together with $(\ast)$ gives $T(E_m)=T(E_0)\cap F_m=H'\cap F_m$. The latter says that the image of $T|_{E_m}$ is induced from $H'=T(E_0)$, and it is actually the regularizing property in \cite{Wehrheim} 3.1 (ii) (if $x\in E_0$ with $T(x)\in F_m$, then $x\in E_m$).

\item $T$ is linear sc-Fredholm with the sc-splitting of $E$ and $F$ induced from the Fredholm splitting of $T: E_0=K\oplus H\to F_0=H'\oplus C$ chosen above (with any fixed choice $H$ and hence $H'$):

Norms on a finite dimensional space are equivalent. Thus by (a), (d) and (e), the Fredholm splitting of $T: E_0\to F_0$ induces Fredholm splittings for $T|_{E_m}$ for all $m\in\mathbb{N}_0$, and $T:K\oplus (H\cap E_m)_m\to (H'\cap F_m)_m\oplus C$ is linear sc-Fredholm with constant filtrations $K$ and $C$.
\end{enumerate}
\end{proof}

Now we try to verify sc-Fredholmness of $f: U\to F$ via the HWZ's criterion \cite {PolyfoldGW} 4.26. As in footnote \ref{NEARZERO}, by shrinking $U$ towards $f^{-1}(0)$ if necessary, we can assume for each $p\in \cap_m U_m$ (thus $f(p)\in \cap_m F_m$), $\tilde f_p:=f-f(p): U_m\to F_m$ has a linear Fredholm linearization at $p$ in the usual sense for all $m\in\mathbb{N}_0$.

For any $p\in \cap_m U_m$, we try to establish those three criteria in \cite{PolyfoldGW} 4.26. Criterion (1) holds by lemma \ref{LINEARSCFREDEQUIV} and the last paragraph. Criterion (3) holds by levelwise $C^\infty$ thus in particular, $C^1$. Criterion (2) is slightly non-trivial and we will show it below:

By shifting of the domain $U$, we can assume $p=0$ (mainly for notational simplicity), and we have the sc-splitting of the linearization $Dh(0): K\oplus H\to H'\oplus C$ of $h:=f-f(0)$ at $p=0$, by lemma \ref{LINEARSCFREDEQUIV} again, here $H=(H_m)_m$ and $H'=(H'_m)_m$ for short. We want to establish that:

\begin{lemma} For any level $m$, a sequence $x_j\in K$ converging to 0, and $v_j\in H_m$ bounded satisfying $Dh(x_j) v_j=y_j+z_j$ for $y_j\in F_m$ converging to $0$ and $z_j\in F_{m+1}$ bounded, then $v_j$ has a convergent subsequence in $H_m$. (Notice that this is not quite \cite{PolyfoldGW} 4.26 (2), as there $\lim_j x_j\not=0$ in $K$ in general, and our weaker condition will only lead to \cite{PolyfoldGW} 4.23 therein valid for $b$ in a germ around 0, which is enough as basic class is a germ condition.)
\end{lemma}

\begin{proof} Note that $Dh(0)v_j=(Dh(0)-Dh(x_j))v_j+y_j+z_j=:\tilde y_j+z_j$, where $\tilde y_j\in F_m$ also converges to 0. Note that $pr_{H'}z_j$ can be solved in $H_{m+1}$ and solution $w_j$ is uniformly bounded using the injective estimate (following from $Dh(0)|_{H_{m+1}}:H_{m+1}\to H'_{m+1}$ is a linear isomorphism between Banach spaces): $c_{m+1}\|w_j\|_{H_{m+1}}\leq \|(Dh(0)|_{H_{m+1}})w_j\|_{H'_{m+1}}=\|pr_{H'}z_j\|_{H'_{m+1}}\leq c'_{m+1}\|z_j\|_{F_{m+1}}$. Then $Dh(0)v_j=\tilde y_j+pr_C z_j+Dh(0)w_j$ in $E_m$. Then projecting onto $C$, we have $\|pr_C z_j\|_{F_m}=\|pr_C \tilde{y}_j\|_{F_m}\leq c''_m\|\tilde y_j\|_{F_m}\to 0$. Since $v_j\in H_m$, use the injective estimate again but on the $m$-level, we have:

\begin{align*}c_m\|v_j-w_j\|_{H_m}&\leq \|Dh(0)(v_j-w_j)\|_{H'_m}=\|\tilde y_j+pr_C z_j\|_{F_m}\\
&\leq (1+c''_m)\|\tilde y_j\|_{F_m}\to 0.\end{align*} 

Thus $v_j-w_j$ converges to 0 in $H_m$. Since $w_j$ is bounded in $H_{m+1}$ which is compact in $H_m$, we can assume a subsequence $w_{j'}$ converges to $w_\infty$ in $H_m$. Then $v_{j'}=(v_{j'}-w_{j'})+w_{j'}$ converges in $H_m$ (to $w_\infty$).
\end{proof}

This proves criterion (2), thus by \cite{PolyfoldGW} 4.26, $f$ is sc-Fredholm at 0 (namely, $[f,0]$ is an sc-Fredholm germ), hence at every $p\in \cap_m U_m$ after shifting back the domain.

\item Due to the regularizing property, the zero sets of $f$ in both viewpoints agree as a set and as a manifold with the same induced smooth structure.
\end{enumerate}
\end{proof}
\begin{remark}\label{SIMPLEYETSIG} We include three remarks:
\begin{enumerate}[(1)]
\item Theorem \ref{SCFREDGENERALIZE} strictly speaking does not fit all Fredholm sections in the usual sense as sc-Fredholm sections (it only says being levelwise Fredholm in the usual sense with the same Fredholm index and regularizing implies sc-Fredholmness). One can fit both types of Fredholm sections in a single framework by a trick I learned from a talk by Hofer that one can define sc-structure that if $E_{m+1}\to E_m$ is not compact for some $m$, then $E_m=E_0$ as Banach spaces for all $m\in \mathbb{N}_0$ (namely, it is either a genuine sc-structure or a constant structure $(E_m=E_0)_{m\in\mathbb{N}_0}$ of possibly infinite dimensional Banach space $E_0$). However, the point of \ref{SCFREDGENERALIZE} is to fit a large class of Fredholm sections as genuine sc-Fredholm sections without modifying the definition of the latter and gives some feeling of the latter as a much more general object. Almost all classical Fredholm sections appeared in moduli spaces satisfy the hypothesis of \ref{SCFREDGENERALIZE}: They often arise as first order elliptic operators thus are regularizing. They are defined as maps between (weighted) Sobolev spaces which can be equipped with natural sc-structures. The reason they are smooth in the usual sense as a map between a fixed level is also the reason they are smooth as a map between higher levels. The Fredholm indices are usually topological invariants independent of levels in the Sobolev set-up (if levels are specified carefully not to change the Fredholm indices, e.g. if the domains of maps appeared in the domain of Fredholm section are fixed but non-compact, and the weights are chosen below the first positive eigenvalues of the asymptotic operators and above the largest non-positive eigenvalues\footnote{Increasing weights are for compact embeddings between levels, namely, sc-structure.}). One can also have a perturbation correspondence in theorem \ref{SCFREDGENERALIZE}, which is left to the readers.
\item \ref{SCFREDGENERALIZE} symbolically denoted as $\textbf{Emb}$ is useful in the following: It is a unit of building sc-Fredholm section in a local chart with no domain reparametrization. To show gluing intertwines with identifications in this paper, we hope to achieve $\textbf{Emb}\circ$(gluing of classically Fredholm sections/Kuranishi charts) can be filled into $\text{(polyfold gluing)}\circ \textbf{Emb}$.
\item \ref{SCFREDGENERALIZE} (2) can also be proved using $Df(\cdot)$ is continuous with respect to the operator norm levelwise.
\end{enumerate}
\end{remark}

\subsection{In good and general position, fiber-standard subbundle, and (finite) dimensional reductions}

Since Kuranishi structures and their variants are compatible systems of local finite dimensional reductions, we will focus on ways of (finite) dimensional reductions in this section. Some of results below are standard for classical Fredholm sections\footnote{We assume that all Banach manifolds are second countable and paracompact.}, but now viewed together with relationship pointed out, or with new applications, and generalized to situations for sc-Fredholm sections in polyfold setting. Although the definition and results are phrased for an sc-Fredholm section, they will also make sense for a Fredholm section in the usual sense by taking all the sc-structures to be constant Banach spaces of possibly infinite dimensions (thus not compact inclusions between levels, but the chain rule also holds classically), the retraction is identity, a local filling is trivially the section itself, and a strong bundle ep-groupoid is then replaced by a Banach bundle \'etale-proper Lie groupoid, so that we do not need to write the same definitions and results twice. Some results are from \cite{Dingyu}, and some are realized in writing this paper and will also be included in \cite{DII}. Some results are of local nature and stated locally but can be thought of as a local situation invariant under the stabilizer group for an sc-Fredholm section in a global strong bundle ep-groupoid. We state  some \textbf{conventions}: We use the same letter for an ep-groupoid as well as its object space, which we will emphasize again in subsection \ref{FORGETSUBSECTION}. We denote $E|_p$ for the fiber of $E$ at $p$, as we might use $\cdot_p$ to denote an object associated to and depending on $p$. A section being transverse means it being transverse to the zero section. A (transverse) perturbation $s$ to $f$ means that $f+s$ is transverse in some sense (e.g. in good position), and a perturbation $\tilde f$ of $f$ means that $f'$ is transverse. For an M-polyfold $B$, we use $B_\infty$ to denote the set of smooth points in $B$ which locally is $\varphi(O_\infty)=\varphi(r(U_\infty))$ for some local chart $\varphi: O\to B$ and $U_\infty:=\cap_m U_m$. We will often loosely call a functor between ep-groupoids as a map between them.

\begin{definition}\label{IGPPOINTWISE} (in good position, \cite{DII}, also c.f. \cite{PolyfoldII} 4.10 4.14, and 4.16) Let $f:B\to E$ be an sc-Fredholm section where $B$ may have corners (based on tame retracts in partial quadrants \cite{PFI} 2.34), $x\in B_\infty$, and $N\subset (E|_x)_\infty=(E_\infty)|_x$ be a finite dimensional vector space containing $f(x)$. Then $f$ is said to be \textbf{in good position to $N$ at $x$}, if in a local coordinate (still using the same $f$), there is a local $sc^+$ section $t$ with $t(0)=f(0)$, $f-t: (r(U),0)\to (R(U\triangleleft F),0)$ has a filling of the form of the graph of $h: U\to F$ where $U\subset C\subset A$ with a partial quadrant $C$ in an sc-Banach $A$, and $N$ is mapped onto $\tilde N$ in $R(0,F_\infty)\subset F_\infty$, such that the linearization $h'(0)$ at $0$ satisfies:
\begin{enumerate}
\item $\text{span}(\text{im} (h'(0)),\tilde N)=F$,
\item $(h'(0))^{-1}(\tilde N)\cap C$ is open in $(h'(0))^{-1}(\tilde N)$, and 
\item there exists an sc-complement $Z$ of $(h'(0))^{-1}(\tilde N)$ in $A$ and $\epsilon>0$ such that $v\in C$ and $v+z\in C$ for $v\in (h'(0))^{-1}(\tilde N)$, $z\in Z$ with $\|z\|_Z \leq \epsilon \|v\|_{(h'(0))^{-1}(\tilde N)}$ are equivalent.
\end{enumerate}
\end{definition}

\begin{definition}(in general position, \cite{DII}, also c.f. \cite{PolyfoldII} 5.17) Let $f:B\to E$ be an sc-Fredholm section where $B$ may have corners (based on tame retracts in partial quadrants \cite{PFI} 2.34), $x\in B_\infty$, and $N\subset (E|_x)_\infty$ be a finite dimensional vector space containing $f(x)$. Then $f$ is said to be \textbf{in general position to $N$ at $x$}, if in a local coordinate (using the same $f$), there is a local $sc^+$ section $t$ with $t(0)=f(0)$, $f-t: (r(U),0)\to (R(U\triangleleft F),0)$ has a filling of the form of the graph of $h: U\to F$ where $U\subset C\subset A$ with a partial quadrant $C$ in an sc-Banach $A$, and $N$ is mapped onto $\tilde N$ in $R(0,F_\infty)\subset F_\infty$, such that the linearization $h'(0)$ at $0$ satisfies:
\begin{enumerate}
\item $\text{span}(\text{im} (h'(0)),\tilde N)=F$, and
\item $\text{span}((h'(0))^{-1}(\tilde N),\;\bigcap_{\mathcal{H}} T_0\mathcal{H})=T_0 A$, where the intersection is taken over all faces $\mathcal{H}$ of $C$. (Here a face is the closure of a connected component of the space of points with the corner index one, and $T_0\mathcal{H}$ and $T_0A$ can be identified with linear $\mathcal{H}$ and $A$ respectively.)
\end{enumerate}
\end{definition}

\begin{remark} For $f(x)=0$ and $N=\{0\}$, being in good (respectively general) position to $N$ at $x$ is exactly the definition of $f$ being \textbf{in good (respectively general) position at $x$} in \cite{PolyfoldII}. Notice that $h'(0)$ depends on $t$, but if $f$ is in good/general position using a choice of chart, a filling and an $sc^+$ section $t$, then for any other chart and filling, there exists an $sc^+$ section making $f$ in good/general position. If $f(x)=0$, then $(f'(x))^{-1}(N)$ is a well-defined object independent of all choices made. Being in general position implies being in good position, thus we will state results for sc-Fredholm sections in good position so that those results will hold for sc-Fredholm sections in general position as well.
\end{remark}

\begin{definition} (fiber-standard, \cite{DII}) A strong M-polyfold bundle $L$ is \textbf{fiber-standard}, if its local models can be chosen to be strong bundle retracts of the forms $R(U\triangleleft N)$ for some strong bundle retraction $R$ with $N_0$ of $N=(N_m)_{m\in\mathbb{N}_0}$ being finite dimensional (thus $N_m=N_0$ for all $m\in\mathbb{N}_0$).

A subbundle $L\to W$ of a strong bundle $E\to W$ is said to be \textbf{fiber-standard} if there exist local coordinates for $E\to W$ of the form $R(U\triangleleft F)$ for which $L\to W$ is of the form $R(U\triangleleft N)$ where $R|_{U\triangleleft N}: U\triangleleft N\to U\triangleleft N$, $N_0$ of $N$ is finite dimensional and $N_0\subset F_\infty$, such that the induced changes of coordinates between local models of $L$ are sc-diffeomorphisms of (automatically strong) bundle retracts. We often just write the local model as $R(U\times N)$, as the information about strongness is trivial.
\end{definition}

\begin{definition}\label{IGP} ($f\pitchfork L$, $f$ in good position to $L$, \cite{DII}) $f$ is an sc-Fredholm section of a strong bundle $E\to B$, and $L\to W$ is a fiber-standard subbundle of $E|_W$ where $W\subset B$ is open. $f$ is said to be \textbf{in good position to $L$}, or $f\pitchfork L$, if $f$ is in good position to $L_x$ as in \ref{IGPPOINTWISE} for all $x\in W$. Similarly for being in general position.
\end{definition}

\begin{lemma}\label{QUOTIENTTHEORY} (quotient theory, \cite{DII}) If $f: B\to E$ is an sc-Fredholm section of a strong bundle $E$, and $L$ is a fiber-standard subbundle of $E|_W$ for $W$ open in $B$. Then the quotient bundle $(E|_W)/L$ is a strong bundle and the induced $(f|_W)/L$ is sc-Fredholm. $f$ is in good position to $L$ if and only if $(f|_W)/L$ is in good position. Similarly for being in general position.
\end{lemma}

\begin{proposition}\label{PFDFDRED} (\cite{DII}) Let $f: B\to F$ be an sc-Fredholm section in a strong bundle ep-groupoid as in \cite{PolyfoldIII}, and let $x\in f^{-1}(0)$ a solution in the object space and $G_x$ the stabilizer group at $x$. Then there exists a $G_x$-invariant finite dimensional subspace $N_x$ of $(E|_x)_\infty=(E_\infty)|_x$ such that $f$ is in good position to $N_x$. For any such choice of $N_x$, there exists a $G_x$-invariant fiber-standard subbundle $L_x\to W_x$ where $W_x$ is a $G_x$-invariant sc-open neighborhood of $x$ in $B$\footnote{\label{FORMINGABASIS}$B$ is the object space of the base ep-groupoid. This means that the ep-groupoid restricted to $W_x$ becomes a translation groupoid $G_x\times W_x\to W_x$ with the usual action, and such $W_x$'s form a basis around $x$ in the object space $B$.} such that $f$ is in good position to $L_x$ as in \ref{IGP}. By \ref{QUOTIENTTHEORY}, $(f|_{W_x})/L_x: W_x\to (E|_{W_x})/L_x$ is an sc-Fredholm section in a strong bundle ($G_x$-translation) ep-groupoid and in good position, then by \cite{PolyfoldIII}, $V_x:=((f|_{W_x})/L_x)^{-1}(0)=(f|_{W_x})^{-1}(L_x)=\{f\in L_x\}:=\{z\in W_x\;|\; f(z)\in (L_x)|_z\}$ is a $G_x$-invariant manifold with corners in the usual sense and $f|_{V_x}:V_x\to L_x|_{V_x}$ is an invariant \textbf{local finite dimensional reduction} of $f$. We often write $f|_{\{f\in L_x\}}$. The same result holds for $f$ being in general position, which is the version used in \cite{Dingyu}.
\end{proposition}

The following is not the most general, but is sufficient for our purposes.

\begin{definition}\label{SCFREDEMBEDDING} (dimensional reduction, embedding, refinement, \cite{DII}) An sc-Fredholm $f_1: B_1\to E_1$ is a \textbf{dimensional reduction} of another sc-Fredholm $f:B\to E$ as sc-Fredholm sections if
\begin{enumerate}
\item $B_1\subset B$ and $E_1\subset E|_{B_1}$ as sets, and $f_1=f|_{B_1}$;
\item around $z\in B_1$, there exists a strong bundle retract local model $R((U_1\oplus M)\triangleleft (F_1\oplus M))$ for $E$, such that $R|_{(U_1\oplus\{0\})\triangleleft(F_1\oplus\{0\})}$ is a strong bundle retraction and its image $R((U_1\oplus \{0\})\triangleleft (F_1\oplus\{0\}))$ is a local model for $E_1$ and coordinate changes respect this structure; (In particular, $B_1$ is a \textbf{sub-M-polyfold} of $B$ in the sense of \cite{PFI} 2.20 with the local retraction $r=\tilde r\circ \tilde\pi=\tilde\pi\circ \tilde r$, where $\tilde r:=R|_{(U_1\oplus M)\triangleleft (\{0\}\oplus \{0\})}$ and $\tilde \pi: U_1\oplus M\to U_1\oplus M, (u,m)\mapsto (u,0)$ are commuting retractions; and $E_1$ is a strong subbundle of $E|_{B_1}$; and we have a submersion $\pi: U\to B_1$ from an open neighborhood $U$ of $B_1$ in $B$ and a strong bundle map $\Pi: E|_U\to E_1$ which is a fiberwise strong projection and covers $\pi$.) 
\item $(f_1)^{-1}(0)=f^{-1}(0)$ (or $(f_1)^{-1}(0)=f^{-1}(0)\cap \pi^{-1}(B_1)$ if $f_1$ is only a \textbf{local dimensional reduction}); and
\item in a local model $R((U_1\oplus M)\triangleleft(F_1\oplus M))$ around any smooth $x\in B_1$ with $x$ mapping to $0$ and using the same notation for $f$, there exist a local $sc^+$-section $t: U_1\oplus\{0\}\to F_1\oplus\{0\}$ with $t(x)=f_1(x)=f(x)$, and a filling $[h,0]:\mathcal{O}(U_1\oplus M,0)\to (F_1\oplus M, 0)$ for $[f,0]$ which, as a germ at 0, is sc-diffeomorphic in the $M$-direction along $U_1$, such that $[h,0]|_{\mathcal{O}(U_1\oplus\{0\},0)}$ is a filling for $[f_1,0]|_{\mathcal{O}(U_1\oplus\{0\},0)}$ and under some strong bundle sc-diffeomorphism $[h-t,0]|_{\mathcal{O}(U_1\oplus\{0\},0)}$ is of the basic class (which implies that $[h-\tilde\pi^\ast t,0]$ is also of the basic class, where $\tilde\pi$ is defined in (2)).
\end{enumerate}

Note that $f|_{\{f\in E_1\}}=f_1$. An embedding of strong bundle $\Phi:\tilde E\to E$ covering a base embedding $\tilde B\to B$ is said to be an \textbf{embedding} from an sc-Fredholm section $\tilde f:\tilde B\to\tilde E$ to an sc-Fredholm $f:B\to E$ if $\Phi_\ast \tilde f$ is a dimensional reduction of $f$ as above. We now introduce the ep-groupoid version: A strong bundle ep-groupoid functor $\Psi: \hat E\to E$ is called a \textbf{Morita-refinement} if $\Psi$ is a local sc-diffeomorphism between the object spaces and between the morphism spaces, the induced bundle map $\underline{\Psi}$ between the orbit spaces is an sc-homeomorphism, and $\Psi$ induces isomorphisms between stabilizer groups (HWZ call it equivalence, see \cite{PolyfoldIII} 2.5, and a strong bundle version can be defined verbatim; and it is also the categorical equivalence if ep-groupoids are regarded as categories). An \textbf{embedding} between two sc-Fredholm sections in strong bundle ep-groupoids is a strong bundle functor between the strong bundle ep-groupoids such that the map between the object spaces is an embedding as defined above, the induced map between the orbit spaces is an sc-homeomorphism onto the image, and the induced maps between stabilizer groups are isomorphisms. An embedding that is open between the object spaces is called an \textbf{open embedding}. A strong bundle ep-groupoid functor is called a \textbf{chart-refinement} if it is a composition of a Morita-refinement followed by an open embedding (namely, it is a Morita-refinement except that the induced map between the orbit spaces is an open map mapping sc-homeomorphically onto the image). An sc-Fredholm $f:B\to E$ is said to be a \textbf{refinement} of an sc-Fredholm $f_2: B_2\to E_2$, if there exist a Morita-refinement from $f_1$ to $f_2$ and an embedding from $f_1$ to $f$. (Equivalently, $f$ is a \textbf{refinement} of $f_2$ if a Morita-refinement of the restriction $f_2|_U$ of $f_2$ to an invariant open neighborhood $U$ of $(f_2)^{-1}(0)$ in $B_2$ embeds into $f$.)
\end{definition}

\begin{definition}(\cite{DII})\label{PFDREQUIV} Two sc-Fredholm sections are \textbf{R-equivalent} as sc-Fredholm sections, if they have a common refinement. This is an equivalence relation, as we can use projections $\pi$ and $\Pi$ to do the roof construction. Two sc-Fredholm sections are \textbf{chart-equivalent} if they admit a common chart-refinement.
\end{definition}

\begin{remark}\label{KEYREMARKS} We note the following observations:
\begin{enumerate}
\item Polyfold perturbation theory in \cite{PolyfoldII} and \cite{PolyfoldIII} factors through the R-equivalence \ref{PFDREQUIV} as well as the chart-equivalence above.
\item \ref{PFDFDRED} is an example of (local finite) dimensional reduction as in \ref{SCFREDEMBEDDING}.
\item $\pi$ and $\Pi$ have to be included in the data for embedding, because in general there is no implicit function theorem in the polyfold setting to turn a map that is a homemorphism onto image and injective between tangent spaces into a submersion onto the image (but there is in the setting of \ref{PFDFDRED}). The complicated looking definition of embedding in \ref{SCFREDEMBEDDING} is also because $\pi$ and $\Pi$ have to adapt to the sc-Fredholmness notion. See the next item:
\item We motivate definition \ref{SCFREDEMBEDDING} a bit using \ref{PFDFDRED}. Suppose in the setting of \ref{PFDFDRED}, instead we have a classical smooth Fredholm section $f:B\to E$ in a Banach bundle $E\to B$ (Banach manifold $B$ is by definition second countable and paracompact). Suppose $L$ is a subbundle of $E|_W$ ($W$ open in $B$) which is \textbf{fiberwise closed} (namely, for every $y\in W$, $L|_y$ is a closed subspace of Banach $E|_y$; note that in \ref{PFDFDRED} $L|_y$ is finite dimensional, but for the following in the classical set-up, it is not necessary and being closed is enough), so the quotient bundle $E/L\to B$ and the quotient section $f/L: B\to E/L$ are well-defined. If $f/L$ is in good position, then by implicit function theorem, $V:=\{f\in L\}=(f/L)^{-1}(0)$ is a Banach submanifold of $B$; and $f|_{\{f\in L\}}: V\to L|_V$ is a classical Fredholm section which sits in $f|_W$, and has the same zero set and has the same kernels and cokernels at zeros as those of $f|_W$. $f|_{\{f\in L\}}$ is a dimensional reduction of $f$. Since the linearization $f'(x): T_xW\to E|_x$ at zero $x\in (f|_W)^{-1}(0)$ induces an isomorphism $T_xW/T_xV \to (E|_x)/(L|_x)$. By using a choice a Riemannian metric or otherwise, we have an open subset $W'\subset W$ around $(f|_W)^{-1}(0)$ and a submersion $\pi: W'\to V$ with disk-like fibers, and we can choose\footnote{By using a connection and a projection and shrinking $W'$ if necessary.} a smooth bundle map $\tilde\pi:\tilde L \to L|_V$ which is a fiberwise isomorphism and covers $\pi: W'\to V$, such that $(f|_{W'})/\tilde L$ is transverse. A perturbation $s$ to $f|_V$ (such that $f|_V+s:V\to L|_V$ is in good position) can be lifted/extended to a perturbation $\tilde s: W'\to \tilde L|_{W'}$ to $f|_{W'}$ via $y\mapsto \tilde \pi(y,\cdot)^{-1}(s(\pi(y)))$ such that $f|_{W'}+\tilde s$ is in good position automatically. Thus the Fredholm theory of $f|_{W'}$ and $f|_{\{f\in L\}}$ are really the same, and this is the basis that $f|_{\{f\in L\}}$ captures $f|_{W'}$ as far as Fredholm theory is concerned. Data $\Pi: E|_{W'}\to \tilde L$ can also be constructed. Therefore, in classic Fredholm setting, $\pi$ and $\Pi$ can be deduced using arguments rooted in implicit function theorem.
\item The following is a non-trivial example of dimensional reduction which will be used later:
\end{enumerate}
\end{remark}

\begin{lemma}\label{SUMMABILITYTRICK} (\cite{Dingyu}, \cite{DII}) If $f:B\to E$ is an sc-Fredholm section of a strong bundle ep-groupoid $\pi_E: E\to B$, then $$f^{(n)}: E^{\oplus n}\to ((\pi_E)^{\oplus n})^\ast(E\oplus E^{\oplus n}),$$ $$(x, v_1,\cdots, v_n)\mapsto ((x, v_1, \cdots, v_n), (f(x), v_1, \cdots, v_n))$$ is also an sc-Fredholm section. $f$ is a dimensional reduction of $f^{(n)}$ as in \ref{SCFREDEMBEDDING}, thus $f^{(n)}$ refines and hence is R-equivalent to $f$. Here $E^{\oplus n}:=E\oplus \cdots\oplus E$ is the Whitney sum of $n$-tuple copies of $E$ over $B$, so $(v_1,\cdots, v_n)$ above satisfies $\pi_E(v_1)=\cdots=\pi_E(v_n)$ in $B$. 
\end{lemma}

\begin{example}\label{DIMREDEX} Let $f: B\to E$ be an sc-Fredholm section with a compact zero set $f^{-1}(0)$, where $B$ is without boundary. We can choose finitely many local $sc^+$-sections $s_i, i=1,\cdots, N$ such that $f(y)+\sum_{i=1}^N a_i s_i(y):B\times B^{\R^N}_\epsilon(0)\to pr_1^\ast E$ is in good position (same as transverse here). Denote $\mathfrak{f}: B\times B^{\R^N}_\epsilon(0)\to pr_1^\ast E\oplus pr_2^\ast \R^N, (y,a_i)\mapsto (f(y)+\sum_{i=1}^N a_i s_i(y), a_i)$. Both $f=\mathfrak{f}|_{\{\mathfrak{f}\in pr_1^\ast E\}}$ and $$pr_2|_{\{(y,a_i)\;|\;f(y)+\sum_{i=1}^N a_i s_i(y)=0\}}=\mathfrak{f}|_{\{\mathfrak{f}\in pr_2^\ast \R^N\}}$$ embed into $\mathfrak{f}$ (with submersions and projections already given or easily constructible), and thus are R-equivalent. A generic point $(a_i)$ in $B^{\R^N}_\epsilon(0)\subset \R^N$ will give a transverse perturbation $f+\sum_{i=1}^N a_i s_i$ of $f$, and classical analogue of this is just the derivation of Sard-Smale theorem from Sard theorem.
\end{example}

Note that for the group action version of \ref{DIMREDEX}, $\mathfrak{f}$ will become a multisection.

If not using local $sc^+$ sections constructed usually as constant sections in some trivializations multiplied with bump functions as in \ref{DIMREDEX}, we can achieve a local \textbf{thickening}, or \textbf{stabilization}, or \textbf{continuous family of perturbation} (which is invariant and single-valued, but only local):

\begin{proposition}\label{STABILIZATION} (\cite{DII} Let $f: B\to E$ be an sc-Fredholm section in a strong bundle ep-groupoid. $L_x$ is an invariant fiber-standard subbundle of $E|_{W_x}$ (e.g. constructed in \ref{PFDFDRED}) to which $f$ is in good position as in \ref{IGP}. By shrinking $W_x$, we can assume $L_x$ is trivial (the method of the construction in \ref{PFDFDRED} will produce a trivial $L_x$ anyway), so $\psi_x:W_x\times N_x\overset{\cong}{\to} L_x$ for some finite dimensional $G_x$-invariant vector subspace $N_x$ in $(E|_x)_\infty=(E_\infty)|_x$. Then $\mathbf{f}^{L_x}: W_x\times N_x\to pr_1^\ast (E|_{W_x}), (y,v)\mapsto f(y)+\psi_x(y,v)$ is sc-Fredholm and in good position. $pr_2|_{(\mathbf{f}^{L_x})^{-1}(0)}=\mathfrak{f}^{L_x}|_{\{\mathfrak{f}^{L_x}\in pr_2^\ast N_x\}}$ is a finite dimensional reduction of $\mathfrak{f}^{L_x}:=(\mathbf{f}^{L_x},pr_2): W_x\times N_x\to pr_1^\ast (E|_{W_x})\oplus pr_2^\ast N_x, (y,v)\mapsto (\mathbf{f}^{L_x}(y,v),v)$ (here $\mathfrak{f}^{L_x}$ refines $f$).
\end{proposition}

The above two pictures in \ref{PFDFDRED} and \ref{STABILIZATION} are fairly standard for classical Fredholm sections with finite dimensional $L$ or for concrete $L$, although the relationship between these two pictures explained in proposition \ref{SAMEFDRED} (1) below is not usually pointed out.

\begin{corollary}\label{SAMEFDRED} (\cite{DII})
\begin{enumerate}
\item For the setting in \ref{STABILIZATION}, if $f|_{W_x}$ is viewed as $\mathfrak{f}^{L_x}|_{W_x\times \{0\}}:W_x\times\{0\}\to pr_1^\ast (E|_{W_x})$ and $L_x$ is viewed as $pr_1^\ast L_x$, $pr_2|_{(\mathbf{f}^{L_x})^{-1}(0)}$ agrees with another finite dimensional reduction $f|_{\{f\in L_x\}}$ (of $f=\mathfrak{f}^{L_x}|_{\{\mathfrak{f}^{L_x}\in pr_1^\ast (E|_{W_x})\}}$, thus also a finite dimensional reduction of $\mathfrak{f}^{L_x}$). (In particular, not only $f|_{\{f\in L_x\}}$ and $pr_2|_{(\mathbf{f}^{L_x})^{-1}(0)}$ are R-equivalent, but they actually are the same under this implicit identification.)
\item Let $L_x\to W_x$ and trivial $L'_x\to W'_x$ be any two such $G_x$-invariant fiber-standard subbundles to which $f$ is in good position. Choose a $G$-invariant $W''_x\subset W_x\cap W'_x$. Then $f|_{\{f\in L|_{W''_x}\}}$ is R-equivalent to $pr_2|_{(\mathbf{f}^{L'}|_{W''_x})^{-1}(0)}$ as $G_x$-equivariant (sc-Fredholm) sections.
\end{enumerate}
\end{corollary}

%{Let $f:B\to E$ be a $G$-equivariant Fredholm section with finite $G$ (or a Fredholm section in Banach bundle \'etale proper Lie groupod with $G=G_x$, as the situation is local), $x\in f^{-}(0)$ and $L=\psi(W\times N)$ a $G$-invariant subbundle chosen as in \ref{SUBBUNDLE} such that $f\pitchfork L$ and $\psi: W\times N\to L$ an equivariant bundle diffeomorphism (a chart). Then $\mathbf{f}^L: W\times N\to pr_1^\ast E|_W, (y,v)\mapsto f(y)+\psi(y,v)$. Then $pr_2|_{(\mathbf{f}^L)^{-}(0)}$ is R-equivalent to $f$ (as equivariant Fredholm sections).}

%{If $f:B\to E$ is a $G$-equivariant section in a Banach bundle with $G$ finite (or a section in a Banach bundle \'etale proper Lie groupoid). For the same choice of the invariant closed trivial subbundle $L$ (for example as in \ref{SUBBUNDLE} if $f$ is Fredholm in the usual sense), and view $f:W\to E|_W$ as a part/restriction of $\mathbf{f}^{L}: W\times N\to pr_1^\ast (E|_W)$ (namely $f:W\times\{0\}\to (pr^\ast E|_W)_{W\times\{0\}}$, we often implicit assume so without $pr_1^\ast$). Then following two ways of dimensional reductions (are Fredholm if $f$ is, and) are exactly the same as equivariant sections (\textcolor{red}{or just isomorphic}): $f|_{\{f\in pr_1^\ast L\}}: \{f\in pr_1^\ast L\}\to L|_{\{f\in pr_1^\ast L\}}$, and $pr_2|_{(\mathbf{f}^L)^{-1}(0)}: (\mathbf{f}^L)^{-1}(0)\to L|_{(\mathbf{f}^L)^{-1}(0)}$.(In particular, the base manifolds also agree.)}

Observe the useful fact: 

\begin{lemma}\label{INDDOESNOTM} (\cite{DII})  If $f:B\to E$ is an sc-Fredholm section of a strong bundle ep-groupoid $\pi_E: E\to B$, and suppose that invariant fiber-standard trivial subbundles $L_i=\psi_i(W_i\times N_i), 1\leq i \leq n$ are chosen around $x$ such that $f$ is in good position to $L_i\to W_i$ for all $1\leq i\leq n$ (e.g. arisen as in \ref{PFDFDRED}), and we do not assume any relationship among $L_i$ at all (they can be independent, or dependent over some region in an uncontrolled way, etc), then $$f^{\oplus_{i=1}^n N_i}: (\cap_{i=1}^n W_i)\times (\oplus_{i=1}^n N_i)\to E|_{\cap_{i=1}^n W_i},$$ $$(y,v_1,\cdots, v_n)\mapsto f(y)+\psi_1(y,v_1)+\cdots+\psi_n(y,v_n)$$ is R-equivalent to $f|_{\cap_{i=1}^n W_i}$.
\end{lemma}

\begin{remark} \ref{INDDOESNOTM} is directly identified with a finite dimensional reduction of \ref{SUMMABILITYTRICK} relative to the direct sum of then independent corresponding subbundles $\mathbf{L}_i$ defined exactly as $\mathbf{L}_{x_i}$ in construction \ref{FORGETFULCONSTRUCTION} (ii) (separated to be independent using the extra room in the added direction in $f^{(n)}$ in \ref{SUMMABILITYTRICK}), after staring. See remark \ref{AMAZINGLYTHESAME}. \ref{INDDOESNOTM} can be applied to various intersection in a cover; and it abstractly underpins the observation that using stabilization, independent $L_i$'s and general $L_i$'s will give rise to R-equivalent (sc-)Fredholm problems. 
\end{remark}

During a discussion with Shaofeng Wang \cite{SHAOFENG}, I found a nice application of the above circle of ideas to simplify one of his arguments in further developing the theory of Lu-Tian's virtual Euler class \cite{LuTian}. I phrase $f$ as sc-Fredholm below although it was originally classical Fredholm in his setting (in that setting, the proof below simplifies significantly, but below also shows how to work with hybrid of sc-Fredholmness and classical Fredholmness without changing definitions, and it is not needed later).

\begin{corollary} (perturbing using infinitely many sections) If $f:B\to E$ is an sc-Fredholm section without boundary and group action for simplicity, and $f^{-1}(0)$ is not compact and suppose that there is a countable locally finite collection of compactly supported $sc^+$-sections $s_i, i\in\mathbb{N}$ such that $$\text{span}(\text{im}(f'(z)),s_i(z), i\in\mathbb{N})=E|_z$$ for all $z\in f^{-1}(0)$ (note that only finitely many $s_i(z)$ are non-zero by assumption, but the number of non-zero $s_i(z)$ as a function of $z$ is not bounded in general). Define $l^2:=\{(y_i)_{i\in\mathbb{N}}\;|\;y_i\in\R,\;\sum_i|y_i|^2<\infty\}$, and $\mathbf{f}: B\times l^2\to E$, $(x,(y_i)_{i\in\mathbb{N}})\mapsto f(x)+\sum_{i\in\mathbb{N}} y_i s_i(x)$. Then $\mathbf{f}^{-1}(0)\cap (B\times B^{l^2}_\epsilon(0))$ is a Banach manifold where $B^{l^2}_\epsilon(0)$ is a ball centered at $0$ of radius of $\epsilon$ in $l^2$ for some $\epsilon>0$, and $pr_2: \mathbf{f}^{-1}(0)\cap (B\times B^{l^2}_\epsilon(0))\to B^{l_2}_\epsilon(0)\subset l^2$ is classical Fredholm, and a regular value $(y_i)_{i\in\mathbb{N}}$ in $B^{l^2}_\epsilon(0)$ corresponds to a transverse perturbed section $f+\sum_iy_i s_i$.
\end{corollary}
\begin{proof}
Because of local finiteness, for any $x\in B$, there exists an open $U_x$ with $x\in U_x\subset B$ and $S_x\subset \mathbb{N}$, such that $s_i|_{U_x}=0$ for $i\in S_x$ and $S^c_x:=\mathbb{N}\backslash S_x$ is finite. Let $l^2_{S_x}:=\{(y_i)_{i\in\mathbb{N}}\in l^2\;|\; y_i=0, i\in S_x^c\}$ and $l^2_{S_x^c}:=l^2\backslash l^2_{S_x}$ with a general element of the latter denoted by $(y_i)_{S^c_x}\in l^2_{S_x^c}$. Note that $\mathbf{f}^{-1}(0)$ decouples into $l^2_{S_x}\times (\mathbf{f}_{S_x^c})^{-1}(0)$ for transverse sc-Fredholm $\mathbf{f}_{S_x^c}:(x, (y_i)_{S^c_x})\mapsto f(x)+\sum_{i\in S_x^c}y_i s_i(x)$, and those local Banach manifolds patch up into a classical Banach manifold $\mathbf{f}^{-1}(0)\cap (B\times B^{l^2}_\epsilon(0))$. We now use the fact that R-equivalence preserves (sc-)Fredholmness. Locally, $pr_2: \mathbf{f}^{-1}(0)\cap (B\times B^{l^2}_\epsilon(0))\to B^{l^2}_\epsilon(0)$ is R-equivalent classically to $pr_2:(\mathbf{f}_{S^c_x})^{-1}(0)\cap (B\times B^{l^2}_\epsilon(0))\to l^2_{S^c_x}\cap B^{l^2}_\epsilon(0)$, which is R-equivalent to $(\mathbf{f}_{S^c_x},pr_2):(x,(y_i)_{S^c_x})\mapsto (\mathbf{f}_{S^c_x}(x,(y_i)_{S^c_x}),(y_i)_{S^c_x})$, which is R-equivalent to $f|_{U_x}$, which is sc-Fredholm. Therefore, $pr_2:\mathbf{f}^{-1}(0)\cap (B\times B^{l^2}_\epsilon(0))\to B^{l^2}_\epsilon(0)$ is a classical Fredholm map between Banach manifolds and one can apply Sard-Smale (namely, the classical analogue of \ref{DIMREDEX} obtained from \ref{DIMREDEX} by applying the remark before conventions at the start of this subsection).
\end{proof}

We never need to use infinitely many local perturbations for sc-Fredholm section with compact zero set in the orbit space, but we can use them for non-compact but locally compact moduli space and the above theorem applies. It becomes particularly interesting when the noncompact zero set in the orbit space forms a pseudocycle, see Wang's forthcoming work.

\subsection{Kuranishi structures}\label{KTHEORY}

Among the finite dimensional reduction approaches, I will first focus on \textbf{Fukaya-Oh-Ohta-Ono's Kuranishi structures} in \cite{FO}, \cite{FOOO}, \cite{FOOONEW}, \cite{FOOONEWER} and \cite{FOOOCF} side by side with some of my interpretations in \cite{DY}, \cite{DI} and \cite{MAXCAT}, in particular \ref{FOOOISDIR}, \ref{GCSWOIND}, \ref{LEVEL1CATEXIST} and \ref{REQUIVISEQUIV} (both viewpoints work and complement well with each other), because 
\begin{enumerate}[(i)]
\item They are ones I considered for the polyfold--Kuranishi correspondence in my thesis \cite{Dingyu};
\item They are more relevant to isomorphism to Kuranishi homology \cite{KH}, the latter of which is written by Joyce starting from good coordinate systems with certain data and properties obtained from FOOO's Kuranishi structures `replacing domain simplicies in a singular chain';
\item They have already been used in moduli spaces with corners featuring finer structures (a countable family of Kuranishi structures on compact spaces with corners exhausting the entire moduli space with the boundary of each Kuranishi structure being a union of fiber products of Kuranishi structures appeared earlier in the family ordered by energy and combinatorial data\footnote{We often summarize it by `filtration by compact Kuranishi structures with fiber product boundaries'.}, cyclic symmetry, and equivariance under compact Lie group action); and
\item They are relevant for one to ultimately construct a master theory commonly extending HWZ-FHWZ's SFT construction in \cite{PolyfoldGW} and \cite{FHWZ} and FOOO's Lagrangian Floer theory in \cite{FOOO}. One possible scenario for this is chain-level relative SFT in full generality of Cieliebak-Latschev-Mohnke's program in \cite{CLM} together with Irie's construction \cite{Irie} of chain-level string topology\footnote{Irie's chain level theory of string topology uses de Rham chain theory and is closely related to CF-perturbation \cite{FOOOCF} and level-1 structure \cite{DI} of Kuranishi structures.}, and another is a degeneration formula of Fukaya category degenerating along a stable Hamiltonian structure.
\end{enumerate}

The notion of a Kuranishi structure was invented by Fukaya and Ono in \cite{FO}, and improved further in FOOO's \cite{FOOO}. A Kuranishi structure roughly speaking is a global geometric structure made up of a coherent organization of local finite dimensional reductions (allowing dimensions to vary among reductions) from an underlying Fredholm problem with analytic limiting phenomena and domain variation. The idea is that one refrains from directly describing and perturbing the Fredholm problem (thus no need to face various subtleties such as nowhere differentiable coordinate changes and gluing in the ambient space), and one instead remembers a Kuranishi structure data out of it to globally capture the essential Fredholm information. One can then precisely describe and perturb the chosen Kuranishi structure using topology and finite dimensional tools, where one faces different kinds of subtleties such as varying dimension among reductions and non-canonical nature of reductions. There are preprints in the literature contributing to the extended discussions on details and subtleties of the theory of Kuranishi structures chronologically in \cite{KH}, \cite{DY}, \cite{FGG}, \cite{MW}, \cite{FOOONEW}, \cite{DI}, \cite{FOOONEWER}, \cite{FOOOCF}. There are also subsequent papers following up \cite{KH} and \cite{MW} before \cite{FOOONEWER} and \cite{FOOOCF} respectively, but since they are not strictly Kuranishi structures in FOOO's sense, they will be introduced in the relevant subsections later.

To be concrete and better aid the discussions in this subsection, I will first briefly introduce the notions (mostly notations) of a Kuranishi structure, a good coordinate system, and a map between Kuranishi structures (introduced in the video talk \cite{DKSG} and not a morphism in the category yet, the latter of which will be discussed in \cite{MAXCAT} following \cite{DKSG}).

\begin{definition}\label{KURANISHISTRUCTURE}(Kuranishi structure \cite{FO}, \cite{FOOO}) Let $X$ be a compact metrizable topological space. A \textbf{Kuranishi structure} on $X$ is the tuple $$\mathcal{K}_X:=(X,\{C_p\}_{p\in X},\{C_q\to C_p\}_{q\in X_p}),\;\;\text{where}$$
\begin{enumerate}[(1)]
\item for every $p\in X$, a \textbf{Kuranishi chart} $C_p=(G_p,s_p,\psi_p)$ consists of a section $s_p: V_p\to E_p$ in a finite dimensional vector bundle $E_p$\footnote{The non-triviality of $E_p$ is just for an easier notation and sometimes better clarity.} with smooth effective action by a finite group $G_p$ (with $\underline{\;\cdot\;}$ denoting its quotient of an object or a map), under which $s_p$ is equivariant, and a homeomorphism $\psi_p$ from $\underline{s_p^{-1}(0)}$ onto the image $X_p:=\psi_p(\underline{s_p^{-1}(0)})$ such that $X_p$ is an open neighborhood of $p$ in $X$, and
\item for every $q\in X_p$, a \textbf{coordinate chang}e $(C_q\to C_p):=(\phi_{pq},\hat\phi_{pq},V_{pq})$, consisting of a $G_q$-invariant subset $V_{pq}$ of $V_q$, called the \textbf{domain of the coordinate change}, such that $q\in X_q|_{V_{pq}}:=\psi_q(\underline{(s_q|_{V_{pq}})^{-1}(0)})$, and a $G_q$-$G_p$ equivariant bundle embedding \footnote{A $G_q$-$G_p$ equivariant embedding, or simply an \textbf{equivariant embedding}, is an embedding which induces an injective map between quotient spaces and induces isomorphisms between stabilizer groups of the domain points and those of their images in the target. An equivariant bundle embedding is a bundle map which is an equivariant embedding.} $\hat\phi_{pq}: E_q|_{V_{pq}}\to E_p$ covering an equivariant base embedding $\phi_{pq}:V_{pq}\to V_p$ such that $$\hat\phi_{pq}\circ (s_q|_{V_{pq}})=s_p\circ\phi_{pq}\;\;\text{and}\;\; \psi_p\circ\underline{\phi_{pq}|_{(s_q|_{V_{pq}})^{-1}(0)}}=\psi_q|_{\underline{(s_q|_{V_{pq}})^{-1}(0)}},$$ satisfying the \textbf{tangent bundle condition} (clearest if stated pointwise): for every $z\in (s_p|_{\phi_{pq}(V_{pq})})^{-1}(0)$, the linearization $s_p'(z): T_zV_p\to (E_p)|_z$ induces an isomorphism $T_zV_p/T_z(\phi_{pq}(V_{pq}))\to (E_p)|_z/(\hat\phi_{pq}(E_{pq}))|_z$,
\end{enumerate}
such that coordinate changes are compatible up to $G_p$-actions (on $C_p$'s) as maps from the common domains $V_{qr}\cap V_{pr}\cap \phi_{qr}^{-1}(V_{pq})$, $r\in X_q\cap X_p, q\in X_p$ of coordinate changes (namely, on the common domain, the diagram commutes up to $G_p$ action on $C_p$):\\
\begin{center}
\hspace{0 cm}
\begin{tikzpicture}[scale=1]
\begin{scope}
\matrix(m)[matrix of math nodes, row sep=2.5em, column sep=2.5 em,
text height=1 ex, text depth=0.25ex]{
C_r & C_q \\
&   C_p\\
};

\path[thick, >= angle 60, ->]
(m-1-1)  edge node [left=0 cm, above=0cm]{$\hat\phi_{qr}$} (m-1-2)
(m-1-2)  edge node [left=-.4 cm, above=-0.3cm]{$\hat\phi_{pq}$} (m-2-2)
(m-1-1)  edge node [left=0.1 cm, above=-0.6cm]{$\hat\phi_{pr}$} (m-2-2);
\end{scope}
\end{tikzpicture}
\end{center}
\end{definition}

\begin{definition}\label{GCS} (good coordinate system \cite{FO}, \cite{FOOO}, also see \cite{DI}) Let $X$ be compact metrizable. A \textbf{good coordinate system} on $X$ is the tuple
$$\mathcal{G}_X:=(X, (S,\leq), \{C_x\}_{x\in S}, \{C_y\to C_x\}_{y\leq x, X_y\cap X_x\not=\emptyset}),$$
\begin{enumerate}[(1)]
\item a finite index set $S\subset X$ with a partial order $\leq$ on it (a total order without antisymmetry),
\item for $x\in S$, a Kuranishi chart $C_x=(G_p,s_p: U_x\to E_x,\psi_p)$ with $G_p$ action and covering $X_x$ as in \ref{KURANISHISTRUCTURE} such that $X=\bigcup_{x\in S} X_x$, and
\item for $y,x\in S$ such that $y\leq x$ and $X_y\cap X_x\not=\emptyset$, a coordinate change $(C_y\to C_x)=(\phi_{xy},\hat\phi_{xy},U_{xy})$ just in \ref{KURANISHISTRUCTURE} except instead of requiring $y\in X_x|_{U_{xy}}$, we require that $X_y|_{U_{xy}}:=\psi_y(\underline{(s_y|_{U_{xy}})^{-1}(0)})=X_y\cap X_x$,
\end{enumerate}
such that coordinate changes are compatible up to $G_x$-actions as maps from the common domains of coordinate changes.
\end{definition}

\begin{definition}\label{THEMAP} (map between Kuranishi structures \cite{DKSG}, \cite{MAXCAT}) A \textbf{map} (not a morphism to be defined in \cite{MAXCAT}) from a Kuranishi structure $$\mathcal{K}_X:=(X,\{C_p\}_{p\in X},\{C_q\to C_p\}_{q\in X_p})$$ to another Kuranishi structure $$\tilde{\mathcal{K}}_Y:=(Y,\{\tilde C_w\}_{w\in Y},\{\tilde C_z\to \tilde C_w\}_{z\in Y_w})$$ is a tuple $(\phi,\{C_p\to \tilde C_{\phi(p)}\}_{p\in X})$ such that
\begin{enumerate}[(1)]
\item $\phi: X\to Y$ is continuous and $\phi(X_p)\subset Y_{\phi(p)}$, and
\item for every $p\in X$, a \textbf{chart map} $(C_p\to \tilde C_{\phi(p)})=(\phi_p,\hat\phi_p, V_p)$, which is a $G_p$-$\tilde G_{\phi(p)}$-equivariant bundle map $\hat\phi_p: E_p\to \tilde E_{\phi(p)}$ covering $\phi_p: V_p\to \tilde V_{\phi(p)}$ satisfying $\tilde s_{\phi(p)}\circ\phi_p=\hat\phi_p\circ s_p$ and $\tilde\psi_{\phi(p)}\circ\underline{\phi_p}|_{\underline{s_p^{-1}(0)}}=\phi\circ\psi_{p}$ without the tangent bundle condition requirement, here $V_p$ is called the \textbf{domain of chart map} which is the entire $V_p$ in the Kuranishi chart $C_p$,
\end{enumerate}
such that the square

\begin{center}
\begin{tikzpicture}[scale=1]
%\hspace{4.5 cm}
\matrix(m)[matrix of math nodes, row sep=3.5 em, column sep=4.5 em,
text height=1 ex, text depth=0.25ex]{
C_q &  \tilde C_{\phi(q)} \\
C_p &   \tilde C_{\phi(p)}\\
};

\path[thick, >= angle 60, ->]
(m-1-1)  edge node [left=0 cm, above=0 cm]{$\hat\phi_{q}$} (m-1-2)
(m-1-1)  edge node [left=0.35 cm, above=-.3 cm]{$\hat\phi_{pq}$} (m-2-1)
(m-2-1)  edge node [left=0 cm, above=0 cm]{$\hat\phi_p$} (m-2-2)
(m-1-2)  edge node [left=-0.8 cm, above=-.3 cm]{$\widehat{\tilde\phi}_{\phi(p)\phi(q)}$} (m-2-2);
\end{tikzpicture}
\end{center}

is commutative up to $\tilde G_{\phi(p)}$ action (as maps from the common domain $V_{pq}$ of definitions) for all $q\in X_p$ (so $\phi(q)\in \phi(X_p)\subset Y_{\phi(p)}$ thus the coordinate change on the right column exists.). We will denote such a map by $\mathcal{K}_X\to \tilde{\mathcal{K}}_Y$ or $\Phi=(\phi,\{C_p\to\tilde C_{\phi(p)}\}_{p\in X})=(\phi, \{(\phi_p,\hat\phi_p, V_p)\}_{p\in X})$.
\end{definition}

\begin{example}\label{EXAMPLEMAP} One of the upshots of the concept of a map is that, chart-refinements, embeddings in \cite{DI}, and strongly smooth maps in \cite{FO} and \cite{FOOO} are all examples of it.
\begin{enumerate}[(1)]
\item (embedding, \cite{DI}) An \textbf{embedding} between Kuranishi structures is a map $\mathcal{K}_X\to \tilde{\mathcal{K}}_Y$ where $Y=X$, $\phi=Id$, $\tilde G_p=G_p$ for all $p\in X$, and all chart maps are chart embeddings. Here, a \textbf{chart embedding} is a chart map whose $\hat\phi_p$ is also a $G_p$-$G_p$-equivariant embedding satisfying the tangent bundle condition analogous\footnote{Namely, for every $z\in (\tilde s_p|_{\phi_p(V_p)})^{-1}(0)$, the linearization $\tilde s_p'(z): T_z\tilde V_p\to (\tilde E_p)|_z$ induces an isomorphism $T_z\tilde V_p/T_z(\phi_p(V_p))\to (\tilde E_p)|_z/(\hat\phi_p(E_p))|_z$.} to that in \ref{KURANISHISTRUCTURE} (2). We denote it by $\mathcal{K}_X\Rightarrow\tilde{\mathcal{K}}_X$.
\item (chart-refinement, \cite{DI}) A \textbf{chart-refinement} is an embedding where all the chart maps $\hat\phi_p$ are open. We denote it by $\mathcal{K}_X\overset{\sim}{\to}\tilde{\mathcal{K}}_X$, and $\mathcal{K}_X$ is also said to be a chart-refinement of $\tilde{\mathcal{K}}_X$. 

If two Kuranishi structures have a common chart-refinement, they are said to be \textbf{chart-equivalent}. This is obviously an equivalence relation.
\item (shrinking, restriction, \cite{DI}) If $\mathcal{K}_X\overset{\sim}{\to}\tilde{\mathcal{K}}_X$ is a chart-refinement for which all the $\hat\phi_p$ are inclusions, we say $\mathcal{K}_X$ is a \textbf{shrinking} or \textbf{restriction} of $\tilde{\mathcal{K}}_X$. 

Given a Kuranishi structure $\mathcal{K}_X$, choose a fixed $G_p$-invariant subset $V'_p$ of $V_p$ for each $p\in X$ such that $p\in X_p|_{V'_p}:=\psi_p(\underline{(s_p|_{V'_p})^{-1}(0)})$, and one can define $V'_{pq}:=V'_q\cap (\phi_{pq})^{-1}(V'_p)\subset V_{pq}$\footnote{Thus, if $q\in X_p|_{V'_p}$, then $q\in X_q|_{V'_{pq}}$, one of conditions in the definition of a Kuranishi structure.} then $$(X,\{C_p|_{V'_p}\}_{p\in X},\{(C_q|_{V'_q}\to C_p|_{V'_p}):=(\phi_{pq}|_{V'_{pq}},\hat\phi_{pq}|_{(E_q|_{V'_{pq}})},V'_{pq})\}_{q\in X_p|_{V'_p}})$$ is also a Kuranishi structure and is a shrinking of $\mathcal{K}_X$, and we also write $\mathcal{K}_X|_{\{V'_p\}_{p\in X}}$.
\item (strongly smooth, \cite{FO} and \cite{FOOO}) If $\tilde{\mathcal{K}}_Y$ has $\tilde E_w=0$ for all $w\in Y$, then $\tilde s_w=0$ and $\widehat{\tilde\phi}_{wz}=\tilde\phi_{wz}$, thus $\tilde{\mathcal{K}}_Y$ is just a representative of an orbifold. If $\tilde G_w=\{Id\}$ for all $w\in Y$, then $\tilde{\mathcal{K}}_Y$ is just a manifold. For those two cases, data $\Phi:\mathcal{K}_X\to\tilde{\mathcal{K}}_Y$ is called a \textbf{strongly smooth} map in \cite{FO} and \cite{FOOO}.
\item We have analogous notions and examples of \textbf{map}s with various extra conditions like (1) - (4) above for good coordinate systems (for maps between Kuranishi structures/good coordinate systems, and for level-1 \ref{LEVELONECC} embeddings, respectively) as well and we omit them here (see \cite{DI} and \cite{MAXCAT}).
\end{enumerate}
\end{example}

\begin{definition}\label{KURANISHIREFINEMENT} (refinement, \cite{DI}) A Kuranishi structure $\mathcal{K}'_X$ is said to \textbf{refine} or be a \textbf{refinement} of $\mathcal{K}_X$ if there exists a diagram $\mathcal{K}_X\overset{\sim}{\leftarrow}\mathcal{K}''_X\Rightarrow\mathcal{K}'_X$.
\end{definition}
\begin{remark}((trivially) equivalent formulation of \ref{KURANISHIREFINEMENT}) Suppose we have a refinement $\mathcal{K}_X\overset{\sim}{\leftarrow}\mathcal{K}''_X\Rightarrow\mathcal{K}'_X$, we can define a shrinking of Kuranishi structure $\mathcal{K}_X$ using the chart-refinement $(\phi^1,\{(\phi^1_p,\hat\phi^1_p, V''_p)\}_{p\in X}):\mathcal{K}''_X\overset{\sim}{\to}\mathcal{K}_X$ as $\tilde{\mathcal{K}}_X:=\mathcal{K}_X|_{\{\phi^1_p(V''_p)\}_{p\in X}}$, see \ref{EXAMPLEMAP} (3). Then the inclusion $\tilde{\mathcal{K}}_X$ into $\mathcal{K}_X$ is a chart-refinement (a shrinking), and the chart-refinement $\mathcal{K}''_X$ into $\mathcal{K}_X$ has the image $\tilde{\mathcal{K}}_X$, so it can be inverted, and we have that $\tilde{\mathcal{K}}_X\overset{\sim}{\to}\mathcal{K}''_X\Rightarrow\mathcal{K}'_X$ which composes into an embedding. Thus, we can always turn a refinement into this special form $\mathcal{K}_X\overset{\sim}{\leftarrow}\mathcal{K}_X|_{\{\phi^1_p(V''_p)\}_{p\in X}}\Rightarrow\mathcal{K}'_X$, where we restrict and embed. Thus, an equivalent (and possibly more intuitive) formulation of a refinement is that $\mathcal{K}'_X$ is a \textbf{refinement} of $\mathcal{K}_X$ if $\mathcal{K}_X|_{\{\tilde{V}_p)\}_{p\in X}}\Rightarrow\mathcal{K}'_X$, that is, a shrinking of $\mathcal{K}_X$ embeds into $\mathcal{K}'_X$. We choose \ref{KURANISHIREFINEMENT} in \cite{DI} instead, so that we do not need to convert to this particular form after each step (e.g. in diagram chasing), and also a chart-refinement will just be a refinement with the `$\Rightarrow$' part being identity\footnote{We cannot define a chart-refinement as a shrinking (too restrictive). So in the same vein to the reformulation of refinement, $\mathcal{K}'_X$ is a chart-refinement of $\mathcal{K}_X$ if a shrinking of $\mathcal{K}_X$ openly embeds into $\mathcal{K}'_X$.}.
\end{remark}

\begin{definition}\label{REQUIV}(R-equivalence \cite{DI}, c.f. \ref{SCFREDEMBEDDING}\footnote{Each formulation is most natural in its own setting.}) Two Kuranishi structures are said to be \textbf{R-equivalent} if they have a common refinement. 

This being an \textbf{equivalence relation} for Kuranishi structures with a version of maximality and topological matching conditions (see \cite{DY}, and c.f. \ref{MAXMATCHING} (1) and (3)) is announced with a strategy in \cite{DY} and established in \cite{DI} 11.4. R stands for refinement; and in \cite{DY},\cite{Dingyu}, \cite{DI}, we just say equivalence, but we choose to call it R-equivalence in this article to avoid possible confusion with chart-equivalence defined in \ref{EXAMPLEMAP} (2). An equivalence class is called an \textbf{R-equivalence class} or an \textbf{R-germ} for short. There is an analogous notion of \textbf{refinement} for good coordinate systems (see \cite{DI} 5.7, 5.10 and 5.15), from which we can define \textbf{R-equivalence} for good coordinate systems via admitting a common refinement; and this R-equivalence for good coordinate systems (with the corresponding maximality and topological matching conditions, see \cite{DI} 11.1 for a proof, and also in general via \ref{FOOOISDIR}) is also an \textbf{equivalence relation}. (However, we observe that the chart-equivalence for good coordinate systems is not an equivalence relation unlike the case for Kuranishi structures.) Similarly, there are notions of refinement and R-equivalence (as an equivalence relation) for maps (and respectively for level-1 \ref{LEVELONECC} embeddings) between Kuranishi structures/good coordinate systems.
\end{definition}

We say a few words on the motivations of introducing refinement and R-equivalence. Since one of my first goals is to have a well-defined map from HWZ's sc-Fredholm sections into FOOO's Kuranishi structures, one needs to account for the ambiguity of the resulting global structure due to non-canonical nature of local finite dimensional reductions and different choices made in the construction. I came up with the notion of R-equivalence (first announced in \cite{DY})). Then the construction from sc-Fredholm structures to Kuranishi structure R-equivalence classes becomes a well-defined map. The tough part is to show R-equivalence is an equivalence relation (transitivity, because some data for doing roof construction is too cumbersome to be directly included in the definition of a Kuranishi structure). Of course, one could have defined two Kuranishi structures $\mathcal{K}_X$ and $\mathcal{K}'_X$ are `equivalent' if there are a finite sequence of Kuranishi structures $\mathcal{K}_X^1:=\mathcal{K}_X, \mathcal{K}^2_X, \cdots, \mathcal{K}^N_X:=\mathcal{K}'_X$ such that $\mathcal{K}^i_X$ and $\mathcal{K}^{i+1}_X$ admit a common refinement $\tilde{\mathcal{K}}^i_X$ for $1\leq i\leq N-1$, namely, a zigzag which makes transitivity a non-issue. One can then note that being R-equivalent (and thus this temporary notion of `equivalence' via zigzag) implies being Kuranishi cobordant via a Kuranishi structure on $X\times [0,1]$. This then suggests that one can either work with (chart-equivalence classes of) Kuranishi structures equivalent up to Kuranishi cobordisms, or work with zigzag-equivalence classes and using Kuranishi cobordism to compare perturbations. The former suggestion is unsatisfactory because a Kuranishi structure is a geometric object (generalization of orbifolds), it would be nice to have a notion of equivalence which is geometric rather than cobordant. The second suggestion is unsatisfactory because one should be able to compare perturbations of `equivalent' Kuranishi structures more directly (preferable in the same space); and moreover, since being R-equivalent means that two Kuranishi structures are really the same (for example, two choices of associated good coordinate systems for a given Kuranishi structure are R-equivalent, but not chart-equivalent, see \ref{EXAMPLEMAP} (2) and (5) above), one should be able to compare two `equivalent' Kuranishi structures directly and geometrically. Also this R-equivalence being defined \textbf{globally} (via common refinements) means that each Kuranishi structure as an entity is not disrupted, so that one can either work on the level of Kuranishi structures and invoke the theory of R-equivalence classes when needed, or work with R-equivalence classes directly. All these considerations explain my choice of the form of the R-equivalence. (Of course, R-equivalence being an equivalence implies that zigzag-`equivalent' Kuranishi structures are R-equivalent.)

Even if we are not concerned with passing from sc-Fredholm sections to Kuranishi structures, refinements and R-equivalence are important and natural notions: Because to define a fiber product it is good to use Kuranishi structures, and to do inductive construction\footnote{E.g. the summing construction c.f. \ref{FORGETFULCONSTRUCTION} (iii), perturbations, and constructing a moduli space (equipped with natural maps) which is compactly filtered by Kuranishi structures (with corners) with boundary being fiber products of Kuranishi structures (with natural maps) appeared previously in the filtration.} it is good to use good coordinate systems, one needs to go back and forth (at least on a part of the space) and be able to relate them; but the induced Kuranishi structure $\mathcal{K}(\mathcal{G}_X)$ of a good coordinate system $\mathcal{G}_X$ obtained from a Kuranishi structure $\mathcal{K}_X$ only refines but not chart-refines $\mathcal{K}_X$, so after this process, we leave the chart-equivalence class but stay in the R-equivalence class. Moreover, if working with Kuranishi structure R-equivalence classes, the boundary of a Kuranishi structure R-equivalence class $[\mathcal{K}_X]_R$, where each of such $\mathcal{K}_X$'s covers a compact part of the moduli space and together they exhaust the moduli space, has a structure of a union of fiber products such that the fiber product operation is graded commutative and associative.

In March 16, 2012, I gave my definition of refinement and a version of level-1 structure of a good coordinate system in \cite{DY} (which is useful for compatibly lifting and extending perturbations to local sections of varying dimensions, as well as for showing admitting a common refinement, the R-equivalence, is an equivalence relation which helps to address the non-canonicality of local finite dimensional reductions), written as a short note at the begining of discussion of Kuranishi google group. The idea is that we can turn the system of embeddings in the coordinate changes into a system of submersions onto the images of embeddings in a compatible way to have a global structure, called \textbf{level-1 structure}. My motivation in doing this is twofold: One is to show that R-equivalence is transitive by composing zigzags via base submersions and bundle projections in level-1 structures. The second motivation is regarding to the inductive construction of perturbations, c.f. a local finite dimensional case of \ref{KEYREMARKS} (4). Because the tangent bundle condition is specified (canonically) at zeros only, if one does the inductive chartwise perturbation, by the very definition of induction, when perturbing lower-ordered charts, one has no prior knowledge of the validity region of the tangent bundle condition (open condition after making some choice) for coordinate changes into the higher-ordered/dimensional charts. So when doing perturbation in the higher-ordered chart, the transversality in the normal directions of the embedded images of lower-ordered charts is possibly not guaranteed by the tangent bundle condition (near the zeros), one needs to either introduce some arbitrary normal perturbation which will mess up the tangent bundle conditions into higher ordered charts later, or go back and modify the perturbed sections in the lower-ordered charts to be smaller. One can also carry a sequence of increasingly uniform small perturbations inductively constructed for each chart, and argue that at each stage of inductive construction of higher ordered charts, for sufficiently small ones, one can still use them as the tangent bundle condition is valid there for any choice of data to define it with smallness depending on this choice; and at each stage, we throw away finitely many of them. A more elegant solution is to establish regions of validity of the tangent bundle condition before carrying out the inductive construction of perturbations (especially for moduli spaces with fiber product structures on the boundaries where we want the perturbations to respect these structures), and to have compatible mechanisms to lift and extend the prior constructed perturbations, and moreover this submersion/bundle structure also gives the correct topology for dimension jump, see \cite{DI} 13 (B). This is the second motivation of defining a level-1 structure, and associating level-1 structures to Kuranishi structures and to embeddings between Kuranishi structures (via good coordinate systems and embeddings between good coordinate systems respectively), and this is done in \cite{DI}.

\begin{definition}\label{LEVELONECC}(level-1 structure for a coordinate change, \cite{DI}, also see \cite{DI} 7.24 for \textbf{level-1 good coordinate system} (indexed more generally, see \cite{DI} subsection 7.1 and remark 7.28 (II)) and \cite{DI} 7.27 for \textbf{level-1 embedding})

Let $(C_x\to C_y)=(\phi_{xy},\hat\phi_{xy},U_{xy})$ be a coordinate change in a good coordinate system $\mathcal{G}_X$. A \textbf{level-1 structure} associated to it (after omitting some technicality, see full details in \cite{DI} 7.16) is the tuple $$(C_y|_{U'_y}\overset{\text{level-1}}{\to} C_x|_{U'_x}):=(U'_y, U'_x, U'_{xy}, W_{xy}, \pi_{xy}, \tilde\pi_{xy}, \hat\pi_{xy}),\;\;\text{where}$$
\begin{enumerate}
\item $U'_y\subset U_y$ is a $G_y$-invariant open subset such that $y\in X_y|_{U'_y}:=\psi_y(\underline{(s_y|_{U'_y})^{-1}(0)})$, similarly for $U'_x$, and $U'_{xy}=U'_y\cap \phi_{xy}^{-1}(U'_x)\subset U_{xy}$;
\item $W_{xy}$ is a $G_y$-invariant tubular neighborhood of $\phi_{xy}(U'_{xy})$ in $U'_x$, and $\pi_{xy}:W_{xy}\to \phi_{xy}(U'_{xy})$ is a $G_y$-equivariant submersion with disk-like fibers such that $(s_x|_{W_{xy}})^{-1}(0)=\phi_{xy}((s_y|_{U'_{xy}})^{-1}(0))$\footnote{When we talk about a level-1 structure of a coordinate change, we just use notation $W_{xy}$; but we talk about a level-1 structure for a coordinate chnage and its shrinking $C_y|_{U'_y}\to C_x|_{U'_x}$, we will use $W'_{xy}$ and $W''_{xy}$ respectively.};
\item $\tilde\pi_{xy}:\tilde E_{xy}\to \hat\phi_{xy}(E_y|_{U'_{xy}})$ is $G_y$-equivariant fiberwise isomorphic bundle map covering $\pi_{xy}$ and $\tilde E_{xy}$ is a $G_y$-invariant subbundle of $E_x|_{W_{xy}}$ extending $\phi_{xy}(E_y|_{U'_{xy}})$ such that $(s_y|_{W_{xy}})/\tilde E_{xy}$ is transverse (to the 0-section), and
\item $\hat\pi_{xy}: E_x|_{W_{xy}}\to \tilde E_{xy}$ is a $G_y$-equivariant bundle projection covering identity between bases $W_{xy}$.
\end{enumerate}
\end{definition}

Working out the form of compatibility of level-1 coordinate changes (and with embeddings) and constructing a level-1 structure globally are highly non-trivial tasks, and see \cite{DI} 7.19 (7.21) and 7.27 for the former, and \cite{DI} section 7 for the latter.

One of the main theorems in \cite{DI} is that:
\begin{theorem}\label{LEVEL1CATEXIST} (\cite{DI}) Let $\mathcal{K}_X$ be a Kuranishi structure satisfying the maximality and topological matching conditions \ref{MAXMATCHING} (1) and (3). (See the obtaining of these two conditions for an arbitrary FOOO's Kuranishi structure up to refinement in \ref{FOOOISDIR}; and from this result obtaining these two conditions for both Kuranishi structures in an embedding up to refinement is trivial.) Then there exists a level-1\footnote{always Hausdorff by construction} good coordinate system $\mathcal{G}_X^{\text{level-1}}$ for $\mathcal{K}_X$. 

For an embedding $\mathcal{K}_X\Rightarrow\mathcal{K}'_X$ between such Kuranishi structures and a fixed choice of level-1 good coordinate system $\mathcal{G}_X^{\text{level-1}}$ for $\mathcal{K}_X$, there exists a level-1 embedding $\tilde{\mathcal{G}}_X^{\text{level-1}}\overset{\text{level-1}}{\Rightarrow}{\hat{\mathcal{G}}}_X^{\text{level-1}}$ associated to the embedding of Kuranishi structures, where ${\tilde{\mathcal{G}}}_X^{\text{level-1}}$ is a level-1 shrinking of ${\mathcal{G}}_X^{\text{level-1}}$ and ${\hat{\mathcal{G}}}_X^{\text{level-1}}$ is a level-1 good coordinate system for $\mathcal{K}'_X$. 

The same results also hold, if starting with a good coordinate system $\mathcal{G}_X$ with the maximality and topological matching conditions, and an embedding $\mathcal{G}_X\Rightarrow\mathcal{G}'_X$ between such objects respectively.
\end{theorem}

\begin{remark} Philosophically speaking, a level-1 embedding between two Kuranishi structures contains the data of a \textbf{global left inverse} for the underlying embedding (forgetting the level-1 structure). Using the terminology of \ref{THEMAP} and \ref{EXAMPLEMAP} (5) and (1), if using $\Phi:\mathcal{G}_X\Rightarrow\tilde{\mathcal{G}}_X$ to denote the underlying embedding between the underlying good coordinate systems for a level-1 embedding $\mathcal{G}_X^{\text{level-1}}\overset{\text{level-1}}{\Rightarrow}\tilde{\mathcal{G}}_X^{\text{level-1}}$ between level-1 good coordinate systems, then the level-1 structure for the embedding provides a map $\Pi: \tilde{\mathcal{G}}_X\to\mathcal{G}_X$ (see \cite{DI} 7.27 for the data in a level-1 embedding) such that $\Pi\circ \Phi=Id_{\mathcal{G}_X}$.
\end{remark}

\begin{remark} I thank Ono for pointing out to me that a data (described for a single coordinate change) similar to a part of a level-1 structure has been used in proof of \cite{FO} 6.4 as an auxiliary tool to construct a virtual fundamental chain without being given any name. A level-1 structure is very useful as we will see later in the article. Another use is that one can use it to do a single (multisectional) perturbation for the entire moduli space with a finer structure (e.g. boundaries being fiber products in Lagrangian Floer theory).
\end{remark}

Having motivated the notions of refinement, R-equivalence and level-1 structure, we will discuss Kuranishi structures from perspectives of FOOO \cite{FOOO}, \cite{FOOONEW}, \cite{FOOONEWER} and my version \cite{DI}, \cite{MAXCAT} and some comparisons and interactions between them. I emphasize that one also gets a complete theory by FOOO's \cite{FOOO}, \cite{FOOONEW}, \cite{FOOONEWER} and \cite{FOOOCF}, but my version complements it well and offers precise structures to facilitate and sometimes simplify certain viewpoints/description/constructions.

First come some philosophical points. There are \textbf{two presentations of a global compatible system of local finite dimensional reductions}: Kuranishi structures and good coordinate systems. The reasons Kuranishi structures are chosen as starting point by \cite{FO} and \cite{FOOO}, to me at least, are the following\footnote{In a future paper, I will show that the Kuranshi structures with the maximality and topological matching conditions enjoy nice properties of both Kuranishi structures and good coordinate systems in this regard. Essentially, one does not need to take a finite cover, but just grouping and ungrouping the index set $X$ of the ordered Kuranishi structure obtained by a shrinking, see the proof of \ref{GCSWOIND} for a definition/construction, and slightly generalizing the theory in \cite{DI} we can achieve a level-1 Kuranishi structure on a precompact shrinking/chart-refinement (if $X$ is compact). This is a key technical input for the future work of constructing a category of Kuranishi structures without the extra level-1 type data (with objects being honest virtual structures, not R-equivalence classes).}:
\begin{enumerate}
\item the situation about every point (the chart based at it and coordinate changes from nearby to the chart of this point) is at the same footing,
\item the fiber product of Kuranishi structures (with weakly submersive maps to a common manifold, for simplicity to illustrate) is simple and natural to define by just chartwise fiber products with fiber product coordinate changes among them (very important in dealing with moduli space with corners and with finer structures),
\item having a common chart-refinement is an equivalence relation for Kuranishi structures, but is \textbf{not} an equivalence relation for good coordinate systems, and
\item less arbitrary data (which can always be constructed if needed) in the description of Kuranishi structures (if considering up to chart-equivalence, related to points 1 and 3 above); and even when one has a good coordinate system, to have better properties, sometimes it is helpful to go to the induced Kuranishi structure and work there.
\end{enumerate}

However, to do certain constructions and perturbations, it is easier to deal with finitely many charts in an ordered cover. We need to be able to do inductive construction and need to do fiber product in the induction (in moduli spaces of Lagrangian Floer theory or SFT for example), so it is essential to go back and forth between these two notions of Kuranishi structures and good coordinate systems, at least partially. One can always easily induce a Kuranishi structure from a good coordinate system, e.g. \cite{DI} 2.11. A key construction is to go from a Kuranishi structure to a good coordinate sytem. In \cite{DI}, I introduced and imposed two natural conditions on Kuranishi structures, the upshot is that with these two conditions, it is extremely simple to go from a Kuranishi structure to a good coordinate system without the need of induction, and go to from a good coordinate system to a Hausdorff good coordinate system, on which one can do further construction, we just simply take a sufficiently small precompact shrinking also without induction (see the proof of \cite{DI} 4.5 and 4.19, and this nature of the latter is made explicit in \cite{DI} compared to \cite{Dingyu}). See \ref{GCSWOIND} below.

\begin{definition}\label{MAXMATCHING} We define some natural conditions for a Kuranishi structure $\mathcal{K}$. We omit the obvious corresponding notions for a \textbf{good coordinate system}. 
\begin{enumerate}
\item (maximality condition, \cite{DI}) Let $\mathcal{K}_X$ be a Kuranishi structure in \ref{KURANISHISTRUCTURE}. $\mathcal{K}_X$ is \textbf{maximal} or satisfies the \textbf{maximality condition} if for every $C_q\to C_p$ (so $q\in X_p$), if $z\in V_q$, $z'\in V_p$, and $\underline{z}$ and $\underline{z'}$ are connected via a finite sequence of quotiented coordinate changes each of which can be in either direction, denoted by $\underline{z}\sim \underline{z'}$, then $z\in V_{pq}$ and $\underline{\phi_{pq}(z)}=\underline{z'}$. Recall that $V_{pq}$ is the domain of the coordinate change $C_q\to C_p$ specified in $\mathcal{K}_X$.
\item ($X$-maximality, \cite{DKSG}, \cite{MAXCAT}) A Kuranishi structure $\mathcal{K}_X$ in \ref{KURANISHISTRUCTURE} is \textbf{$X$-maximal} or satisfies the \textbf{$X$-maximality condition} if $$X_q|_{V_{pq}}:=\psi_q(\underline{(s_q|_{V_{pq}})^{-1}(0)})=X_q\cap X_p.$$ In particular, a maximal Kuranishi structure is $X$-maximal.
\item (topological matching condition, \cite{Dingyu}, \cite{DI}) Let $\mathcal{K}_X$ be a Kuranishi structure in \ref{KURANISHISTRUCTURE} and let $\sim$ be the relation defined in \ref{MAXMATCHING} (1). $\mathcal{K}_X$ is \textbf{topologically matching} or satisfies the \textbf{topological matching condition} if for any two charts $C_p$ and $C_{p'}$, where $p$ and $p'$ are any two points in $X$, and any pair $v\in \underline{s_p^{-1}(0)}$ and $v'\in \underline{s_{p'}^{-1}(0)}$ such that there exist a sequence $v_n\in \underline{V_{p}}$ and a sequence $v'_n\in \underline{V_{p'}}$ satisfying $v_n\sim v_n'$, $v_n\to v$ in $\underline{V_p}$, and $v'_n\to v'$ in $\underline{V_{p'}}$, we have $v\sim v'$.
\end{enumerate}
\end{definition}

\begin{remark}\label{INHERITANCE} A shrinking of and a good coordinate system of a Kuranishi structure will inherit these conditions. A good coordinate system is $X$-maximal by definition (or it would be useless). \cite{DI} 1.17 can also be useful conceptually but is not much used.
\end{remark}

\begin{theorem}\label{GCSWOIND} (existence of (Hausdorff) good coordinate system, \cite{DI}) Let $\mathcal{K}_X$ be a Kuranishi structure \ref{KURANISHISTRUCTURE} with the maximality and topological matching conditions. There exists a good coordinate system $\mathcal{G}_X$ associated to $\mathcal{K}_X$ (constructed from data of $\mathcal{K}_X$ and the induced Kuranishi structure $\mathcal{K}(\mathcal{G}_X)$ refines $\mathcal{K}_X$). A sufficiently small precompact shrinking $\tilde{\mathcal{G}}_X$ of $\mathcal{G}_X$ without changing $(S,\leq)$ is Hausdorff. In addition (not important for the perturbation theory), by construction and definition of shrinking, we have $X_y|_{\tilde{U}_y}\cap X_x|_{\tilde U_x}\not=\emptyset$ in $\tilde{\mathcal{G}}_X$ (if and only if $X_y|_{U_y}\cap U_x|_{U_x}\not=\emptyset$ in $\mathcal{G}_X$), then $y\in X_x$ or $x\in X_y$ in the original $\mathcal{K}_X$. (We remark that the `$\leq$' constructed respects the dimension, which is important for grouping, see \cite{DI} subsection 7.1.)
\end{theorem}
\begin{proof} We only show the very first part, which is a slight modification of \cite{DI} for better presentation in the sense of bringing closer to the analogy to a compact topological space. We can define a partial order on $X$ by $q\leq p$ if $\text{dim} E_q< E_p$ or ``$\text{dim} E_q=\text{dim} E_p$ and $|G_q|\leq |G_p|$''. 

We claim that there is a shrinking $\mathcal{K}_X|_{\{V'_p\}_{p\in X}}$ of $\mathcal{K}_X$ (recall definition in \ref{EXAMPLEMAP} (3)) such that if $q\leq p$ and $X_q|_{V'_q}\cap X_p|_{V'_p}\not=\emptyset$, then $q\in X_p$ or ``$p\in X_q$ and $C_p\to C_q$ is a coordinate change with $|G_p|=|G_q|$ and $\text{dim} E_p=\text{dim} E_q$''. So in the first case we have $C_q|_{V'_q}\to C_p|_{V'_p}$ with the domain $V'_{pq}$ restricted from the original data of $\mathcal{K}_X$, and in the second case, we can invert the coordinate change $C_p\to C_q$ and restrict the latter  to have a coordinate change $C_q|_{V'_q}\to C_p|_{V'_p}$ with the domain $V'_{pq}:=\phi_{qp}(V'_{qp})$.

We now prove the claim: Since $X$ is metrizable, choose a metrc $d$ inducing its topology, to define an open metric ball $B_r(p)\subset X$. Since $X_p$ is open around $p$ in $X$, we can choose $r_p>0$ such that $B_{r_p}(p)\subset X_p$. Define $V'_p:=V_p\backslash \text{quot}_p^{-1}(\psi_p^{-1}(X_p\backslash B_{r_p/2}(p)))$ so that $X_p|_{V'_p}=B_{r_p/2}(p)$, where $\text{quot}_p: V_p\to \underline{V_p}$ is the quotient map. Here, $C_p|_{V'_p}$ is still $G_p$-invariant and $p\in X_p|_{V'_p}$.

Then if $q\leq p$ and $X_q|_{V'_q}\cap X_p|_{V'_p}\not=\emptyset$, then the latter contains some $z$ in $X$. Then $d(q,p)\leq d(q,z)+d(z,p)\leq \frac{1}{2}(r_q+r_p)\leq \max\{r_q, r_p\}$. Since we have $r_q\leq r_p$ or $r_p\leq r_q$, so we have $q\in B_{r_p}(p)\subset X_p$ or $p\in B_{r_q}(q)\subset X_q$. 

Suppose $q\not\in X_p$, then we have the alternative $p\in X_q$, so there exists a coordinate change $C_p\to C_q$ with an non-empty domain of the coordinate change. It implies that $|G_p|\leq |G_q|$ and $\text{dim} E_p\leq \text{dim} E_q$. Hypothesis says $q\leq p$, by definition, we must have in this case $\text{dim} E_p=\text{dim} E_q$ and $|G_q|\leq |G_p|$. Thus we also have $|G_p|=|G_q|$. This proves the claim. For easy future reference, such a $\mathcal{K}_X|_{(\{V'_p\}_{p\in X}, \leq)}$ is called an \textbf{ordered Kuranishi structure} restricted from $\mathcal{K}_X$.

Since $X$ is compact, we just take a finite cover with the induced $\leq$ and coordinate changes among the charts selected, and denote it by $\mathcal{G}_X$. Since $\mathcal{K}_X$ satisfies the maximality and topological matching condition in \ref{MAXMATCHING} (1) and (3), $\mathcal{G}_X$ also satisfies these conditions, see \ref{INHERITANCE}, and in particular $\mathcal{G}_X$ is $X$-maximal as in \ref{GCS} (3) (c.f. \ref{MAXMATCHING} (2)). Thus $\mathcal{G}_X$ is a good coordinate system as defined in \ref{GCS}.
\end{proof}

\begin{remark} An important distinction of Kuranishi structures compared to other variants: In the theory of Kuranishi structures, one can always go to a good coordinate system (to do perturbation or other constructions), even forgetting the data leading to the construction of Kuranishi structures (e.g. combinatorics of how local bundles are summed giving rise to global structure); and the data for Kuranishi structures is fairly minimal, which can be very useful in manipulating them in complicated situations.
\end{remark}

Theorem \ref{GCSWOIND}, theorem \ref{LEVEL1CATEXIST}, fiber product(-like) construction \cite{DI} 9.3 via tripling \cite{DI} 9.1, together with a few elementary observations, we can derive the following key result:

\begin{theorem}\label{REQUIVISEQUIV} (\cite{DI} 11.4) The R-equivalence of Kuranishi structures with the maximality and topological matching conditions \ref{MAXMATCHING} (1) and (3) is an equivalence relation. (Note that the definition of a Kuranishi structure in \cite{DI} has these two conditions built-in.)
\end{theorem}

One natural and important question is \textbf{how my version of Kuranishi structure with these two conditions in \cite{DI} and FOOO's are related}. This is what I have been thinking about for a while, and I especially want to have charts in a good coordinate system as global quotient and coordinate changes between them as global quotient as well (rather than orbifold section and orbibundle embedding as in \cite{FOOOCF}).

\begin{theorem}\label{FOOOISDIR} (\cite{MAXCAT}) Let $\mathcal{K}_X$ be a Kuranishi stucture in the sense of FOOO \cite{FOOO}, namely \ref{KURANISHISTRUCTURE}, then there exists a refinement $\mathcal{K}'_X$ in the sense of \ref{KURANISHIREFINEMENT} which satisfies the maximality and topological matching conditions \ref{MAXMATCHING} (1) and (3). Two such refinements are R-equivalent as Kuranishi structures with the maximality and topological matching conditions, and this implies that the R-equivalence for FOOO's Kuranishi structures (and respectively FOOO's good coordinate systems \ref{GCS}) is an equivalence relation. In particular, by applying \ref{GCSWOIND}, we can have a (Hausdorff) good coordinate system with global quotient in charts and global quotient in coordinate changes with the corresponding maximality and topological matching conditions, \ref{INHERITANCE}. Moreover, the level-1 theory in \cite{DI} 7.32, 7.36 applies to an FOOO's Kuranishi structure and the perturbation theory in \cite{DI} 8.6, 10.5 and 12.3 (e.g. \ref{LEVEL1CATEXIST}) applies to an oriented FOOO's Kuranishi structure (and its R-equivalence classes) respectively; and in particular, the perturbation theory factors through R-equivalence classes. Obtaining the maximality and topological matching conditions up to refinement also works for maps.
\end{theorem}

Establishing the $X$-maximality condition \ref{MAXMATCHING} (2) up to refinement is the key to the above theorem.

To have a good coordinate system with global quotient in charts and global quotient in coordinate changes is useful in several ways. Going from an FOOO's Kuranishi structure without the maximality and topological matching conditions to a good coordinate system, more directly one has charts being orbifold sections and coordinate change being orbibundle embeddings. To have the same category of Kuranishi charts and coordinate changes (this is necessary if one wants to switch between descriptions using Kuranishi structures for constructing fiber products and good coordinate systems for other constructions/perturbation), one naturally needs to define a Kuranishi structure as having orbifold sections in Kuranishi charts and orbibundle embeddings in coordinate changes. This poses one extra layer of structure but no extra conceptual difficulty (and sometimes sheaf theory is needed), see \cite{DI} 1.3 (4) and \cite{FOOOCF}. But \cite{FOOO} and \cite{DI} are written using global action for charts and coordinate changes\footnote{We often use grouping of the index set in \cite{DI} for constructions, but the theory is set-up in a way, we can group and recover the index set later after we finished the constructions, and in the end we produce some extra structure on charts in a good coordinate system indexed by the original index set.}, and of course, one can go back and rephrase or re-interpret everything using orbifold sections and orbibundle embeddings, or one can just invoke \ref{FOOOISDIR}. In particular, one can go from a Kuranishi structure defined in \cite{FOOO} to a good coordinate system defined in \cite{FOOO}. Moreover, this good coordinate system with global action of charts and coordinate changes can be used as starting points in other variants of Kuranishi structures and is more directly related to others so that different approaches can be directly compared, see section \ref{UNIFICATION}. One might not be comfortable with the fact that changing from $\mathcal{K}_X$ to its refinement can change dimensions of Kuranishi charts as well, but the induced Kuranishi structure $\mathcal{K}(\mathcal{G}_X)$ from an associated good coordinate system $\mathcal{G}_X$ of $\mathcal{K}_X$ is only a refinement of $\mathcal{K}_X$ anyway, and also perturbation and level-1 theory factors through refinement and R-equivalence in any case, see \cite{DI}.

In particular, we have:

$\{[\mathcal{K}^{\text{FOOO}}_X]_R\}\overset{I}{\to} \{[\mathcal{K}^{\text{maxmatching}}_X]_R\}\overset{II}{\to} \{[\mathcal{G}^{\text{Hausdorff}}_X]_R\}\overset{III}{\to} \{[\mathcal{G}^{\text{level-1}}_X]_R\}.$

Here, $[\cdot]$ stands for R-equivalence class and we denote a Kuranishi structure with maximality and topological matching conditions by $\mathcal{K}^{\text{maxmatching}}_X$. We describe the construction going from one structure to another, which only becomes a map after taking R-equivalence classes for the target structures (except for the inverse construction of $I$ and $III$). Construction $I$ is \ref{FOOOISDIR} and its inverse construction is inclusion (forgetting the maximality and topological matching condition). $II$ is \cite{DI} 3.1 or \ref{GCSWOIND}, and its inverse construction is inducing, see for example, \cite{DI} 2.11. The composition of $I$ and $II$ is also given in \cite{FOOONEW} and \cite{FOOONEWER} if the good coordinate system has orbifold sections as Kuranishi charts with coordinate changes being orbibundle embeddings. $III$ is \ref{LEVEL1CATEXIST}, see \cite{Dingyu} 4.4, 7.32 and 7.36, and the inverse construction is forgetting the level-1 structure. There is an appropriate notion of a morphism between two structure R-equivalence classes of the same kind so that each set above becomes a category (see \cite{MAXCAT}), and each construction above (also its inverse construction) considered as a map from a type of structure R-equivalence classes in the domain to another type in the target becomes a functor, and is actually the identity functor. Thus each functor induced from the construction and the functor of its inverse construction are mutual two-sided inverses. Therefore, we have:

\begin{theorem}\label{ALLTHESAME} (\cite{MAXCAT}) We have functors identifying R-equivalence classes between categories.
$$\{[\mathcal{K}^{\text{FOOO}}_X]_R\}\overset{I}{\leftrightarrow} \{[\mathcal{K}^{\text{maxmatching}}_X]_R\}\overset{II}{\leftrightarrow} \{[\mathcal{G}^{\text{Hausdorff}}_X]_R\}\overset{III}{\leftrightarrow} \{[\mathcal{G}^{\text{level-1}}_X]_R\}.$$
\end{theorem}

\subsection{Identification on the nose of FOOO's inductive chartwise perturbation and CF-perturbation via a level-1 structure}\label{CWPISCF}

Suppose that there is a recipe to do inductive chartwise perturbation on an oriented Hausdorff good coordinate system, so that at the end, one has a multisectionally perturbed section on each chart and coordinate changes intertwines such perturbed sections, so that the set that is covered by the quotients of the zero sets of multisectionally perturbed sections (in those charts) up to identifications is compact, and any two such perturbations are cobordant. If the Kuranishi structure is equipped with a map into an orbifold \ref{EXAMPLEMAP} (4), one then has a map from the zero set of the perturbed Kuranishi structure into this orbifold, and uses a triangulation (or a pseudocycle) and singular homology to obtain a rational singular chain/cycle of this orbifold for invariants. For example, this perturbation idea is explained in \cite{FO}, \cite{FOOO} together with more details in turning an FOOO's Kuranishi structure into a Hausdorff good coordinate system in \cite{FOOONEW} and \cite{FOOONEWER}. There is also a another way to perturb and extract invariants in de Rham cohomology via continuous family (or in abbreviation CF-) perturbation explained in detail in \cite{FOOOCF}.

Alternatively, starting from an FOOO's Kuranishi structure $\mathcal{K}_X$, one can apply \ref{FOOOISDIR} to get a refinement Kuranishi structure $\mathcal{K}'_X$ with the maximality and topological matching conditions \ref{MAXMATCHING} (1) and (3), then apply \ref{LEVEL1CATEXIST} (see also \ref{GCSWOIND}) and the theory explained in detail in \cite{DI} section 8 (via a level-1 structure, see \ref{LEVELONECC}) to get a compact perturbation constructed chartwise via induction, and in \cite{DI} 12.3 and 10.5 it is showed that two such perturbations starting from two R-equivalent Kuranishi structures (thus independent of choices via applying \ref{FOOOISDIR} to these two starting Kuranishi structures) if sufficiently small (general ones in the respective setting are cobordant to sufficiently small ones) can be lifted and compared in a common level-1 good coordinate system commonly refining the respective associated level-1 good coordinate systems and are cobordant there.

In fact, one can use a level-1 structure and perturbation method in \cite{DI} adapted to FOOO's Kuranishi structures and summarized in the previous paragraph to show that two perturbation methods of FOOO in the first paragraph of this subsection agree on the nose via the perturbation used in \cite{DI}. 

First, let us briefly recall how the compact chartwise perturbation in \cite{DI} section 8 is carried out. There, one first chooses a level-1 good coordinate system with charts $\{C_i|_{U'_i}\}_{i\leq N}$ and total order $\leq$ on $\{0,\cdots, N\}$ associated to a given Kuranishi structure (after applying \ref{FOOOISDIR}, \ref{GCSWOIND}, \ref{LEVEL1CATEXIST}, and \cite{DI} 7.1), then take a level-1 precompact shrinking $\{C_i|_{U''_i}\}$ of that and denote the domain of the coordinate change as $C_j|_{U'_j}\overset{\text{level-1}}{\to} C_i|_{U''_i}$ by $U^m_{ij}$ for $j\leq i-1$. In the inductive step, for each $j\leq i-1$, we already have a perturbation $\tau_j|_{U'_{ij}}$ on $U'_{ij}$ (if non-empty) restricted from a perturbation $\tau_j$ compactly supported in $U'_j$ such that the restriction of $\tilde s_j:=s_j+\sum_{n\leq j-1}(\pi_{jn}, \hat\pi_{jn})^\ast (\tau_n|_{U'_{jn}})+\tau_j$ to $U'_{ij}$ is transverse on $\overline{U''_{ij}}$. Here the lift $(\pi_{jn},\hat\pi_{jn})^\ast$ is constructed using the level-1 structure mapping a section of $E_n|_{U'_{jn}}$ to a section of $E_j|_{\text{orbit}(W'_{jn})}$ as in \cite{DI} 8.3 (essentially the same as the way towards the end of \ref{KEYREMARKS} (4)) followed by taking the orbit of $W'_{jn}$ in $U'_j$ under the morphism. We define $\tilde{\tilde{s}}_i:=s_i+\sum_{j\leq i-1}(\pi_{ij},\hat\pi_{ij})^\ast(\tau_j|_{U'_{ij}})$. Observe that $\tilde{\tilde {s}}_i|_{\text{orbit}(W'_{ij})}=(\pi_{ij},\hat\pi_{ij})^\ast(\tilde s_j|_{U'_{ij}})$ for $j\leq i-1$, and $\tilde{\tilde {s}}_i$ is already transverse on $\overline{\bigcup_{j\leq i-1}\text{orbit}({W''_{ij}})}$. Then we can perturb $\tilde s_i$ by an invariant multisection supported on $$(U''_i\cup\bigcup_{j+1\leq k\leq N} U^m_{ki})\backslash\overline{\bigcup_{j\leq i-1} \text{orbit}(W''_{ij})}$$ such that $\tilde s_i:=\tilde{\tilde {s}}_i+\tau_i$ is transverse on $\overline{U''_i}$. Here, the genericity trick to choose $\tau_i$ in \cite{DI} 8.5 is based on \ref{DIMREDEX}. When the induction terminates after step $N$, we restrict the perturbed section in each chart to $U''_i$\footnote{So that the transversality holds chartwise and the perturbed sections are compatible.} to obtain a perturbation for $\{C_i|_{U''_i}\}_{i\leq N}$. So basically one uses a level-1 structure to lift the previously constructed perturbations in the lower ordered charts and then extend it, and we do not change the constructions in the previous steps.

We now identify two FOOO's perturbation theories via the perturbation in the last paragraph. For a level-1 good coordinate system obtained, one can take a level-1 chart-refinement (with a possibly larger index set) so that $E_x$'s and $\tilde E_{xy}$'s are trivial bundles invariant under group actions. Choose a level-1 precompact shrinking. (A) Using a finite dimensional classical analogue of \ref{PFDFDRED} where $f$ is just $s_x$ (a section in a finite dimension bundle with group action as a special case of sc-Fredholm section) and $L_x$ is just $E_x$, $s_x$ is itself a finite dimensional reduction; and (B) one can also apply \ref{STABILIZATION} using now trivial $E_x$ for $L_x$ in the CF-perturbation for each chart; and recall \ref{SAMEFDRED}, both (A) and (B) agree exactly. For (B), denote $E_x:=\Psi_x(U'_x\times F_x)$, and recall the notation in \ref{STABILIZATION}, $\mathbf{s}_x^{F_x}: U'_x\times F_x\to pr_1^\ast E_x, (z,v)\mapsto \mathbf{s}_x^{F_x}(z,v):=s_x(z)+\Psi_x(z,v)$. Let $\mathfrak{s}^\epsilon_x(z,v):=\mathbf{s}^{F_x}(z,\epsilon v)$. Then $\mathcal{S}_x:=(F_x, \omega_x, \mathfrak{s}^\epsilon_x)$ for some appropriate choice of $\omega_x$ on $F_x$ will be a CF-perturbation for $C_x|_{U'_x}$ in the sense of \cite{FOOOCF} 7.3. Note that we can relax $\omega_x$ to depend on the base $U'_x$ direction and only ask it to be of degree $\text{rank} E_x$ and integrate into 1 along each fiber, and generalize \cite{FOOOCF} 7.5 slightly to allow a projection of the form $\Pi: U'_x\times F'_x\to U'_x\times F''_x$ covering identity $U'_x\to U'_x$ and linear in the fiber direction depending on the $U'_x$ coordinate. In fact, we can relax even further, and let $\mathbf{s}_x$ to be defined on a non-trivial bundle by the same formula which only uses the fiberwise linear structure, and the preimage image of the zero section still makes sense. So we can forget about $\Psi_x$, and write $\mathfrak{s}_x^\epsilon: E_x\to E_x, (z,v)\mapsto s_x(z)+(z,\epsilon v)$\footnote{We do not need $\epsilon\not=1$ as the tangent bundle condition is achieved on the whole $W'_{xy}$.}, and we can group index and replace subscript by $i$. Observe that we can construct $\omega_i$ and hence $\mathcal{S}_i$ inductively for the level-1 good coordinate system of charts $C_i|_{U''_i}$ such that $\mathcal{S}_i|_{W''_{ij}}$ is equivalent to $\hat\pi_{ij}^\ast(\mathcal{S}_i|_{W''_{ij}})$ via $\hat\pi_{ij}$ in the above generalized version of \cite{FOOOCF} 7.5, and $\hat\pi_{ij}^\ast(\mathcal{S}_i|_{W''_{ij}})$ can pull back to $\mathcal{S}_j|_{U''_{ij}}$ via $\tilde\pi_{ij}$ as in \cite{FOOOCF} 7.41, where $\hat\pi_{ij}$ and $\tilde\pi_{ij}$ are data in the level-1 structure \ref{LEVELONECC}.\footnote{Those statements hold true for the primed version as well except for the $\omega_i$ part.} We now use the compatibility of the level-1 structure to construct $\omega_i$. In the inductive step, we have $\omega_j$, $j\leq i-1$ already. We can pullback $\omega_{j}$ to $\hat\pi_{ij}^\ast(\tilde\pi_{ij}^\ast\omega_j)$, and wedge it with a highest degree invariant form supported near 0 in $\ker\hat\pi_{ij}$ that integrates to 1 along each fiber, and denote the end result as $\omega'_{ij}$, and we choose a $\omega'_i$ on $U'_i$ as well. We choose a partition of unity of the cover $\{\text{orbit}(W'_{ij}), j\leq i-1, U'_i\}$ where the cut-off function $\lambda_{ii}$ for $U'_i$ is supported away from $\text{orbit}(W^{w}_{ij})$, where $W^{w}_{ij}$ is associated to the level-1 coordinate change from $C_j|_{U''_j}\overset{\text{level-1}}{\to} C_i|_{U'_i}$, where $\{C_i|_{U''_i}\}_{i\leq N}$ is the grouped level-1 precompact shrinking. Define $\omega_i:=\sum_{j\leq i-1}\lambda_{ij}\omega'_{ij}+\lambda_{ii}\omega'_i$, which is the missing piece for $\mathcal{S}_i|_{U''_i}$. We have constructed compatible CF-perturbations $\mathcal{S}_i|_{U''_i}, i\leq N$. The upshot is that CF-perturbations associated to these charts are globally compatible. 

To relate two FOOO's perturbations on the nose, we construct $\omega'_i$ from $\tilde s_i$ in the two paragraphs before. Using the local structure of the multisection $\tilde s_i$ and a partition of unity, we can construct $\text{rank}(E_x)$-form $\omega'_i$ to be supported very close to $\text{gr}(\tau_i)$ and integrate to $1$ along each fiber and locally decomposable into a sum, each term of which is supported near a branch and fiberwise integrates to the absolute value of its weight. The $\omega_i$ produced by the above recipe also satisfies this property. Intuitively, by letting form $\omega_i$ concentrate on the branches, a perturbation by lifting and extending via the level-1 structure and the same perturbation viewed as CF-perturbation using the projections provided by the level-1 structure give rise to the \textbf{same} perturbed solution space. To make it precise, if we have a map from this Kuranishi structure into an orbifold, we have an induced map from the perturbed solution space by restricting, and an induced map from $\mathcal{S}_i|_{U''_i}, i\leq N$ via pullback by $E_i|_{U''_i}\to U''_i$. If we triangulate the perturbed space such that any simplex will lie entirely in some chart, each singular chain corresponds to a de Rham chain (see \ref{DERHAMCHAIN}\footnote{We can either use a version of manifold with corners (where each chain is equipped with an embedding into some $\R^N$ so that its automorphism intertwining this data is trivial) or choose a fixed way to turn a standard $k$-simplex into $\R^k$.}) and different choices corresponding to different $\mathcal{S}_x$'s are equivalent by the formulation, agreement follows because the singular chain complex is quasi-isomorphic to the de Rham chain complex. Alternatively, one can integrate a de Rham form on orbifolds on the pushforward under the given map of a perturbed solution space, and compare it with \cite{FOOOCF} (7.6) and (7.36), and show the agreement. Both perturbations above are respectively examples of  FOOO's chartwise and CF perturbations and each agrees with other general perturbations in each framework by respective recipes of choice-independence.

\subsection{Forgetful construction from sc-Fredholm structures to Kuranishi structure R-equivalence classes}\label{FORGETSUBSECTION}

The forgetful construction, independence of choices and certain functorial properties were completed by the author in a write-up with full proofs on August 1, 2011. The detail in \cite{DII} is unchanged except we choose fiber-standard subbundles to which $f$ is in good position rather than in general position (details and proofs for this setting are almost identical). The intervening time has been dedicated to develop a theory of structures from the output of the forgetful construction in \cite{DI} and precisely relating it to FOOO's Kuranishi structures in \cite{MAXCAT}.

We will describe ideas in the forgetful construction, and see \cite{Dingyu} and \cite{DII} for details.

We start with any fixed sc-Fredholm section (functor) $f$ in a strong bundle ep-groupoid $E\to B$ such that the zero set in orbit space $\underline{f^{-1}(0)}$ is compact and homeomorphically identified with $X$. Note that the object space $B$ of the base is paracompact, second countable and metrizable. Thus $X$ involved is always compact metrizable. \textbf{We always use the same letter for the ep-groupoid as well as the object space for lighter notation; and this causes no confusion, since the construction will always be globally invariant collection of invariant local constructions.} We call the data $$\mathcal{P}_X:=(X, f:B\to E, \psi:\underline{f^{-1}(0)}\to X)$$ an \textbf{sc-Fredholm structure} or \textbf{polyfold Fredholm structure}. We can take a Morita-refinement $\Psi$ from a strong bundle ep-groupoid $E''$ to $E$ such that $\Psi$ induces a map between the zero sets in the orbit spaces that intertwines $\psi$'s. We can also shrink $B$ to an invariant open neighborhood $B'$ of $f^{-1}(0)$ in $B$, and consider the open inclusion $f|_{B'}$ to $f$. Recall from \ref{SCFREDEMBEDDING} (with the obvious adaptation to include $\psi$ in $\mathcal{P}_X$): A map of the form of a composition of a Morita-refinement followed by an open inclusion\footnote{Equivalently, a composition of a Morita-refinement followed by an open embedding; and see \ref{SCFREDEMBEDDING} for another trivially equivalent formulation.} is called a \textbf{chart-refinement}. We define $\mathcal{P}'_X$ to be a \textbf{refinement} of another $\mathcal{P}_X$ if there exists a diagram $\mathcal{P}_X\overset{\sim}{\leftarrow}\mathcal{P}''_X\Rightarrow\mathcal{P}'_X$, where the first arrow is a chart-refinement and the second arrow is an embedding. We define chart-equivalence, respectively \textbf{R-equivalence} as before by requiring to admit a common chart-refinement, respectively a common refinement. The construction of the forgetful map below will descend to the chart-equivalence class as well as the R-equivalence class of $\mathcal{P}_X$.

We will describe each stage of the construction and define some terminologies for easy reference later.

\begin{construction}\label{FORGETFULCONSTRUCTION} (\cite{Dingyu}, \cite{DII})
\begin{enumerate}[(i)]
\item For any $x\in f^{-1}(0)$, apply the first part of \ref{PFDFDRED} (1), we can get an invariant fiber-standard $L_x\to W_x$ with $W_x$ open in $B$ such that $f\pitchfork L_x$ (recall the shorthand in \ref{IGP} for $f$ being in good position to $L_x$) and choose an invariant open set $U_x\subset\overline{U}_x\subset W_x$ around $x$ (recall such $W_x$ forms a basis, see footnote \ref{FORMINGABASIS}). We make a globally invariant family of such locally invariant choices for all $x\in f^{-1}(0)$. Since $\underline{f^{-1}(0)}$ in the orbit space is compact, we can choose an invariant $I\subset f^{-1}(0)$ such that $J:=\underline{I}$ is finite, and $f^{-1}(0)\subset\bigcup_{x_i\in I} U_{x_i}$ (notice that this cover is locally finite). We call such a choice $\mathcal{D}:=\{(L_{x_i},W_{x_i}, U_{x_i})\}_{x_i\in I}$ a \textbf{presummable data}.
\item If $L_{x_i}$'s in $\mathcal{D}$ are fiberwise independent among themselves over the common bases, then $\mathcal{D}$ is called \textbf{summable}. In applications, summability can be achieved by changing $L_{x_i}$ slightly using consideration in the geometric context. For any fixed sc-Fredholm section $f$ in this abstract setting, it is hard to achieve it directly. But we can pass $f$ to $\mathbf{f}:=f^J: \mathbf{B}:=E^{\oplus J}\to \mathbf{E}:=((\pi_E)^{\oplus J})^\ast(E\oplus E^{\oplus J})$ as in \ref{SUMMABILITYTRICK}, which refines and is equivalent to $f$ as we discussed. Here we use $J$ to index the factors instead of $|J|$. We identify $f^{-1}(0)$ with $\mathbf{f}^{-1}(0)$ and use the same notation for its points and the points in $J$. Define $\mathbf{W}_{x_i}:=(\pi^{\oplus J})^{-1}(W_{x_i})$ and $\mathbf{U}_{x_i}:=(\pi^{\oplus J})^{-1}(U_{x_i})$. Denote $\mathbf{L}_{x_i}$ is the diagonal subbundle of the two copies of $(\pi^{\oplus J})^\ast L_{x_i}$ in the first and $\underline{x_i}$-th Whitney summands of $\mathbf{E}$ over $\mathbf{W}_{x_i}$. Note that $\mathbf{f}\pitchfork \mathbf{L}_{x_i}$ by choice, because $f\pitchfork L_{x_i}$. $\mathbf{D}:=\{(\mathbf{L}_{x_i}, \mathbf{W}_{x_i}, \mathbf{U}_{x_i})\}_{x_i\in I}$ is completely determined from $\mathcal{D}$ and is now summable (recall by definition, fiberwise independent among themselves over the common bases).
\item For each $x\in \mathbf{f}^{-1}(0)\cong f^{-1}(0)$, over an invariant open neighborhood $\mathbf{O}_{x_i}$ of $x$ in $\mathbf{B}$, we sum finitely many bundles to be a bigger bundle. We choose to include the bundle $\mathbf{L}_{x_i}$ to be summed if $\overline{\mathbf{U}}_{x_i}$ covers $x$. The point is that being closed, if $\overline{\mathbf{U}}_{x_i}$ does not cover $x$, then it does not cover any $y$ from a sufficiently small open neighborhood of $x$ in $\mathbf{B}$. So the subcollection of bundles to be summed for $x$ is no smaller than that for any zero $y$ sufficiently nearby, which leads to the coordinate change condition in Kuranishi structure \ref{KURANISHISTRUCTURE}. We want to do the finite dimensional reduction using the summed bundle, so we need to choose aforementioned $\mathbf{O}_x$ be an open set in $\mathbf{B}$, that is why (for $x\in \overline{\mathbf{U}}_{x_i}\backslash\textbf{U}_{x_i}$) we need some open neighborhood around (the boundary of) $\overline{\mathbf{U}}_{x_i}$, that is what we keep $\mathbf{W}_{x_i}$ for. To make the above more succinct, we denote $I_x$ for the subset of bundles to be summed for $x$, namely, $I_x:=\{x_i\in I\;|\; x\in\overline{\mathbf{U}}_{x_i}\}$. Just choose an invariant $\mathbf{O}_x$ around $x$ in $\bigcap_{x_i\in I_x}\mathbf{W}_{x_i}\bigcap\bigcap_{x_j\in I\backslash I_x}(\mathbf{B}\backslash \overline{\mathbf{U}}_{x_j})$, then for $y\in \mathbf{O}_x\cap \mathbf{f}^{-1}(0)$, we have $I_y\subset I_x$. We choose a globally invariant family of locally invariant $\mathbf{O}_x$, $x\in f^{-1}(0)$. Now we have picked data $\mathcal{O}:=\{\mathbf{O}_x\}_{x\in f^{-1}(0)}$. From the data $\mathcal{D}$ and $\mathcal{O}$, we can define a globally invariant family of invariant local fiber-standard bundles $\mathbf{H}_x:=\bigoplus_{x\in I_x}(\mathbf{L}_{x_i}|_{\mathbf{O}_x})\to \mathbf{O}_x$ for $x\in \mathbf{f}^{-1}(0)$. $\{\mathbf{H}_x\to\mathbf{O}_x\}_{x\in f^{-1}(0)}$ is called a \textbf{reducible presentation}.
\item $\mathbf{f}\pitchfork\mathbf{H}_x$ as $\mathbf{f}\pitchfork \mathbf{L}_{x_i}|_{\mathbf{O}_x}$ already. By applying the second part of \ref{PFDFDRED}. We get a $G_x$-invariant manifold with corners $V_x:=\{\mathbf{f}\in \mathbf{H}_x\}=((\mathbf{f}|_{\mathbf{O}_x})/\mathbf{H}_x)^{-1}(0)$ and the local finite dimensional reduction $s_x:=\mathbf{f}|_{\{\mathbf{f}\in \mathbf{H}_x\}}:V_x\to E_x:=\mathbf{H}_x|_{V_x}$ is the Kuranishi chart, except it is not indexed by $p:=\psi(\underline{x})\in X$ yet. We have the coordinate change $\hat \phi_{xy}$ covering $\phi_{xy}$ induced by the inclusion in $\mathbf{E}$ and $\mathbf{B}$ and the domain of the coordinate change is $V_{xy}:=\mathbf{O}_y\cap \mathbf{O}_x\cap \mathbf{f}^{-1}(0)$. By the invariance of the construction, the compatible data $$(\{\mathbf{f}^{-1}(0), \{(G_x, s_x)\}_{x\in \mathbf{f}^{-1}(0)},\{(\phi_{xy}, \hat\phi_{xy}, V_{xy})\}_{y\in \mathbf{O}_x\cap \mathbf{f}^{-1}(0)} \})$$
descends to $$\mathcal{K}(\mathcal{P}_X; \mathcal{D},\mathcal{O}):=(X,\{(G_p, s_p, \psi_p:=\psi|_{\underline{s_p^{-1}(0)}})_{p\in X}, \{(\phi_{pq},\hat\phi_{xy}, V_{xy})\}_{q\in X_p}).$$ By descending, I mean $V_p:=V_x$ for any fixed choice of $x\in \underline{x}\in \psi^{-1}(p)$, etc. Coordinate changes up to $G_p$ actions are immediate (and also natural from this perspective). This Kuranishi structure is maximal, topological matching and even Hausdorff, as $B$ is.
\item We make two non-canonical choices $\mathcal{D}$ and $\mathcal{O}$ in this construction. We examine the independence of choices. For a fixed $\mathcal{D}$, $\mathcal{K}(\mathcal{P}_X;\mathcal{D},\mathcal{O})$ and $\mathcal{K}(\mathcal{P}_X;\mathcal{D},\mathcal{O}')$ are both chart-refined by $\mathcal{K}(\mathcal{P}_X;\mathcal{D},\mathcal{O}\cap \mathcal{O}')$, here the notation $\mathcal{O}\cap \mathcal{O}':=\{\mathbf{O}_x\cap \mathbf{O}'_x\}_{x\in f^{-1}(0)}$. In general, $\mathcal{K}(\mathcal{P}_X;\mathcal{D},\mathcal{O})$ and $\mathcal{K}(\mathcal{P}_X;\mathcal{D}',\mathcal{O}')$ are refined by $\mathcal{K}(\mathcal{P}_X;\mathcal{D}\bigsqcup\mathcal{D}',\mathcal{O}''\subset \Pi^{-1}\mathcal{O}\cap (\Pi')^{-1} \mathcal{O}')$. Here $\mathcal{D}\bigsqcup\mathcal{D}'$ means we take disjoint union of presummable data even if points in $I$ and $I'$ might coincide. There exist natural projections $\Pi$ and $\Pi'$ from $\mathbf{B}'':=E^{J\sqcup J'}$ onto $\mathbf{B}$ and $\mathbf{B}'$. In the above notation, we mean elementwise, namely $\mathbf{O}''_x\subset \Pi^{-1}(\mathbf{O}_x)\cap (\Pi')^{-1}(\mathbf{O}'_x)$. Thus, we have a well-defined map $\check{\mathfrak{forget}}:\mathcal{P}_X\mapsto [\mathcal{K}(\mathcal{P}_X;\mathcal{D},\mathcal{O})]_R$ from a polyfold Fredholm structure to a Kuranishi structure R-equivalence class.
\end{enumerate}
\end{construction}

\begin{remark}\label{FORGETREMARK} Here are some brief remarks about the forgetful map above:
\begin{enumerate}
\item The above also works for a classical Fredholm section in a Banach bundle \'etale-proper Lie groupoid with a compact zero set in the orbit space, by applying classical analogue of \ref{PFDFDRED} instead.
\item The orientation for an oriented sc-Fredholm section naturally induces an orientation for a Kuranishi structure (and the associated R-equivalence class) constructed by the above forgetful construction.
\item The forgetful map is the first ever global finite dimensional reduction for any (polyfold) Fredholm section with group actions (ep-groupoid structure) and a compact zero set in orbit space into a \textbf{geometric structure}. Moreover, the ambiguity is explained via fairly geometric R-equivalence, rather than cobordism, and perturbation theory for oriented Kuranishi structure descends to R-equivalence classes and one can compare perturbations in a common refinement.
\item To have compact perturbation in polyfold theory, one needs sc-smooth partitions of unity and an auxiliary norm, and thus needs retracts are retracting from sc-Banach spaces with level-0 separable Hilbert spaces for the former and requires reflexivity of $(0,1)$-fibers if reflexive auxiliary norm is needed. However, using the forgetful construction and perturbation theory in \cite{DI}, one can compactly perturb any sc-Fredholm section $f:B\to E$ even when its base $B$ does not support cut-off functions. This can be thought of as an application of the forgetful construction, although for the symplecto-geometric moduli spaces in mind this is purely academic.
\item The forgetful map $\check{\mathfrak{forget}}$ descends to a map $\mathfrak{forget}$ mapping from R-equivalence classes of polyfold Fredholm structures, defined just by $\mathfrak{forget}:[\mathcal{P_X}]_R\mapsto \check{\mathfrak{forget}}(\mathcal{P}_X)$ using any choice $\mathcal{P}_X\in[\mathcal{P}_X]_K$. This is a well-defined map; and the map is invertible, see \ref{INVERTIBILITY}.
\item The forgetful map between R-equivalence classes can be upgraded into a functor which also maps between morphisms.
\item The forgetful functor intertwines perturbation theories for the oriented structure R-equivalence classes.
\item If group action/ep-groupoid is trivial, then there is a single chart finite dimensional reduction first pointed out to me by Hofer. Just apply \ref{DIMREDEX} (or its classical analogue if we are dealing with classical Fredholm sections). First pass $f$ to equivalent $\mathfrak{f}$, then do finite dimensional reduction $\mathfrak{f}|_{\{\mathfrak{f}\in pr_2^\ast \R^N\}}$. This can be shown to be R-equivalent to the construction \ref{FORGETFULCONSTRUCTION} above by idea similar to \ref{FORGETFULCONSTRUCTION} (v).
\item (This answers a question by Ono.) Using ideas in \ref{FORGETFULCONSTRUCTION}, we can show for example, for a fixed $J$, Kuranishi structures for a concrete moduli space without boundary (for simplicity) from two different choices in the constructions of \cite{FO} and \cite{FOOONEW} are R-equivalent.
\item R-equivalent Kuranishi structures are Kuranishi cobordant via Kuranishi structures on $X\times I$.
\item One reason that functoriality between structures does not feature much in either polyfold theory or theory of Kuranishi structures so far is that the ambient space as a whole is a single polyfold and when multiple intrinsically different Kuranishi structures appear in the same context, there are often just inclusions (or cobordisms) among them.
\end{enumerate}
\end{remark}

\subsection{Globalization construction from level-1 good coordinate systems to sc-Fredholm structure R-equivalence classes}\label{GLOBALIZATION}

We can construct a polyfold Fredholm structure from a (always Hausdorff, maximal and topologically matching) level-1 good coordinate system $\mathcal{G}^{\text{level-1}}_X$. Namely, starting with any level-1 good coordinate system $\mathcal{G}^{\text{level-1}}_X$, there exists a level-1 precompact shrinking $\hat{\mathcal{G}}^{\text{level-1}}_X$ (via level-1 structure) of $\mathcal{G}^{\text{level-1}}_X$ such that denoting $\underline{B}:=\bigsqcup_{x\in S}\hat V_x/\sim$ from $\{\hat V_x\}_{x\in S}$ in $\hat{\mathcal{G}}^{\text{level-1}}_X$ identified using $\sim$ introduced in \ref{MAXMATCHING} (1) and $\underline{E}:=\bigsqcup_{x\in S}E_x|_{\hat V_x}/\sim^{\text{bundle}}$ identified similarly using the corresponding bundle version $\sim^{\text{bundle}}$, we have an induced section $\underline{f}:\underline{B}\to\underline{E}$ from $s_x|_{\hat V_x}$. Moreover, $\underline{E}\to \underline{B}$ is actually the orbit space of a strong bundle $E\to B$ constructed crucially using level-1 structure on $\mathcal{G}^{\text{level-1}}_X$ and some (strong bundle) retract models, and we have an sc-Fredholm $f: B\to E$ inducing $\underline{f}$ between the orbit spaces. Denote the resulting sc-Fredholm structure $P(\hat{\mathcal{G}}^{\text{level-1}}_X,\text{data}):=(X, f, \psi:\underline{f}^{-1}(0)\cong \underline{f^{-1}(0)}\to X)$, where $\text{data}$ is some information obtained from $\mathcal{G}^{\text{level-1}}_X$. Different choices in the construction leads to R-equivalent sc-Fredholm structures, using the fact that different level-1 (chart-) refinements of a level-1 good coordinate system are R-equivalent as level-1 good coordinate systems.

Therefore, to go from a Kuranishi structure $\mathcal{K}_X$ exactly defined by FOOO in \cite{FOOO} to a polyfold Fredholm structure of HWZ, we need to apply theorems \ref{FOOOISDIR} and \ref{LEVEL1CATEXIST} and then the above, and denote the combined effect of this process by $\mathcal{G}^{\text{level-1}}(\mathcal{K}_X,\text{choice})$. To see that the composition of the above processes leads to a well-defined map from an FOOO's Kuranishi structure into an sc-Fredholm structure R-equivalence class in spite of various choices made, we apply the last sentence in the previous paragraph again, together with functoriality for embeddings between Kuranishi structures to level-1 embedding between associated level-1 good coordinate systems.

This sketches the main result in \cite{DIII}:
\begin{theorem} (globalization functor, \cite{DIII}) We have a well-defined map from FOOO's Kuranishi structures to sc-Fredholm R-equivalence classes. This map descends to a map from FOOO's Kuranishi structure R-equivalence classes. This map can be upgraded to a functor $$\mathfrak{globalize}:[\mathcal{K}_X]_R\mapsto [\mathcal{P}(\hat{\mathcal{G}}^{\text{level-1}}_X, \text{data})]_R$$ where $\mathcal{G}^{\text{level-1}}_X:=\mathcal{G}^{\text{level-1}}(\mathcal{K}_X, \text{choice})$, and morphisms in the domain and target categories will be introduced respectively in \cite{MAXCAT} and \cite{DIII}.
\end{theorem}

\begin{remark} 
\begin{enumerate}
\item The construction $\mathcal{P}(\hat{\mathcal{G}}^{\text{level-1}}_X, \text{data})$ from $\hat{\mathcal{G}}^{\text{level-1}}_X$ uses some information in $\mathcal{G}^{\text{level-1}}_X$. We do an inductive globalization construction, and after each step, the chart numbers in the level-1 good coordinate systems decrease by 1. However after the first step, each Kuranishi chart becomes an sc-Fredholm structure with embeddings \ref{SCFREDEMBEDDING} among them (so we have some sort of a polyfold Kuranishi structure in the intermediate step). Thanks to the definition of embedding of sc-Fredholm sections \ref{SCFREDEMBEDDING}, the globalization process is functorial and can still be applied in the inductive step; and after finitely many steps, it terminates yielding an sc-Fredholm section. During the construction, one also needs to switch from the orbifold description to the ep-groupoid one, see section \ref{ORBIFOLD}.
\item (Answering a question by Ohta.) If a Kuranishi structure has $G$-equivariance structure, then the sc-Fredholm section constructed by globalization construction in \cite{DIII} will also have this $G$-equivariance structure. There should be results beyond this and this direction might get explored in a future paper.
\end{enumerate}
\end{remark}

\subsection{Polyfold--Kuranishi correspondence}\label{PKC}

\begin{theorem}\label{INVERTIBILITY} (\cite{DI}, \cite{MAXCAT}, \cite{DII}, \cite{DIII}) $\mathfrak{forget}$ aforementioned in \ref{FORGETREMARK} (5) descended from $\check{\mathfrak{forget}}$ in \ref{FORGETFULCONSTRUCTION} and $\mathfrak{globalize}$ in subsection \ref{GLOBALIZATION} as functors between categories of structure R-equivalence classes are mutual two-sided inverses. Each perturbation theory factors through oriented structure R-equivalence classes. Each functor as a map between oriented structure R-equivalence classes intertwines the respective perturbation theories.
\end{theorem}

\begin{corollary} (R-equivalence to a single chart in the trivial stabilizer case) If an FOOO's Kuranishi structure $\mathcal{K}_X$ is stabilizer-free (namely $G_p=\{Id\}$ for all $p\in X$), then it is R-equivalent to a Kuranishi structure induced by a section in a finite dimensional vector bundle, namely a single chart. If $\mathcal{K}_X$ is oriented, so is this single chart, and the perturbation theory of this single chart (a good coordinate system) agrees with that of $\mathcal{K}_X$.
\end{corollary}

\begin{proof} By applying \ref{FOOOISDIR} and \ref{LEVEL1CATEXIST}, then \ref{GLOBALIZATION}, one has an sc-Fredholm section $f$ with the base M-polyfold based on retracts retracting from sc-Banach spaces with level-0 being separable Hilbert spaces, due to the retract models chosen. Then sc-smooth cut-off functions exist, and one can apply \ref{DIMREDEX} and its finite dimensional reduction $\mathfrak{f}|_{\{\mathfrak{f}\in pr_2^\ast \R^N\}}$ over there is the desired single chart. Since $\mathfrak{globalize}[\mathcal{K}_X]_R=[f]_R=[\mathfrak{f}]_R$. Using idea of \ref{FORGETFULCONSTRUCTION} (v), one can show $[\mathfrak{forget}(\mathfrak{f})]_R=[\mathcal{K}(\mathfrak{f}|_{\{\mathfrak{f}\in pr_2^\ast \R^N\}})]_R$, where $\mathcal{K}(\cdot)$ denotes a Kuranishi structure induced by a good coordinate system in the input. Now $[\mathcal{K}(\mathfrak{f}|_{\{\mathfrak{f}\in pr_2^\ast \R^N\}})]_R=[\mathfrak{forget}(\mathfrak{f})]_R=\mathfrak{forget}[\mathfrak{f}]_R=\mathfrak{forget}\circ\mathfrak{globalize}[\mathcal{K}]_R=Id[\mathcal{K}_X]_R=[\mathcal{K}_X]_R$, which says the Kuranishi structure $\mathcal{K}_X$ is R-equivalent to a Kuranishi structure induced by a single chart $\mathfrak{f}|_{\{\mathfrak{f}\in pr_2^\ast \R^N\}}$. Then apply theorem \ref{INVERTIBILITY} regarding the perturbation theories.
\end{proof}

\begin{remark}\label{PKAPPLICATION} Here are some immediate applications of polyfold--Kuranishi correspondence:
\begin{enumerate}[(1)]
\item For fixed $J$, the moduli space can be equipped with a canonical Kuranishi structure R-equivalence class. The Kuranishi structure R-equivalence class itself can be considered as an invariant (for a fixed $J$). Kuranishi structure R-equivalence classes for the moduli spaces for different $J$ are compared via Kuranishi cobordism R-equivalences similar to the usual case.
\item One can transfer analytic foundation from one framework to the other, and then get the same results by using one framework with the given analytic foundation and using the other framework with the transferred analytic foundation. So for any given moduli space, analytic foundation only needs to be established for one framework. Viewed differently, one can also transfer perturbation theories across the frameworks.
\item (Answering Solomon's question.) Recall that polyfold construction for the moduli space is fairly canonical: fixing an exponential gluing profile and a function space choice for sc-structure (a sequence of appropriate weights for $H^{3+\cdot}$ or some other function space choice), a geometric auxiliary data (almost complex structure $J$, contact form etc), one then has a canonical object $f: B \to E$, an sc-Fredholm section over a polyfold ambient space $B$ whose zero set orbit space is identified with the coarse moduli space. By \ref{INVERTIBILITY}, polyfold theory becomes even more canonical: One can apply forgetful construction \ref{FORGETFULCONSTRUCTION} to contruct a same Kuranishi structure, up to chart-equivalence (in particular, R-equivalence), for different weights or function spaces; so by polyfold--Kuranishi correspondence, different choices of function spaces give R-equivalent sc-Fredholm structures $(X, f:B\to E, \psi)$. (Essentially one form a fiber product in the sense of \cite{DI} for two different $f$ over a level-1 good coordinate system of the Kuranishi structure, producing a common refinement sc-Fredholm structure.) So construction of an R-equivalence class of an sc-Fredholm section $f$ on a moduli space $\mathcal{M}(\mathfrak{D})$, where $\mathfrak{D}$ is a necessary geometric auxiliary data, is canonical after just fixing an admissible\footnote{See some discussion of \cite{Polyfoldanalysis} 4.2. Although not explicitly given, a general criterion satisfying which candidates of gluing profiles should work in polyfold theory in the place of the exponential gluing profile can be extracted from HWZ's work in principle.} gluing profile.
\end{enumerate}
\end{remark}

\section{Different views of orbifolds/groupoids (and intermediates) and Morita-equivalence}\label{ORBIFOLD}

There are several approaches to keep track of symmetry, and different virtual theories are phrased using different ways. In this section, we will describe how to translate between different descriptions. For simplicity, we restrict to the effective group action case.

In the sc-setting, we have discussed ep-groupoids.

In finite dimensions, we have theories of \'etale-proper Lie groupoids, and (representatives of) orbifolds.

For orbifolds, we can just take $E_p=0, p\in X$ (thus with open coordinate change $(\phi_{pq},V_{pq})$) in \ref{KURANISHISTRUCTURE}, and note that coordinate change $V_{pq}\subset V_q$ from a chart $(G_q, V_q, \psi_q)$ to $(G_p, V_p, \psi_p)$ for $q\in X_p:=\psi_p(\underline{V_p})=\psi_p(V_p/G_p)$ just needs to have the following property $q\in X_q|_{V_{pq}}:=\psi_q(\underline{V_{pq}})$ and nothing more. So $V_{pq}$ can be really small. Two representatives are (chart-) equivalent (now same as Morita-equivalent) if they are commonly chart-refined by another such structure, and such an equivalence class is an \textbf{orbifold}. Good places to learn about this formulation of orbifolds are appendix of \cite{CRuan} and \cite{ALR}. A special case of one argument in the process of proving \ref{FOOOISDIR} shows that for $X$, $\mathcal{O}_X:=(X,\{(G_p, V_p,\psi_p)\}_{p\in X},\{(\phi_{pq},V_{pq})\}_{q\in X_p})$ has a refinement $\mathcal{O}'_X$ (with charts of the same dimensions as that of $\mathcal{O}_X$, with extended coordinate changes but both charts and coordinate changes are still of the form of global quotients), such that $\mathcal{O}'_X$ satsifies $X$-maximality (same as $X$-maximality), so in particular, $X_q|_{V'_{pq}}=X_q|_{V'_q}\cap X_p|_{V'_p}$ (and coordinate changes are compatible up to group actions as usual).

There is also a cover definition of an orbifold representative which is not parametrized pointwise as in $\mathcal{O}_X$ which can be thought of as a less stringent version of a corresponding good coordinate system of $\mathcal{O}_X$. There are two flavors: The Satake's definition is a paracompact second-countable topological space $X$ with a collection of orbifold charts $\{(G_\alpha, V_\alpha, \psi_\alpha: V_\alpha/G_\alpha\to X_\alpha)\}_{\alpha\in \Lambda}$ such that $X=\cup_{\alpha\in\Lambda} X_\alpha$, and to impose compatibility, for all $z\in X_\alpha\cap X_\beta\not=\emptyset$, there exists $\gamma\in \Lambda$ (depending on $z, \alpha, \beta$) such that $z\in X_\gamma\subset X_\alpha\cap X_\beta$ and there exist open equivariant embeddings from $(G_\gamma, V_\gamma, \psi_\gamma)$ into respective $(G_\alpha,V_\alpha,\psi_\alpha)$ and $(G_\beta, V_\beta, \psi_\beta)$ with both domains of coordinate changes being $V_\gamma$. An \textbf{orbifold} is just the maximal atlas, with the usual notion of compatibility of atlases. Another definition is mentioned in Thurston's online book \cite{Thurston}, which requires that $X_\alpha\cap X_\beta$ is $X_\gamma$ for some $\gamma\in \Lambda$, and the inclusions into $X_\alpha$ and $X_\beta$ is induced by some equivariant embeddings. If one says that $X_\alpha$ is uniformized by $(G_\alpha,V_\alpha,\psi_\alpha)$, then such an orbifold representative definition of Thurston can be summarized as an open cover consisted of uniformized open sets which are closed under finite intersections, and one can also go to its maximal atlas which is an orbifold in the sense of Thurston.

There is also another way to describe the coordinate change between orbifold charts as in \cite{PolyfoldGW} and \cite{PolyfoldSurveyNew}, and also used at places in \cite{DI} e.g. 7.3 (1). This is very close to \'etale-proper Lie groupoid and Thurston's orbifold representative definition. Denote an orbifold chart as $C_\alpha:=(G_\alpha, V_\alpha, \psi_\alpha)$, then we can define coordinate change between $C_\alpha$ and $C_\beta$ as $M(C_\alpha, C_\beta):=\{(z,g,w)\;|\;z\in V_\alpha, w\in V_\beta, g: z\to w\}$. Here $g$ is an arrow identifying $x$ to $z$, and its origin can be understood in three ways and with the fourth way arising from the normalization in the paragraph following (3).
\begin{enumerate}
\item (\cite{PolyfoldGW}, \cite{PolyfoldSurveyNew}) If $\{C_\alpha\}_{\alpha\in \Lambda}$ are good uniformizers for $X$ where the topological space $X$ is a collection of isomorphism classes of certain objects. Namely, for any smooth parametrization of objects whose isomorphism classes are in $X$, $\tau\mapsto x_\tau$ based at $x_{\tau_0}$, fix an isomorphism $\phi_0: x_{\tau_0}\to z_0$ and $z_0\in V_\alpha$  for some $\alpha\in\Lambda$, then the \textbf{good uniformizer property} says there exists a unique smooth $\tau\mapsto z(\tau)\in V_\alpha$ and a unique smooth family of isomorphisms $\phi_\tau: x_\tau\to z(\tau)$ such that $z(\tau_0)=z_0$ and $\phi_{\tau_0}=\phi_0$. In particular if $\tau\mapsto x_\tau$ is provided by $y\in V_\beta$, then the good uniformizer property provides coordinate changes between charts. The above $g$ in $M(C_\alpha, C_\beta)$ can be understood as an isomorphism between two objects. $M(C_\alpha, C_\beta)$ is a manifold due to the good uniformizer property.
\item (see e.g. \cite{PolyfoldIII}) If $C_\alpha=(G_\alpha, V_\alpha, \psi_\alpha)$ arises as a natural representation of stabilizer group for an \'etale-proper Lie groupoid, then $M(G_\alpha,G_\beta)=s^{-1}(V_\alpha)\cap t^{-1}(V_\beta)$ an open subset of the morphism space of the groupoid, and $g$ is just a morphism/arrow from $z$ to $w$.
\item (see e.g. certain parts of \cite{DI}) If $C_\alpha=(G_\alpha, V_\alpha, \psi_\alpha: V_\alpha/G_\alpha\to X_\alpha)$ is a chart from Thurston's orbifold representative, then for $M(C_\alpha, C_\beta)$ would be empty unless $X_\alpha\cap X_\beta\not=\emptyset$ in $X$. In that case there exists $C_\gamma$ with $X_\gamma=X_\alpha\cap X_\beta$ and equivariant embeddings $\phi_{\alpha\gamma}$ and $\phi_{\beta\gamma}$ from $V_\gamma$ into $V_\alpha$ and $V_\beta$ respectively. Then $g$ in $M(C_\alpha, C_\beta)$ is $g_\beta\cdot \phi_{\beta\gamma}\circ(\phi_{\alpha\gamma}|_{\phi_{\alpha\gamma}(V_\gamma)})^{-1}\cdot g_\alpha$ for some $\gamma_\alpha\in G_\alpha, \gamma_\beta\in G_\beta$.
\end{enumerate}

There is yet another way to deal with orbifolds, which is an abstraction of the orbifold setting I learned from discussion with Shaofeng Wang whose work (Banach setting) is based on work of Lu-Tian \cite{LuTian}. This viewpoint is also similar to McDuff's treatment \cite{MNEW} (finite dimensional setting). Suppose there are two charts $C_\alpha$ and $C_\beta$ overlapping $X_\alpha\cap X_\beta\not=\emptyset$ and denote $\varphi_\alpha:=\psi_\alpha\circ\text{quot}_\alpha: V_\alpha\to X_\alpha$. We can take the fiber-product $V_{\{\alpha,\beta\}}:=V_{\alpha}\;_{\varphi_\alpha}\times_{\varphi_\beta} V_\beta$ which is not a manifold, but there is a \textbf{normalization} $\tilde V_{\{\alpha,\beta\}}$ which abstractly separates out branches (for example using nearby information), and this exactly agrees with $M(C_\alpha,C_\beta)$ in (3) above. Namely $V_{\{\alpha,\beta\}}$ is not a manifold because we forget about $g$ in $(w,z)$ for $w$ and $z$ connected by $g$. One can do this repeatedly to get $V_I$ and $\tilde V_I$ for any subset $I\subset \Lambda$ for a locally finite cover $\{C_\alpha\}_{\alpha\in \Lambda}$ of $X$ (and an element in $\tilde V_I$ will correspond to $(z_1,g_{12}, z_2, g_{23}, \cdots, g_{n-1 n}, z_n)$ for $|I|=n$). The difference between perturbations of sections in bundles (global Banach bundle or Kuranishi (type) structure) between setting (3) above and this approach in the current paragraph is that (3) uses $M(C_\alpha, C_\beta)$ as coordinate changes to get globally invariant multisectional perturbation on the object space level, and the approach in the current paragraph using inclusion order relation on the power set of $\Lambda$ to get a globally compatible single valued perturbation on the morphism space level $\tilde V_I$ (equivalently $M(\{C_\alpha\}_{\alpha\in I})$) which induces an invariant weighted (unbranched) pseudo-cycle. In the latter approach, because of the normalization, points on the singular set are counted more than once thus it requires the singular set is of codimension at least 2 (so it is harmless for pseudocycles), which is the case in the finite dimensional oriented orbifold case and assumed in the hypothesis in Lu-Tian's \cite{LuTian} and Wang's recent work on the abstract virtual Euler class. This also explains why in the Satake's or Thurston's orbifold case, at the overlap region group action $G_\gamma$ is isomorphic to subgroup of $G_\alpha$ and $G_\beta$ (working on the orbit space level via the object space), but in the description in the current paragraph the group acting on the overlap $\tilde V_I$ is the product $\times_{\alpha\in I}G_\alpha$ (can be thought of as working on the morphism space level, also clear from the descripion of form of $g$ in (3) above). 

We summarize various ways of describing orbifold representatives:

\begin{enumerate}[(a)]
\item \'etale-proper Lie groupoid,
\item HWZ's ep-groupoid in polyfold setting,
\item orbifold representative $\mathcal{O}_X$ (indexed pointwise) in the spirit of FOOO's Kuranishi structure,
\item orbifold representative $\mathcal{O}_X$ (indexed pointwise) with ($X$-) maximality,
\item Satake's orbifold representative (indexed as a cover),
\item Thurston's orbifold representative (indexed as a cover which is closed under finite intersections) in the spirit of ($X$-) maximal good coordinate system,
\item HWZ's good uniformizers and $M(C_\alpha, C_\beta)$ in (1) above,
\item HWZ's `island' description of an \'etale-proper Lie groupoid/ep-groupoid (after chart-refining and take natrual representations of the stabilizer groups) in (2) above,
\item orbifold representative with coordinate change between chart described in one go in $M(C_\alpha,C_\beta)$ in (3) above, and
\item Lu-Tian's description via power set and normalization (in Banach bundle case) and McDuff's description (in Kuranishi atlas case).
\end{enumerate}

We describe how to transfer/translate among orbifold settings (without concerning change of finite dimensional reduction descriptions, forgetful construction and globalization construction, which is the subject of the next section \ref{UNIFICATION}).

$(a)\leftrightarrow (b)$ in finite dimensions: (a) and (b) agree in finite dimensions.

$(a),(b)\to (c), (d), (e), (f):$ Take natural representations of stabilizer groups as in \cite{PolyfoldIII}, we can get an orbifold representative described.

$(c)\to (d):$ This is a sub-argument in proving \ref{FOOOISDIR}. The statement is that there exists a $X$-maximal (same as maximal) $\mathcal{O}'_X$ which refines (of the same dimension of course, chart-refinement on charts, but extending coordinate changes of) $\mathcal{O}_X$ in (c).

$(c),(e)\;(\text{hence easier cases}\;(d),(f)) \to (a)$: Use almost free $O(n)$-action on frame bundle which is a manifold (learned from Ono), then take quotient groupoid (via local slices) as a Morita-refinement.

$(c)\leftrightarrow (e)$: (c) can be regarded as a special cover in setting of (e), and (e) induces (c) in a similar but easier way as a good coordinate system induces a Kuranishi structure.

$(e)\to (f)$: $(e)\to (c)\to (d)$, then take a subcover which refines the cover given in (e), then we have obtained an orbifold representative (f) which chart-refines $(e)$.

$(d)\to (c), (f)\to (e)$: special cases.

$(g)\leftrightarrow(h)\leftrightarrow(i)$: The same structure (just different origins).

$(a),(b)\to (h)$: Via natural representations of stabilizer groups.

$(h)\to (a), (b)$: Take the disjoint union of $\{V_\alpha\}_{\alpha\in\Lambda}$ for the object space, and the disjoint union of $\{M(C_\alpha,C_\beta)\}_{(\alpha,\beta)\in\Lambda\times\Lambda}$ for the morphism space, with the natural structure maps induced by $g$.

$(f)\to (i)$: Explained in (3).

$(i)\to (d)$: By inducing.

$(i)\leftrightarrow (j)$: Explained in the paragraph following (3).

In the next section, we focus on how to organize different systems of finite dimensional reductions among various approaches and in relation to sc-Fredholm sections with compact zero sets in the orbit spaces, with translation among descriptions of orbifolds understood in the background (sometimes briefly commented on).

\section{Other virtual theories and pairwise identifications up to R-equivalences. Part A: Abstract perturbative structures}\label{UNIFICATION}

The unification diagram below indicates how to go from a virtual theory to another (except last three rows on the rightmost column which are algebraic constructions). Following the arrow direction and returning to the starting point, the effect is identity up to appropriate R-equivalence at that point. Any identification between two (oriented) structures intertwines perturbation theories as well as algebraic constructions chosen.

We could have just indicated a minimal number of connections and then compose to connect between any two, but the point is that the direct identification/correpondence is just as easy and clear. The sc-Fredholm sections/structures will always have compact zero sets in the orbit spaces (for simplicity) and we will omit mentioning this in the descriptions below.

\begin{figure}
\begin{tikzpicture}[scale=1]

%\hspace{-3.9 cm}
%\path (-3.9 cm, 0 cm);

\matrix(m)[matrix of math nodes, row sep=0.5em, column sep=1.5 em,
text height=1 ex, text depth=0.25ex]{
 \text{non-perturbative}     & \text{Kuranishi structure} & \text{Kuranishi sp./d-orbifold}\\
\text{theory}             & \emph{Fukaya-Oh-Ohta-Ono}                & \emph{Joyce}\\
  & &\\
  &  &\\
  &  &\\
  &  &\\
  &  &\\
  &  &\\
                      & \text{GCS}                 & \text{singular homology}\\                     
                      & \emph{Fukaya-Oh-Ohta-Ono}                &\\
 & &\\
 & &\\
  &  &\\
  &  &\\
  &  &\\
  &  &\\
\text{polyfold Fredholm structure} &\text{level-1 GCS (or K.)} & \text{de Rham cohomology}\\
 \emph{Hofer-Wysocki-Zehnder}                        &    \emph{Yang} & \\
 & &\\
 & &\\
  &  &\\
  &  &\\
  &  &\\
  &  &\\
 \text{Kuranishi atlas} & \text{virt. orbibundle section} & \text{de Rham homology}\\
  \emph{McDuff-Wehrheim}      &   \emph{Chen-Li-Wang}                             & \emph{Irie}\\
&&\\
&&\\
&&\\
&&\\
&&\\
&&\\
\text{implicit atlas/VFC}&  \text{geom. pert. via Don.}&\\
\emph{Pardon} &  \emph{Cieliebak-Mohnke}&\\
&  \text{and}\;\;\emph{Gerstenberger}&\\};

\path[thick, >= angle 60, ->]
(m-17-1)  edge node [left=0.3 cm, above=0cm]{5.2 PFH \emph{Yang}} (m-2-1)
(m-9-2)   edge node [below, sloped]{} (m-1-1)
(m-17-1)  edge node [above=0.1cm]{2.7} (m-2-2)
(m-17-2)  edge node [above, sloped]{2.8} (m-17-1)
(m-17-2)  edge node [below, sloped]{} (m-17-3)
(m-17-2)  edge node [below, sloped]{} (m-9-3)
(m-10-2)  edge node [above, sloped]{CF-pert.} (m-17-3)
(m-9-2)  edge node [below, sloped]{} (m-9-3)
(m-10-2)  edge node [below, sloped]{} (m-25-3)
(m-25-2)  edge node [below, sloped]{} (m-17-3)
(m-17-2)  edge node [below, sloped]{} (m-25-3)
%(m-17-2)  edge node [below, sloped]{} (m-25-1)
(m-17-2) edge node [below, sloped]{}(m-1-3)
(m-1-3)  edge [bend right=12] node [above,sloped]{5.1 KH/virtual chain \emph{Joyce}} (m-1-1)
(m-17-2) edge node [below, sloped]{} (m-33-1);

\path[thick, >= angle 60, <->]
(m-2-2) edge node [below, sloped]{}(m-9-2)
(m-10-2) edge node [below, sloped]{}(m-17-2)
(m-18-2) edge node [below, sloped]{}(m-25-2)
(m-18-1) edge node [below, sloped]{}(m-25-1)
(m-25-1) edge node [below, sloped]{}(m-25-2)
(m-18-1) edge node [below, sloped]{}(m-25-2)
(m-1-2) edge node [below, sloped]{}(m-1-3)
(m-18-1)  edge node [below=.8 cm, right=0 cm]{? \ref{OPENORTODO}} (m-33-2)
(m-25-1) edge node []{} (m-10-2);
\end{tikzpicture}

\caption[Unification diagram: Identifications of structures and algebraic constructions in virtual theories]{Unification diagram: Identifications of structures and algebraic constructions in virtual theories. (GCS stands for good coordinate system, K. stands for Kuranishi structure, sp. stands for space, pert. stands for perturbation, geom. pert. via Don. means geometric perturbation via Donaldson divsior, and virt. stands for virtual.)}\label{UNIFDIAG}

\end{figure}

\subsection{McDuff-Wehrheim's Kuranishi atlas and identifications with sc-Fredholm section, FOOO's Kuranishi structure}

References for Kuranishi atlas are \cite{MW} and \cite{MNEW}. After completion of this article, three more McDuff-Wehrheim papers appeared: \cite{MW1} and \cite{MW2} are updates of \cite{MW}; and \cite{MW3} follows up \cite{MNEW} and deals with the non-trivial stabilizer case. Paper \cite{MW1} contains brief comments and remarks on various approaches of virtual theory, which can complement this article. In a Kuranishi atlas, Kuranishi charts are indexed by (a subset of) the power set of a finite set $J$ indexing a finite cover (without coordinate changes). This finite cover consists of fiberwise independent\footnote{The independence is achieved by a slight perturbation of the data using the underlying geometry.} data for local finite dimensional reductions, then each Kuranishi chart is produced from the local finite dimensional reduction of the summed data in the overlap (if there is an ambient space, or constructed from the underlying geomertic setting). Using the knowledge that the index set for the collection of final local finite dimensional reductions being subsets of a given $J$ with order determined by the inclusion, one can shrink using the dual polyhedron see \cite{MW} 7.1.1 and 7.1.2\footnote{This process is called reduction there, but in order not to confuse with finite dimensional reductions discussed extensively in this paper, we will call it shrink along the dual polyhedron (of intersection pattern of the cover).} and obtain that two charts intersect if and only if one index is a subset of another index, namely a good coordinate system is obtained using the knowledge of how charts arised from summing and then by shrinking accordingly. (This contrasts with the mechanism for obtaining a good coordinate system from a Kuranishi structure.)

To go from an sc-Fredholm section $f:B\to E$ to a Kuranishi atlas, we proceed along the lines of \ref{FORGETFULCONSTRUCTION} with appropriate changes. Since McDuff-Wehrheim uses finite dimensional reductions similar to \ref{STABILIZATION} in abstract terms rather than \ref{PFDFDRED}, we will use the former, although either can be used (and are really the same by \ref{SAMEFDRED} (1)).

\begin{enumerate}
\item Similarly to \ref{FORGETFULCONSTRUCTION}(1), we choose presummable data $\{(L_{x_i}, W_{x_i})\}_{x_i\in I}$, but we do not need to choose $U_{x_i}$ here. Recall from \ref{PFDFDRED} and \ref{STABILIZATION} that $L_{x_i}\to W_{x_i}$ for each $x_i\in I$ can be assumed to be a trivial bundle $\psi_{x_i}(W_{x_i}\times N_{x_i})$ and is $G_{x_i}$-invariant, $f$ is in good position to $L_{x_i}$, and $f^{-1}(0)\subset \bigcup_{x_i\in I}W_{x_i}$. The collection is globally invariant.
\item Here we have a few choices, and we explain one here and comment another two ways in the remark below. Recalling \ref{INDDOESNOTM}, we can do finite dimensional reduction \ref{STABILIZATION}, despite $L_{x_i}$ might be dependent in an uncontrolled way in $E$, the result of this choice is R-equivalent to results of the other two ways. Denote $J:=\underline{I}$. For any $K\subset J$, let $W_K:=\emptyset$ if $\cap_{\underline{x_i}\in K}\underline{W_{x_i}}=\emptyset$ in $\underline{B}$. Recall, $\underline{\;\;}$ means taking the quotient under group action to become a subset in the orbit space $\underline{B}$; and for $z, z'\in \underline{x_i}$, we have $\underline{W_z}=\underline{W_{z'}}$ due to global invariance. If the intersection is non-empty, we can choose $x^K_i$ for each $\underline{x_i}\in K$ such that $\cap_{x^K_i} W_{x^K_i}\not=\emptyset$, and denote it by $W_K$, and two choices will give rise to $W_K$'s identified sc-diffeomorphically under the morphism of $B$. Denote $N_{K}:=\oplus_{x^K_i} N_{x^K_i}$. Then by \ref{STABILIZATION} we have $\mathbf{f}^{N_K}$, and its finite dimensional reduction is a Kuranshi chart $pr_2|_{(\mathbf{f}^{N_K})^{-1}(0)}=:s_K:V_K\to N_K$.
\item Denote $\mathbb{J}:=\{K\subset J\;|\; W_K\not=\emptyset\}$. Let $K_1, K_2\subset \mathbb{J}$ such that $K_1\subset K_2$, then $N_{K_1}\subset N_{K_2}$, and using $\psi_{x^{K_j}_i}$ and the morphism structure of $E$, we have embedding of $\mathbf{f}^{N_{K_1}}|_{W_{K_2}\times N_{K_1}}$ into $\mathbf{f}^{N_{K_2}}$ in the sense of a local version of \ref{SCFREDEMBEDDING}. From this we have a coordinate change $s_{K_1}$ into $s_{K_2}$ with the domain of coordinate change $V_{K_2K_1}:=(\mathbf{f}^{N_{K_1}}|_{W_{K_2}\times N_{K_1}})^{-1}(0)$.
\item To describe the group actions, one uses section \ref{ORBIFOLD} (j) ($V_K$ can also be helpfully regarded as subset of $M(\{(G_{x^K_i}, W_{x^K_i}\times N_K)\})$ in \ref{ORBIFOLD} (i)). This conforms with McDuff's orbifold version of Kuranishi atlas in \cite{MNEW}.
\end{enumerate}

\begin{remark}\label{AMAZINGLYTHESAME}
In (2) above, one can also pass to $\mathbf{f}:\mathbf{B}\to\mathbf{E}$ with corresponding $\mathbf{L}_{x_i}, x_i\in I$ as in \ref{FORGETFULCONSTRUCTION} (ii). Now $\mathbf{L}_{x_i}$ are independent, one can then use either \ref{STABILIZATION} as above or use \ref{PFDFDRED}, to get a Kuranishi atlas. One can find out that two results agree with the above result exactly on the nose, essentially due to \ref{SAMEFDRED} and the settings.
\end{remark}

After shrinking along the dual polyhedron of intersection patterns of $V_K, K\in\mathbb{J}$, one obtains a good coordinate system $s_K|_{V'_K}$ in the sense of \ref{GCS} (after transferring orbifold viewpoints as in section \ref{ORBIFOLD}), because after this shrinking, whenever two footprints $X_K|_{V'_K}$ (the images in $X$ of $\underline{(s_K|_{V'_K})^{-1}(0)}$ in the orbit space under homeomorphisms) intersect, then the index set of one has to be subset of the other, therefore the embedding exists. Thus, we can talk about a refinement of a shrunken Kuranishi atlas just as in \cite{DI}. The construction from an sc-Fredholm section into an R-equivalence class of such a good coordinate system is well-defined.

To convert a Kuranishi atlas into an sc-Fredholm section is easy, as a Kuranishi atlas can shrink into a good coordinate system, and one just apply the globalization construction in subsection \ref{GLOBALIZATION}.

One could go between Kuranishi structures and Kuranishi atlases via sc-Fredholm sections, but the direct conversion is also illuminating.

First we make a remark:

If one remembers the subcollection $I_x\cong\underline{I_x}$ of $\underline{x_i}\in J$ from which $E_x$ is originated from, then apply \ref{GCSWOIND} to the resulting $\mathcal{K}_X$ which is always maximal and topological matching. Because of the last condition in \ref{GCSWOIND}, whenever $X_y|_{\tilde U_y}\cap X_x|_{\tilde U_x}\not=\emptyset$, we have $y\in X_x$ or $x\in X_y$ in $\mathcal{K}_X$, namely $I_y\subset I_x$ or $I_x\subset I_y$. Thus it behaves like the shrinking from a Kuranishi atlas. However, the point in Kuranishi structure theory discussed in subsection \ref{KTHEORY} is that we do not need to remember this information and can still construct a good coordinate system from other considerations, and this becomes more flexible in dealing with fiber products relevant to moduli spaces with coners and finer structures such as SFT or Lagrangian Floer theory.

A Kuranishi atlas gives rise to a good coordinate system by shrinking which in turn induces a Kuranishi structure, see \ref{ALLTHESAME}.

To go from a Kuranishi structure to a Kuranishi atlas, first apply \ref{FOOOISDIR} and \ref{GCSWOIND}, one has a good coordinate system with global group actions on charts and coordinate changes and with the maximality and topological matching conditions. After one has maximality of coordinate changes, one can construct Kuranishi charts for finite intersections of $X_x$'s induced from the chart with the largest order defined in \ref{GCSWOIND} (for example), which is already a Kuranishi atlas\footnote{After changing the orbifold description, or already so if doing the Kuranishi charts associated to the intersections in this description directly.}. It is interesting to note that the tripling process in \cite{DI} (essential for constructing level-1 `fiber product'-like structure) plays the same role as taking finite intersection and then shrink along the dual polyhedron, although tripling is not needed in \cite{DI} for perturbation theory on a single good coordinate system (and if one compares just via Kuranishi cobordisms, then it is also not needed to show perturbation result being independent of choices made).

\subsection{Chen-Li-Wang's virtual obibundle section, and identifications with sc-Fredholm section, FOOO's Kuranishi structure and MW's Kuranishi atlas}\label{VIRTUALSECTION}

Virtual orbifold theory in \cite{CLW} is based on \'etale-proper Lie groupoid (actually a weak version of it which roughly says that coordinate changes are continuous but when restricting to the part of the object space consisting of smooth geometric maps are smooth).

Going from an sc-Fredholm section to a section in a virtual orbifold bundle, one applies a nice trick in \cite{CLW} 3.1.2 due to Bohui Chen.\footnote{Note that the globalization construction plus this trick will not provide a simpler proof of the existence of level-1 structure as the globalization uses a level-structure.} One needs (sc-) smooth cut-off function, so let us restrict to the case that the sc-Fredholm section $f:B\to E$ where $B$ is based on retracts $(O=r(U),E)$ with 0-level $E_0$ being separable Hilbert spaces (already so for moduli space constructions anyway). The rough idea is after applying a clever cut-off of \ref{INDDOESNOTM} ($L_{x_i}$ does not require to be indepedent as we have observed conceptually and used before), one applies finite dimensional reductions, and get Kuranishi charts indexed by power set of a covering as well, but also the projection between bases of charts with containment between index sets as well as projections between bundles come automatically for free. Here although one perturbs the section, it is still an R-equivalent sc-Fredholm problem, see \ref{DIMREDEX}. On the other hand, it might be more difficult to obtain a similar structure this way for moduli spaces with corners where the boundary of the structure in the inductive step is a union of fiber products of moduli spaces already constructed. So to carry out the conversion from an sc-Fredholm section to a virtual section, the only new input for this method is a local finite dimensional reduction for sc-Fredholm section as in \ref{STABILIZATION}. Then the rest is verbatim as in \cite{CLW} 3.1.2 and I will not reproduce their punchline here.

To go from a section in a virtual bundle to an sc-Fredholm section, one just applies the procedure in subsection \ref{GLOBALIZATION}. The structure in a virtual bundle is a sub-structure of level-1 structure if one organizes the projection/submersion data between bundles and bases into coordinate changes between charts. So the same construction will work, see details of subsection \ref{GLOBALIZATION} in the forthcoming \cite{DIII}.

To go from an FOOO's Kuranishi structure to a section in a virtual bundle, one gets a level-1 good coordinate system with global group actions for charts and coordinate changes after invoking \ref{FOOOISDIR} and \ref{LEVEL1CATEXIST}, then induces Kuranishi charts from finite intersections of charts in a level-1 good coordinate system as in the last subsection, which is still level-1. Then reorganize the level-1 data of coordinate changes into projections between bundles and bases and forget the extra information, one will get a section in a virtual bundle indexed by the power set of the index set of the level-1 good coordinate system. One also needs to convert the orbifold description into an \'etale-proper Lie groupoid one as in section \ref{ORBIFOLD}.

To go from a section in a virtual bundle to an FOOO's Kuranishi structure, one just forgets about submersions/projections data and do a shrinking along the dual polyhedron, and obtains a good coordinate system, which induces a Kuranishi structure.

To go from a section in a virtual bundle to Kuranishi atlas, one just forgets the projection data.

To go from a Kuranishi atlas to a section in a virtual bundle, one first shrink along the dual polyhedron to get a good coordinate system which induces a Kuranishi structure, then applies \ref{FOOOISDIR} to get a Kuranishi structure with the maximality and topological matching condition, and then apply \ref{LEVEL1CATEXIST} to get a level-1 good coordinate system, then take the induced level-1 good coordinate system indexed by a power set after taking finite intersections, then organize the level-1 structure into a virtual bundle to obtain a virtual section.

\begin{remark} Although a level-1 structure and virtual bundle have similar data, but they come from different motivations and origins. A level-1 structure can be equipped to any Kuranishi structure (up to refinement) after much hard work in \cite{DI} and it is useful for constructing inductive perturbations and do roof constructions and will work for moduli spaces with corners and finer structures of `the boundaries being fiber products'; whilst the virtual bundle is obtained for free by a clever trick to a perturbed Fredholm section, but might need non-trivial generalization for moduli spaces with corners and it requires existence of cut-off functions and thus cannot be applied to all sc-Fredholm sections.
\end{remark}

\subsection{Joyce's d-orbifold and Joyce's Kuranishi space and relation with FOOO's Kuranishi structure and others}

Joyce has two equivalent space structures: a d-orbifold \cite{Joyce} is a derived geometric object and can be thought of as the orbifold version of 2-categorical truncation of a derived manifold of Spivak; and a Kuranishi space \cite{JK} is a cover of a Hausdorff secound countable space $X$ by Kuranishi neighborhoods together with a collection of 1-morphisms between each pair of neighborhoods intersecting in footprints in each direction (despite two charts having different dimensions in general), such that 1-morphisms are relaxed versions of coordinate changes, invertible up to 2-morphisms, and compatible up to 2-isomorphisms, where those 2-isomorphisms are also a part of the data. Joyce has focused more on the second approach recently. Before diving into a Kuranishi space formulated in a 2-categorical and orbifold setting, \cite{JK} introduces a 1-category truncation based on manifolds, called $\mu$-Kuranishi structure, where coordiante changes are considered up to equivalence \cite{JK} definition 2.3 (2.1), and a key piece of data saying two coordinate changes are equivalent can be considered as and will be upgraded into a 2-isomorphism. We will discuss this key notion of equivalence below from the viewpoint of a level-1 structure. For a Kuranishi space, the coordinate changes are generalized maps formulated in bibundle form, see \cite{JK} 4.2 and 2-morphism is defined in \cite{JK} 4.3, considered up to equivalence \cite{JK} 4.3 (4.2), also see \cite{2CAT} 3.1. A Kuranishi space is defined in \cite{JK} 4.17, \cite{2CAT} 4.1.

The exact expression saying two 1-morphisms are 2-isomorphic might not be easy to absorb at the first sight, and I will try to illustrate it below.

The reason for the form of equivalence relation (2-isomorphisms) between coordinate changes in \cite{JK} (2.1) is as follows: We can find a left inverse for a coordinate change $(\phi_{xy}, \hat{\phi}_{xy}, U_{xy})$ up to some precompact shrinking of charts (more precisely, a left inverse for the restricted coordinate change $(\phi_{xy}, \hat{\phi}_{xy}, U'_{xy})=(C_y|_{U'_y}\to C_x|_{U'_x})$), e.g., $$(\phi_{xy}^{-1}|_{\phi_{xy}(U'_{xy})}\circ\pi_{xy}, \hat\phi_{xy}^{-1}|_{\hat\phi_{xy}(E_y|_{U'_{xy}})}\circ\hat{\pi}_{xy}, W_{xy})$$ defined using the data in a level-1 structure \ref{LEVELONECC} chosen for this restricted coordinate change provides such a left inverse, and it is a coordinate change in the sense of \cite{JK} 2.3. However, the composition in the other order is of course not the identity if there is a dimension jump, and the form of the equivalence is there to absorb the difference of this to the identity, as well as the ambiguity of choices of left inverses, and the form of dependence on the Kuranishi section in the equivalence means that these identifications happen only in the direction slice restricting to which the Kuranishi section is diffeomorphic. To unravel the meaning a bit, letting $\iota_{xy}: \phi_{xy}(U'_{xy})\to W_{xy}$ and $\hat\iota_{xy}:\hat\phi_{xy}(E_y|_{U'_{xy}})\to E_x|_{W_{xy}}$, we want to show that $(\iota_{xy}\circ\pi_{xy}:W_{xy}\to W_{xy},\hat\iota_{xy}\circ \hat\pi_{xy}: E_x|_{W_{xy}}\to E_x|_{W_{xy}}, W_{xy})$, which is the composition of the above left inverse and the restricted coordinate change (in the sense of FOOO) in the other order, as a coordinate change in Joyce's sense is equivalent to the identity coordinate change.

To verify for the base map, we want to show that for any smooth function $h:W_{xy}\to\mathbb{R}$, we have $$h\circ id=h\circ \iota_{xy}\circ\pi_{xy}+\Lambda((s_x|_{W_{xy}})\otimes (\iota_{xy}\circ\pi_{xy})^\ast dh)+O((s_x|_{W_{xy}})^2)\footnote{This is equivalent to \cite{JK} 2.3 (2.1) with terminology defined in 2.1 (v), as the equivalence is an equivalence relation and hence in particular symmetric.}$$ for some $\Lambda$ (not depending on $h$ for the case below). Since $\pi_{xy}$ is a submersion, we can choose local coordinates $(x,v)$ with the base and fiber coordinate as $x$ and $v$. $h(x,v)-h(x,0)=D_2 h(x,0)v+ \tilde h(x,v)(v\otimes v)$ for a smooth section $\tilde h$ of $((\ker\pi_{xy})^\ast)^{\otimes 2}\to W_{xy}$. Let $\nabla^{\text{vert}} s_x: \ker d\pi_{xy} \to (E_{xy}|_{W_{xy}})/\tilde E_{xy}$ be defined using some connection on $E_{xy}|_{W_{xy}}$\footnote{Different choices have the same effect due to the error term.} with $\tilde E_{xy}$ from the level-1 structure and it is a bundle isomorphism over $W_{xy}$\footnote{Due to the transversality condition in \ref{LEVELONECC} (3).}, and let $\text{quot}: E_x|_{W_{xy}}\to (E_{xy}|_{W_{xy}})/\tilde E_{xy}$ be the bundle quotient map. Then choose $\Lambda$ to be $(\nabla^{\text{vert}} s_x)^{-1}\circ \text{quot}$, and then $\Lambda(s_x|_{W_{xy}}\otimes \alpha)(\cdot)$ is $(x,v)\mapsto (x,v,0,v)\alpha$ up to an order-two error, where the fourth coordinate on the right hand side is the direction of $(\ker d\pi_{xy})_{(x,v)}\subset T_{(x,v)}W_{xy}$. In terms of the above coordinates, $\iota_{xy}\circ\pi_{xy}$ is $\pi: (x,v)\mapsto (x,0)$, and $(\Lambda((s_x|_{W_{xy}})\otimes \pi^\ast dh))(x,v)=D_2 h(x,0) v$ up to an order-two error, because on $\ker d\pi_{xy}$ the coordinate of $x$ is fixed. The equivalence of the bundle map is verified using the same $\Lambda$. This shows that this left inverse\footnote{Hence all the left inverses by the remark in the above paragraph.} to an FOOO coordinate change is an inverse up to this equivalence.

To summarize the above discussion, the data of a level-1 good coordinate system consists of coordinate changes in Joyce's sense in both directions between any two charts whose coverages/footprints intersect (coordinate changes and projections/submersions in the level-1 structure going the opposite way), and the collection of these coordinate changes are compatible up to 2-isomorphisms (including up to group actions), and all are invertible up to 2-isomorphism as shown above. The other properties (besides invertibility up to 2-isomorphism showed above) follow from projections/submersions being left inverses and compatibility of level-1 coordinate changes. Thus, a level-1 good coordinate system contains data for the underlying good coordinate system to be a Kuranishi space. One can compare this with Joyce's approach in \cite{JK} 4.72. One might think of Joyce's equivalence/2-isomorphisms between coordinate changes conceptually as a R-equivalence of (coordinate changes as) local maps.

The upshots for the 2-category formulation in \cite{JK} are that morphisms between two spaces are local as a sheaf and the fiber product will have the universal property, and although these two properties have not been used in symplectic geometry so far, they are essential in Joyce's approach to moduli spaces as he wants to mimick the approach in algebraic geometry (functor of points) and construct the moduli 2-functor.

Joyce claimed in \cite{2CAT} 4.19 that he can construct an FOOO's Kuranishi structure from his Kuranishi space, which is important and seems non-trivial to have functorial behavior and the proof is not available at the time of writing. After applying \ref{FOOOISDIR}, \ref{LEVEL1CATEXIST} and the above argument turning a level-1 good coordinate system into a Kuranishi space, one obtains a Kuranishi space from an FOOO Kuranishi structure. Since d-orbifolds are claimed to be equivalent to Kuranishi spaces, these two structures are identified with previous sc-Fredholm section, Kuranishi structure, Kuranishi atlas and virtual section via the identification of Kuranishi space with FOOO's Kuranishi structure and previous identifications. Joyce also discusses how to turn FOOO's Kuranishi structure, Kuranishi atlas, and polyfold Fredholm structure (via the forgetful functor in \ref{FORGETFULCONSTRUCTION}) into his Kuranishi space and he shows that the R-equivalence becomes an equivalence in $\mathbf{Kur}$ in the update of \cite{JK} and preprint \cite{2CAT}.

\section{Other virtual theories and pairwise identifications up to R-equivalences. Part B: Algebraic construction for invariants}

%{This section will be brief.}

\subsection{Joyce's Kuranishi homology and identification with rational singular homology}

Joyce has a non-perturbative homology theory in \cite{KH}, called Kuranishi homology. Namely, one replaces the role of each simplex in the singular chain by a Kuranishi structure\footnote{In \cite{KH}, certain good coordinate systems is used instead of Kuranishi structures as done here, but one can also use Kuranishi structures by \ref{ALLTHESAME} meanwhile one needs to show that the theory factors through to R-equivalence classes.}, and one places some gauge-fixing data to make sure each generator in the chain group has finite automorphisms (or the theory will be trivially zero, and also note that the standard simplex also has finite automorphisms). The boundary operator is just taking the boundary in the way similar to the boundary in the singular chain complex. The upshot is that, for example, after being equipped with gauge fixing data and passing to certain equivalence class, a moduli space $\overline{\mathcal{M}}_{g,m,A}$ equipped with a Kuranishi structure and natural maps $\text{st}\times \text{ev}$ to $\overline{\mathcal{M}}_{g,m}\times Q^{\times m}$ will be a Kuranishi cycle in Kuranishi homology of $\overline{\mathcal{M}}_{g,m}\times Q^{\times m}$ without perturbation. This becomes extremely convenient when moduli spaces have finer structures (but one needs to be able to cook up gauge-fixing data that respect finer structures).

To do calculation and relate to classical homology groups, one can explore exact sequences or excisions, or as Joyce did in \cite{KH} one can use equivalence relations built-in and local perturbations to show that Kuranishi homology of an orbifold is isomorphic to its rational singular homology.

Joyce said that he will replace the above version by virtual chains for Kuranishi spaces in M-(co)homology in a follow-up to \cite{MHOM}, and the proof for quasi-isomorphism between virtual chain complex and singular chain complex is expected to be more elegant than the proof in \cite{KH}. An overview can be found in \cite{2CAT}. On the surface, I think that there seems to be an analogy between this to (an orbifold version with corners of) de Rham chain/homology of Irie \cite{Irie} (see \ref{DERHAMCHAIN} below) following an idea of Fukaya, using a section instead of a (compactly supported) de Rham form on the smooth chains, with maybe more complicated relations to turn a map from a compact Kuranishi space with corners into a virtual chain.

\subsection{Polyfold Fredholm homology, identification with Kuranishi homology, rational singular homology and HWZ's de Rham theory}

The analogue of polyfold Fredholm homology respecting finer structure (e.g. appearing in SFT or Lagrangian Floer moduli spaces) will be constructed in \cite{Yang}. Using the polyfold--Kuranishi correspondence \ref{PKC}, this will be shown to be isomorphic to Kuranishi homology. Using globalization, ideas from e.g. normalization and perturbation, this will be shown to be isomorphic to rational singular homology. Via singular homology, its relation with HWZ's de Rham theory \cite{Polyfoldinteg} can be connected via triangulating a (normalization of) weighted branched suborbifold.

\subsection{Identification of singular homology with de Rham cohomology (after coefficient change) using FOOO's CF-pertubation and CLW's virtual integration}

Recall from \ref{CWPISCF} for the first two. The last two are identified using the level-1 structure for CF-perturbation as in \ref{CWPISCF} and conversion in \ref{VIRTUALSECTION}.\footnote{Note the fact that the collection of canonical sections in CF-perturbation constructed from level-1 structure is level-1 R-equivalent to level-1 structure associated to starting Kuranishi structure, although we do not need to use it to prove the statement here.} Note that $\{\omega_i\}_i$ produced in \ref{CWPISCF}, if the original index set is from a subset of cover index and after ungrouping the index $i$'s back to the original, is the virtual Euler form and the forms chosen for all the $\ker \hat\pi_{ij}$ are called transition Thom forms in CLW's language, see \cite{CLW} 3.11 and 3.14.

\subsection{CF-perturbation and identification with Irie's de Rham homology via a level-1 structure/CF-perturbation and triangulation}\label{DERHAMCHAIN}

Inspired by a notion of Fukaya in \cite{Fukayaloop} 6.4, Irie defines a de Rham homology theory \cite{Irie} (recently merged with another paper of his to become \cite{IrieNew}) to be able to define string topology operation on the chain level. If considering a Kuranishi structure equipped with a map into an orbifold or any differentiable space (in the sense of diffeology) as Irie pointed out, one uses level-1 structure and CF-pertrubation as in subsection \ref{CWPISCF}, normalizes (or inductively triangulates) and degenerates the solution of a CF-perturbed section along piecewise smooth hypersurfaces (or faces of simplices) and push them to infinity, then obtains a de Rham chain (cycle if Kuranishi structure has no boundary), as the data from different charts related by level-1 coordinate changes are deemed equivalent in de Rham chain theory on the level of formulation, see \cite{Irie} 2.3 (Z3) or \cite{IrieNew} 2.3. One can also do multisectional perturbation and triangulate, then degenerating along the boundaries of simplices to get a singular chain as a special case of a de Rham chain. Both de Rham chains agree with the integration along the fiber/pushout of CF-perturbation theory in \cite{FOOOCF} 7.7.

\subsection{Pardon's implicit atlas and virtual chain and identification with CF-perturbation via a level-1 structure}

Pardon constructed an implicit atlas and virtual chain package in \cite{Pardon}, and the nice feature is that it works when coordinate changes between local models of the moduli space under consideration are only required to be continuous. Locally, the picture is similar to the integration in the CF-perturbation and the virtual integration, to globalize, instead of using structure of moduli spaces to glue, he homotopy glues certain homotopy types of local models (constructed from the initial structure of the moduli space similar to a Kuranishi atlas) on the level of chain complexes via mapping cones. Pardon's cohomology theory constructed (or a class of the degree of the virtual dimension) satisfies homotopy cosheaf/Mayer-Vietoris axiom respect to a \v Cech cover.\footnote{On the formal level, it looks like an example of factorization homology theory or topological chiral homology, explained to me by Gregory Ginot.} 

After obtaining to a level-1 good coordinate system indexed by subsets of the index of a cover from a Kuranishi atlas, one can use the level-1 structure and de Rham chain complex to carry out Pardon's first proposal \cite{Pardon} (2.3.6) and we will explain below. For clarity, we will carry out his proposal for his example of two charts without group action. For a general cover with group actions, one needs an orbifold version of de Rham chain\footnote{Where the domain of de Rham chain has a global effective action, one puts one more relation as in \cite{KH} allowing lifting group action via a division by its order as coefficient, and one fixes an embedding of the quotient space into $\R^N$ so that the automorphism group of the chain is finite, a form of gauge fixing data.} and more combinatorial expression. Pardon's $X_\alpha, X_{\alpha\beta}$ and $U_\alpha$ will correspond to $U_{(\alpha)}, U_{(\alpha\beta)}$ and $s_{(\alpha)}^{-1}(0)$ in our notation and we will use the latter below in order to compare with other parts of the article. We use unprimed notations for the level-1 Kuranishi atlas for notational clarity, and we can choose $W_{(\alpha\beta)(\alpha)}=U_{(\alpha\beta)(\alpha)}$ in the level-1 structure etc. Write $\overset{\leftharpoonup}{\pi}_{(\alpha\beta)(\alpha)}:=(\phi_{(\alpha\beta)(\alpha)})^{-1}|_{\phi_{(\alpha\beta)(\alpha)}(U_{(\alpha\beta)(\alpha)})}\circ\pi_{(\alpha\beta)(\alpha)}: U_{(\alpha\beta)(\alpha)}=W_{(\alpha\beta)(\alpha)}\to U_{(\alpha)}$. With the compatible $\omega_{(\alpha)}$ chosen (as in \ref{CWPISCF}),

\begin{enumerate}
\item the (absolute) de Rham chain giving rise to a class in $$H_{\text{dim} E_{(\alpha)}}(U_{(\alpha)}, U_{(\alpha)}\backslash s_{(\alpha)}^{-1}(0))$$ that corresponds to the non-trivial class in $H_{\text{dim} E_{(\alpha)}}(E_{(\alpha)}, E_{(\alpha)}\backslash 0)$ in \cite{Pardon} (2.3.1) is $(Id_{(\alpha)}: U_{(\alpha)}\to U_{(\alpha)}, s_{(\alpha)}^\ast \omega_{(\alpha)})\in C_\ast^{dR}(U_{(\alpha)})$;
\item the de Rham chain giving rise to a class in $$H_{\ast}(U_{(\alpha\beta)}, U_{(\alpha\beta)}\backslash s_{(\alpha\beta)}^{-1}(0))$$ that corresponds to the non-trivial class in $H_{\text{dim} E_{(\alpha)}}(E_{(\alpha)}, E_{(\alpha)}\backslash 0)$ after \cite{Pardon} (2.3.1) and (2.3.6) and projecting to the first coordinate is $(\overset{\leftharpoonup}{\pi}_{(\alpha\beta)(\alpha)}: U_{(\alpha\beta)}=W_{(\alpha\beta)(\alpha)}\to U_{(\alpha)}, s_{(\alpha)}^\ast \omega_{(\alpha)})\in C_\ast^{dR}(U_{(\alpha)})$; and
\item the de Rham chain giving rise to a class in $$(\cdot\cap [s^{-1}_{(\beta)}(0)])(H_\ast(U_{(\alpha\beta)}, U_{(\alpha\beta)}\backslash s_{(\alpha\beta)}^{-1}(0)))\subset H_\ast(U_{(\alpha)}, U_{(\alpha)}\backslash s_{(\alpha)}^{-1}(0))$$ that corresponds to the non-trivial class in $H_{\text{dim} E_{(\alpha)}}(E_{(\alpha)}, E_{(\alpha)}\backslash 0)$ is $$(\overset{\leftharpoonup}{\pi}_{(\alpha\beta)(\alpha)}: U_{(\alpha\beta)}=W_{(\alpha\beta)}\to U_{(\alpha)}, s_{(\alpha)}^\ast \omega_{(\alpha)}\wedge s_{(\beta)}^\ast\omega_{(\beta)})\in C_\ast^{dR}(U_{(\alpha)}).$$
\end{enumerate}

In de Rham chains, $(Id_{(\alpha\beta)}: U_{(\alpha\beta)}\to U_{(\alpha\beta)}, s_{(\beta)}^\ast\omega_{(\beta)})=(\text{incl}: s_{(\beta)}^{-1}(0)\to U_{(\alpha\beta)},1)$ (via projection $\overset{\leftharpoonup}{\pi}_{(\alpha\beta)(\alpha)}$ and the formalism of de Rham chains, see \cite{Irie} 2.3 (Z3)), and intersection theory of de Rham chains\footnote{The observation that de Rham chains can be used as a chain-level commutative and associative chain model for intersection theory of manifolds seems to be new but should be implicit in Irie's work \cite{Irie}. It would be interesting to interpret results in Fulton's Intersection theory book from this viewpoint.} is the fiber product on the map part (which is always submersion by definition) and the wedge product on the de Rham form part. With the help of a (globally compatible) level-1 structure, this realization of chain level intersection theory, and the above correspondence we realize \cite{Pardon} (2.3.6) on the chain level (which is the relative version of below), namely,
\begin{align*}(\overset{\leftharpoonup}{\pi}_{(\alpha\beta)(\alpha)},s_{(\alpha)}^\ast\omega_{(\alpha)})\cap (Id_{(\alpha\beta)}, s_{(\beta)}^\ast\omega_{(\beta)})&=(\overset{\leftharpoonup}{\pi}_{(\alpha\beta)(\alpha)},s_{(\alpha)}^\ast\omega_{(\alpha)}\wedge s_{(\beta)}^\ast\omega_{(\beta)})\\
&=(Id_{(\alpha)},s_{(\alpha)}^\ast\omega_{(\alpha)}),
\end{align*}
where the second identity is again a part of the formalism of de Rham chains. The cases for the second projection and for an arbitrary finite number of charts with group actions in general are similar. It is clear that the above identification on the overlaps and globalization of Pardon's local constructions will have the same effect as integration in the CF-perturbation and virtual integration, and the above is also a geometric analogue of Pardon's homotopic version via \v Cech cochains.

One notes that after obtaining (which requires smoothness in general) a global level-1 structure (as topological submersions/projections) of a Kuranishi atlas to be used as a globalizing method, one might be able to do constructions (not using de Rham chains though) in continuous category to get an invariant; while Pardon achieved globalization via homotopic algebra in the continuous category, see \cite{Pardon} 2.3.12 and 4.2.6. Pardon also commented that his machinery should agree with other abstract perturbation methods.

\section{What are not included and possible future follow-up}\label{OPENORTODO}

There is also a geometric perturbation method applicable for any general closed symplectic manifolds via Donaldson divisors initiated by Cieliebak-Mohnke for genus-0 Gromov-Witten theory in \cite{CM}. Gerstenberger further developed the technique to treat general genus Gromov-Witten theory in \cite{G}. There is also another approach for general genus Gromov-Witten theory by Ionel-Parker in \cite{IP}, where the (in)dependence on Donaldson divisors is still being developed. Gerstenberger and the author will investiage the relation between polyfold perturbation and geometric perturbation via domain stabilizing Donaldson divisor in the near future, which involves a version of Morse-Bott target-$S^1$-equivariant SFT of degenerating along a contact hypersurface with $S^1$-action via polyfolds and expressing the polyfold Fredholm virtual class in terms of CMG-pseudocycles of reduced invariants weighted by enumeratively meaningful coefficients.

There are other interesting theories in various general or specialized settings, some in algebraic geometry and others in symplectic geometry, such as Lurie and Toen-Vezzosi's derived moduli stack, Behrend-Fantechi's intrinsic normal cone and perfect obstruction theory, Li-Tian's virtual moduli cycle, Lee's virtual structure sheaf, Ciocan-Fontanine--Kapranov's dg-scheme, Behrend's symmetric obstruction theory, Kiem-Li's cosection localization, Hutchings-Taubes's obstruction bundle gluing (a special case of gluing in Kuranishi structures but with its specific features), and Lu-Tian's virtual Euler class (and Wang's recent work on removing a technical assumption). Relationship and compatibility of each of the above with the virtual theories discussed in this paper can be investigated in certain common framework, but this is beyond the scope of this article at least for now.

The compatibility of various gluing methods might also be investigated and surveyed by the author. If they are compatible in a certain sense under the identifications in this article, then each virtual theory can run independently and yet produce the same or a relatable result.

\bibliographystyle{amsplain}
%\bibliography{FFunctor}

\end{document}